\documentclass[reqno]{amsart}
\usepackage{amssymb,mathrsfs}%,refcheck}
\usepackage[dvips]{graphicx}
\newtheorem{thm}{Theorem}[section]
\newtheorem{prop}[thm]{Proposition}
\newtheorem{lem}[thm]{Lemma}
\newtheorem{cor}[thm]{Corollary}
{\theoremstyle{definition}
\newtheorem{remark}[thm]{Remark}}

\title{Petersson scalar products and $L$-functions arising from modular forms}
\author{Shigeaki Tsuyumine}
\address{Mie University, Faculty of Education, Department of Mathematics, Tsu, 514-8507, Japan}
\email{tsuyu@edu.mie-u.ac.jp}

\allowdisplaybreaks
\begin{document}
\maketitle
\markboth{Shigeaki Tsuyumine}{L-functions arising from modular forms}
\footnote[0]{This work was partly supported by Grants-in-Aid for Scientific Research (C) from the Ministry of Education, Science, Sports and Culture of Japan, Grant Number 16K05056.}

The Rankin-Selberg method developed by Rankin \cite{Rankin} and by Selberg \cite{Selberg} gives us $L$-functions from cuspidal automorphic forms by taking the scalar products of them with  real analytic Eisenstein series of weight $0$. Some of analytic properties of Eisenstein series inherit to the $L$-functions such as functional equations or positions of possible poles. In Zagier \cite{Zagier}, he shows that this important method can be applied to the automorphic forms on $\mathrm{SL}_{2}(\mathbf{Z})$ which are not cuspidal.  The researches in this direction are made also by Gupta \cite{Gupta}, Chiera \cite{Chiera}.

In the present paper we consider mainly modular forms on $\Gamma_{0}(N):=\{\left({a\ b\atop c\ d}\right)\in\mathrm{SL}_{2}(\mathbf{Z})\mid c\equiv0\pmod{N}\}$. Let $z=x+\sqrt{-1}y\in\mathfrak{H}$ where $\mathfrak{H}$ denotes the complex upper half plane. Instead of Eisenstein series of weight $0$ we employ Eisenstein series of integral weight $k$
\begin{align*}
y^{s}E_{k,\chi}(z,s):=2^{-1}y^{s}\sum_{c,d}\chi(d)(cz+d)^{-k}|cz+d|^{-2s}\hspace{1em}(s\in\mathbf{C})
\end{align*}
with a Dirichlet character $\chi$ modulo $N$ where $c,d$ runs over the set of second rows of matrices in $\Gamma_{0}(N)$, or Eisenstein series of half integral weight (see Sect.~\ref{sect:ESHIW}). We show that the analytic properties of these Eisenstein series inherit to the following $L$-functions. Let $f(z)=\sum_{n=0}^{\infty}a_{n}\mathbf{e}(nz),g(z)=\sum_{n=0}^{\infty}a_{n}\mathbf{e}(nz)$ with $\mathbf{e}(z)=e^{2\pi\sqrt{-1}z}$ be modular forms for $\Gamma_{0}(N)$ of integral or half integral weight
so that their weights are not necessarily equal to each other as well as characters. Define
\begin{align}
L(s;f,g):=\sum_{n=1}^{\infty}\frac{a_{n}\overline{b}_{n}}{n^{s}}, \label{eqn:lseries}
\end{align}
which converges for $s$ with sufficiently large $\Re s$. Let $f,g$ be of weight $l,l'$ respectively, and suppose that $l\ge l'$. We show that $L(s;f,g)$ extends meromorphically to the whole complex plane, and has the functional equation under $s\longmapsto-(l-l')+1-s$ (Corollary \ref{cor:lseries}, Corollary \ref{cor:lseries2}). If $l=l'$ and $f,g$ have a same character, then 
\begin{align*}
\langle f,g\rangle_{\Gamma_{0}(N)}=c\,\mathrm{Res}_{s=l}L(s;f,g)
\end{align*}
for a suitable constant $c$ where $\langle f,g\rangle_{\Gamma_{0}(N)}$ denotes the Petersson scalar product (see Sect.~\ref{sect:PSP}), and $\mathrm{Res}_{s=l}$ implies the residue at $s=l$. This was proved in Petersson \cite{Petersson}, Satz 6 for cusps forms $f,g$ of integral weight. If either wights of $f,g$ or charcters are distinct, then $L(l-1;f,g)$ is written in terms of the scalar product involving $f,g$ and a suitable Eisenstein series (Corollary \ref{cor:lseries}, Corollary \ref{cor:lseries2}).

A Dirichlet character $\chi$ is called {\it even} or {\it odd} according as $\chi(-1)$ is $1$ or $-1$. Let 
\begin{align*}
\theta(z):=\sum_{n=-\infty}^{\infty}\mathbf{e}(n^{2}z)=1+2\sum_{n=1}^{\infty}\mathbf{e}(n^{2}z)
\end{align*}
be the theta series, and let $\theta_{\chi}(z):=\sum_{n=-\infty}^{\infty}\chi(n)\mathbf{e}(n^{2}z)$ be its twist by an even Dirichlet character $\chi$.  For an odd $\chi$, put $\Theta_{\chi}(z):=\sum_{n=-\infty}^{\infty}\chi(n)n\mathbf{e}(n^{2}z)$, which is a cusp form of weight $3/2$. Then the Riemann zeta function and the Dirichlet $L$-functions appear as $L(s;\theta,\theta)=4\zeta(2s)$,  $L(s;\theta_{\chi},\theta)=4L(2s,\chi)$ for $\chi$ even, $L(s;\Theta_{\chi},\theta)=4L(2s-1,\chi)$ for $\chi$ odd, and we may expect that many other interesting $L$-functions appear by this method. We have the following application. Let $Q,Q'$ be positive definite integral quadratic with $2l,2l'$ variables ($l,l'\in\frac{1}{2}\mathbf{N},\,l\ge l'$). Let $r_{Q}(n),r_{Q'}(n)$ be the numbers of integral representations of $n$ by $Q,Q'$ respectively. Then the Dirichlet series $\sum_{n=1}^{\infty}r_{Q}(n)r_{Q'}(n)n^{-s}$ extends meromorphically to the whole $s$ plane. Indeed if $f(z),g(z)$ are the theta series associated with $Q,Q'$ respectively, then $L(s;f,g)=\sum_{n=1}^{\infty}r_{Q}(n)r_{Q'}(n)n^{-s}$. Then it satisfies the functional equation under $l-1+s\longmapsto l'-s$ such as (\ref{eqn:fn-eq3}) or (\ref{eqn:fn-eq4}).  Further the asymptotic value of $X^{-l-l'+1}\sum_{0<n\le X}r_{Q}(n)r_{Q'}(n)$ as $X\longrightarrow\infty$ is obtained in terms of the $0$-th Fourier coefficients of $f,g$ at cusps and the constant  terms of Eisenstein series appropriately taken (Section \ref{sect:ALF}).

Let us fix our notation. Let $N\in\mathbf{N}$. We denote by $(\mathbf{Z}/N)^{\ast}$, the group of Dirichlet characters modulo $N$. The identity element of $(\mathbf{Z}/N)^{\ast}$ is denoted by $\mathbf{1}_{N}$, where $\mathbf{1}_{N}(n)$ is $1$ or $0$ according as $n$ is coprime to $N$ or not. In particular $\mathbf{1}_{1}$ always takes the value $1$, which we denote simply by $\mathbf{1}$.  We denote by $\mathfrak{f}_{\chi}$ the conductor of $\chi$, and denote by $\widetilde{\chi}$ the primitive character associated with $\chi$.  For a prime $p$, $v_{p}(n)$ denotes the $p$-adic valuation of $n\in\mathbf{Z},\ne0$. We put
\begin{align*}
\mathfrak{e}_{\chi}:=\mathfrak{f}_{\chi}\prod_{p\nmid\mathfrak{f}_{\chi},\chi(p)=0}p,\hspace{2em}\mathfrak{e}_{\chi}':=\mathfrak{f}_{\chi}\prod_{p\nmid\mathfrak{f}_{\chi},\chi(p)=0}p^{2}.
\end{align*}
For a prime $p|N$, $\{\chi\}_{p}\in(\mathbf{Z}/p^{v_{p}(N)})^{\ast}$ denotes the $p$-part of $\chi$, and there holds an equality $\chi=\prod_{p|N}\{\chi\}_{p}$.

For $D$ a discriminant of a quadratic number field, $\chi_{D}$ is defined to be the Kronecker-Jacobi-Legendre symbol, and for  $D=1$,  $\chi_{1}$ is defined to be $\mathbf{1}$.  We extend this notation as follows. Let $a\in\mathbf{Z},\ne0$. Then we define $\chi_{a}=\chi_{D(a)}\mathbf{1}_{a}$ for the discriminant $D(a)$ of the quadratic number field $\mathbf{Q}(\sqrt{a})$. When $a$ is odd, $\chi_{a^{\vee}}$ denotes $\chi_{a}$ if $a\equiv1\pmod{4}$, $\chi_{-a}$ if $a\equiv-1\pmod{4}$. Let $\mu(n),\varphi(n)$ for $n\in\mathbf{N}$ denote  as usual, the M\"obius function, the Euler function respectively.  %For $n\in\mathbf{Z},\ne0$ and $s\in\mathbf{C}$ and for Dirichlet characters $\chi,\chi'$, we define $\sigma_{s,\chi}^{\chi'}(n)=\sum_{0<d|n}\chi(d)\chi'(n/d)d^{s}$. We drop the notation of $\chi$ or $\chi'$ if they are $\mathbf{1}$.

Let $\mathfrak{F}$ be the fundamental domain of the group $\mathrm{SL}_{2}(\mathbf{Z})$;\ $\{z=x{+}\sqrt{{-}1}y\in\mathbf{C}\mid |x|\le1/2,|z|\ge1\}$. Let $\Gamma$ be a congruence subgroup. The fundamental domain of $\Gamma$ is obtained by $\mathfrak{F}(\Gamma):=\bigcup_{A} A\mathfrak{F}$ where $A$ runs over the set of left representatives of $\mathrm{SL}_{2}(\mathbf{Z})$ modulo $\Gamma$. We denote by $\mathfrak{F}(N)$, the fundamental domain $\mathfrak{F}(\Gamma_{0}(N))$ of $\Gamma_{0}(N)$. We denote by $\mathcal{M}_{s_{1},s_{2}}(\Gamma)$ for $s_{1},s_{2}\in\mathbf{C}$, the space of real analytic functions on $\mathfrak{H}$ which satisfy
\begin{align*}
f(Az)=(cz+d)^{s_{1}}(\overline{cz+d})^{s_{2}}f(z)\quad(A=\left(a\ \,b\atop c\ \,d\right)\in\Gamma)
\end{align*}with $-\pi<\arg(cz+d)\le\pi$ and $\arg(\overline{cz+d})=-\arg(cz+d)$. The imaginary part $y$ of $z$ is in $\mathcal{M}_{-1,-1}(\Gamma)$. We define 
\begin{align*}
f|_{A}(z):=(cz+d)^{-s_{1}}(\overline{cz+d})^{-s_{2}}f(Az)\quad(A=\left(a\ \,b\atop c\ \,d\right))
\end{align*}with $a,b,c,d\in\mathbf{Q},ad-bc>0$. For two congruence subgroups $\Gamma_{1},\Gamma_{2}$ with $\Gamma_{1}\supset\Gamma_{2}$, the trace of $f\in\mathcal{M}_{s_{1},s_{2}}(\Gamma_{2})$ is defined by $\mathrm{tr}_{\Gamma_{1}/\Gamma_{2}}(f):=\sum_{A}f|_{A}$ where $A$ runs over the set of left representatives of $\Gamma_{1}$ modulo $\Gamma_{2}$. The trace of $f$ is in $\mathcal{M}_{s_{1},s_{2}}(\Gamma_{1})$. 
For $k\in\mathbf{Z},\,s\in\mathbf{C}$ and for a Dirichlet character $\chi$ modulo $N$ with the same parity as $k$, $\mathcal{M}_{k+s,s}(N,\chi)$ denotes the space of elements $f\in\mathcal{M}_{k+s,s}(\Gamma_{1}(N))$  satisfying
\begin{align*}
f(Az)=\chi(d)(cz+d)^{k}|cz+d|^{2s}f(z)\quad(A\in\Gamma_{0}(N)).
\end{align*}
Let $j(A,z):=\theta(Az)/\theta(z)\ (A\in\mathrm{SL}_{2}(\mathbf{Z}))$. For $A=\left(a\ \,b\atop c\ \,d\right)\in\Gamma_{0}(4)$, $j(A,z)=1$ if $c=0$, and $j(A,z)=\chi_{c}(d)\iota_{d}^{-1}(cz+d)^{1/2}$ if $c>0$ where $\iota_{d}$ is $1$ or $\sqrt{{-}1}$ according as $d\equiv1$ modulo $4$ or $d\equiv3$ and where $-\pi/2<\arg(cz+d)^{1/2}\le\pi/2$.  On the group $\Gamma_{0}(4)$, $j(A,z)$ gives the automorphy factor. 

Let $4|N$. For a Dirichlet character $\chi$ modulo $N$ with the same parity as $k$, $\mathcal{M}_{k+1/2+s,s}(N,\chi)$ denotes the space of elements  $f\in\mathcal{M}_{k+1/2+s,s}(\Gamma_{1}(N))$  satisfying
\begin{align*}
f(Az)=\chi(d)j(A,z)(cz+d)^{k}|cz+d|^{2s}f(z)\quad(A=\left(a\ \,b\atop c\ \,d\right)\in\Gamma_{0}(N)).
\end{align*}We denote $\mathbf{M}_{l}(N,\chi)$ for $l\in(1/2)\mathbf{N}$, the space of holomorphic modular forms in $\mathcal{M}_{l,0}(N,\chi)$. If $\chi=\mathbf{1}_{N}$, we drop $\chi$ from the notations $\mathbf{M}_{l}(N,\chi)$ or $\mathcal{M}_{l,0}(N,\chi)$.

\section{Petersson scalar product}\label{sect:PSP}
In the section, we define the Petersson scalar product of modular forms which are not necessarily holomorphic. Further we obtain a formula between the scalar product and a special value of the $L$-function (\ref{eqn:lseries}) of holomorphic modular forms of same weight and with same character.  

Let $\Gamma$ be a congruence subgroup of $\mathrm{SL}_{2}(\mathbf{Z})$. Let $r$ be a cusp of $\Gamma$, and let $A_{r}$ be a matrix so that
\begin{align}A_{r}(\sqrt{-1}\infty)=r\hspace{1em}(A_{r}\in\mathrm{SL}_{2}(\mathbf{Z})).\label{defAr}
\end{align}
Let $w^{(r)}=w_{\Gamma}^{(r)}$ be the width of a cusp $r$ of $\Gamma$, namely, $w^{(r)}$ is the least natural number so that $A_{r}\left(\pm1\ \,w^{(r)}\atop 0\ \ \,\pm1\right)A_{r}^{-1}\in \Gamma$. Let $f,g$ be real analytic modular forms for $\Gamma$ of wight $l\in\tfrac{1}{2}\mathbf{Z},\ge0$ with same character so that $y^{l}f\overline{g}$ has the Fourier expansions at each cusp $r$ in the form
\begin{align}
(y^{l}f\overline{g})|_{A_{r}}(z)=P_{y^{l}f\overline{g}}^{(r)}(y)+\sum_{n=-\infty}^{\infty}u_{n}^{(r)}(y)\mathbf{e}(nx/w^{(r)}),\hspace{.7em}P_{y^{l}f\overline{g}}^{(r)}(y)=\sum_{j}c_{\nu_{j}^{(r)}}^{(r)}y^{\nu_{j}^{(r)}}\label{eqn:dfP}
\end{align}
where $u_{n}^{(r)}(y)$ is a rapidly decreasing function as $y\longrightarrow\infty$ and where the constant term $P_{y^{l}f\overline{g}}^{(r)}(y)$ with respect to $x$, is a finite linear combination of powers of $y$ with $\nu_{j}^{(r)},c_{\nu_{j}^{(r)}}^{(r)}\in\mathbf{C}$. For $T>1$, we put
\begin{align}
Q_{y^{l}f\overline{g}}^{(r)}(T):&=w^{(r)}\{c_{1}^{(r)}\log T+\int_{0}^{T}\sum_{\Re \nu_{j}^{(r)}\ge1,\nu_{j}^{(r)}\ne1}c_{\nu_{j}^{(r)}}^{(r)}y^{\nu_{j}^{(r)}-2}dy\}\nonumber\\
&=w^{(r)}\{c_{1}^{(r)}\log T+\sum_{\Re \nu_{j}^{(r)}\ge1,\nu_{j}^{(r)}\ne1}c_{\nu_{j}^{(r)}}^{(r)}\tfrac{T^{\nu_{j}^{(r)}-1}}{\nu_{j}^{(r)}-1}\}. \label{eqn:defQr}
\end{align}
For  $T>1$, let $\mathfrak{F}_{T}(\Gamma)$ denote the domain obtained from $\mathfrak{F}(\Gamma)$ by cutting off neighborhoods of all cusps $r$ along the lines $\Im(A_{r}^{-1}z)=T$ (see Figure \ref{fig:fdomain0}) where $\Im$ means the imaginary part.
\begin{figure}[htbp]
\centering
\includegraphics{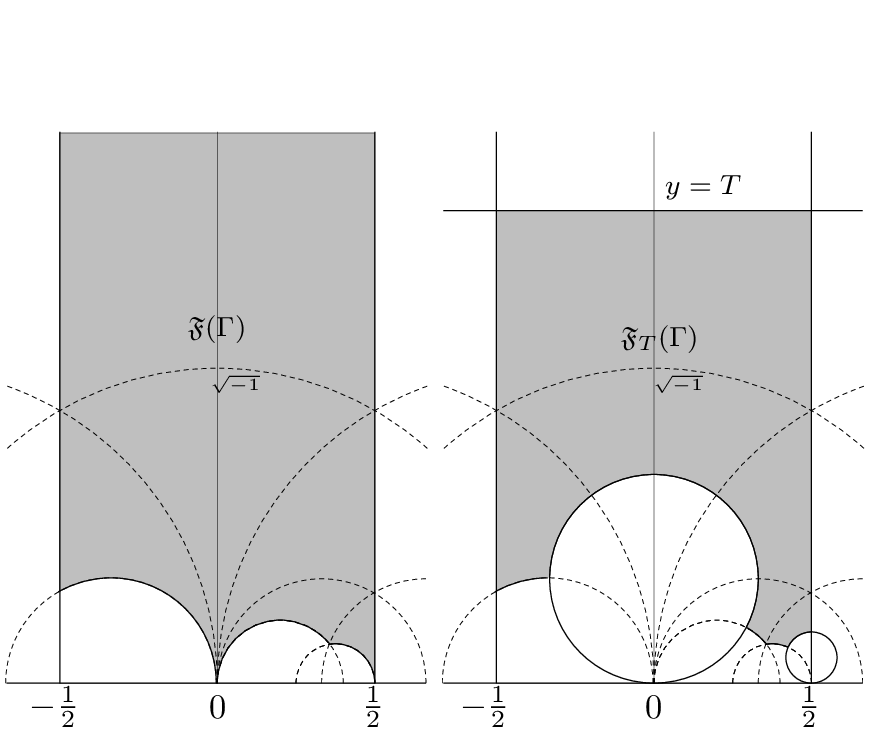}
\caption{ }
\label{fig:fdomain0}
\end{figure}
Then we define the Petersson scalar product of $f$ and $g$ by
\begin{align}
\langle f,g\rangle_{\Gamma}:=\lim_{T\to\infty}\left(\rule{0cm}{1.1em}\right.\int_{\mathfrak{F}_{T}(\Gamma)}y^{l}f(z)\overline{g}(z)\tfrac{dxdy}{y^{2}}{-}\sum_{r}Q_{y^{l}f\overline{g}}^{(r)}(T)\left.\rule{0cm}{1.1em}\right),  \label{eqn:psp}
\end{align}
$r$ running over the set of the representatives of cusps of $\Gamma$, or equivalently,
\begin{align*}
\langle f,g\rangle_{\Gamma}=&\int_{\mathfrak{F}_{T}(\Gamma)}y^{l}f(z)\overline{g}(z)\tfrac{dxdy}{y^{2}}\\
&+\sum_{r}\left(\int_{T}^{\infty}\hspace{-.5em}\int_{0}^{w^{(r)}}\hspace{-.5em}\{(y^{l}f\overline{g})|_{A_{r}}(z)-P_{y^{l}f\overline{g}}^{(r)}(y)\}dx\tfrac{dy}{y^{2}}-Q_{y^{l}f\overline{g}}^{(r)}(T)\right),
\end{align*}
where the right hand side is independent of $T>1$.

Suppose that $f,g$ are holomorphic with $l\in\tfrac{1}{2}\mathbf{Z},>0$. Let $a_{0}^{(r)},b_{0}^{(r)}$ be the $0$-th Fourier coefficients at a cusp $r$, of $f,g$ respectively. Then $P_{y^{l}f\overline{g}}(y)$ is equal to $a_{0}^{(r)}\overline{b}_{0}^{(r)}y^{l}$, and an equality (\ref{eqn:psp}) turns out to be $\langle f,g\rangle_{\Gamma}=\lim_{T\to\infty}\left(\rule{0cm}{.9em}\right.\int_{\mathfrak{F}_{T}(\Gamma)}y^{l}f(z)\overline{g}(z)\tfrac{dxdy}{y^{2}}-\sum_{r}a_{0}^{(r)}\overline{b}_{0}^{(r)}w^{(r)}\tfrac{T^{l-1}}{l-1}\left)\rule{0cm}{.9em}\right.$ $(l\ne1)$, $\langle f,g\rangle_{\Gamma}=\lim_{T\to\infty}$ $\left(\rule{0cm}{.9em}\right.\int_{\mathfrak{F}_{T}(\Gamma)}yf(z)\overline{g}(z)\tfrac{dxdy}{y^{2}}-\sum_{r}a_{0}^{(r)}\overline{b}_{0}^{(r)}$ $\times w^{(r)}\log T\left)\rule{0cm}{.9em}\right.\ (l=1)$. In Zagier \cite{Zagier} the definition is found in the case $\Gamma=\mathrm{SL}_{2}(\mathbf{Z})$.  If at least one of $f,g$ is a cusp form, then the definition coincide with the common definition of the scalar product (Petersson\cite{Petersson}).  For $l=1/2$, $\tfrac{T^{l-1}}{l-1}$ tends to $0$ as $T\longrightarrow\infty$, and so the equality $\langle f,g\rangle_{\Gamma}=\int_{\mathfrak{F}_{T}(\Gamma)}y^{1/2}f(z)\overline{g}(z)\tfrac{dxdy}{y^{2}}$ holds. The convergence of the integral has been pointed out in Serre and Stark \cite{Serre-Stark} Appendix by Deligne. 

Let $f(z)=\sum_{n=0}^{\infty}a_{n}\mathbf{e}(nz/w),\,g(z)=\sum_{n=0}^{\infty}b_{n}\mathbf{e}(nz/w)$ be the Fourier expansions with $w=w^{(\sqrt{-1}\infty)}$ the width of the cusp $\sqrt{-1}\infty$. We define the Dirichlet series associated with them by $L(s;f,g):=\sum_{n=1}^{\infty}a_{n}\overline{b}_{n}n^{-s}$, namely we make the same definition as (\ref{eqn:lseries}) ignoring $w$.

Replacing $\Gamma$ by $\pm\Gamma$ if necessary, we may assume that $\pm1_{2}\in\Gamma$ where $1_{2}$ denotes the identity matrix.  We define the Eisenstein series $E_{\Gamma,r}(z,s):=y^{s}\sum_{(c,d)}|cz{+}d|^{-2s}$ where $(c,d)$ runs over the set of the second rows of matrices in $A_{r}^{-1}\Gamma$ with $c>0$, or with $c=0,d>0$. It converges absolutely and uniformly on any compact subset of $\mathfrak{H}$ for $\Re s>1$, and gives a function in $\mathcal{M}_{0,0}(\Gamma)$. As functions of $s$, it extends meromorphically to the whole plane, and satisfies a functional equations under $s\longmapsto1-s$ (cf. Kubota \cite{Kubota}). We denote $E_{\Gamma,r}(z,s)$ by $E_{\Gamma}(z,s)$ when $r=\sqrt{-1}\infty$. The constant term of the Fourier expansion of $E_{\Gamma,r}(z,s)$  with respect to $x$, at the cusp $r$ is in the form $y^{s}+\xi(s;r)y^{1-s}$, and at a cusp $r'$ not equivalent to $r$, in the form $\xi^{(r')}(s;r)y^{1-s}$. Here both $\xi(s;r)$ and $\xi^{(r')}(s;r)$ are involving the Riemann zeta function or partial zeta functions, and they are meromorphic functions on the $s$-plane and holomorphic on the domain $\Re s>1$.  They have poles of order $1$ at $s=1$, and $\mathrm{Res}_{s=1}E_{\Gamma,r}(z,s)=\mathrm{Res}_{s=1}\xi(s;r)=\mathrm{Res}_{s=1}\xi^{(r')}(s;r)$.

Let $l\in\frac{1}{2}\mathbf{Z},\ge 3/2$. An alternative definition of the scalar product is
\begin{align*}
\langle f,g\rangle_{\Gamma}:=\int_{\mathfrak{F}(\Gamma)}\{y^{l}f(z)\overline{g}(z)-\sum_{r}a_{0}^{(r)}\overline{b}_{0}^{(r)}E_{\Gamma,r}(z,l)\}\frac{dxdy}{y^{2}}
\end{align*}
where the integral is well-defined since the integrand is decreasing at all the cusps, indeed it is $O(y^{-l-1})$ as $y\longrightarrow\infty$ at each cusp (cf Zagier \cite{Zagier}, Sect. 5).

We denote the Laplacian by
\begin{align*}
\Delta:=y^{2}\left(\tfrac{\partial^{2}}{\partial x^{2}}+\tfrac{\partial^{2}}{\partial y^{2}}\right)=4y^{2}\tfrac{\partial}{\partial z}\tfrac{\partial}{\partial \overline{z}}
\end{align*}
where $\frac{\partial}{\partial z}=2^{-1}(\frac{\partial}{\partial x}-\sqrt{-1}\frac{\partial}{\partial y}),\,\frac{\partial}{\partial \overline{z}}=2^{-1}(\frac{\partial}{\partial x}+\sqrt{-1}\frac{\partial}{\partial y})$. It is an $\mathrm{SL}_{2}(\mathbf{R})$ invariant differential operator on $\mathfrak{H}$. As is well known, the Eisenstein series $E_{\Gamma,r}(z,s)$ is an eigenfunction of $\Delta$, in fact, we have $\Delta(E_{\Gamma,r}(z,s))=s(s-1)E_{\Gamma,r}(z,s)$.

\begin{thm}\label{thm:psp}
Let $\Gamma$ be a congruence subgroup of $\mathrm{SL}_{2}(\mathbf{Z})$ with the width $w_{\Gamma}$ of the cusp $\sqrt{-1}\infty$. Let $f,g$ be holomorphic modulars forms for $\Gamma$ of weight $l\in\frac{1}{2}\mathbf{Z},\ge 3/2$ or $l=1/2$ with the same character. Then $L(s;f,g)$ extends meromorphically to the whole $s$-plane, and satisfies a functional equation under $l-1+s\longmapsto l-s$ which comes from the functional equation of $E_{\Gamma}(z,s)$ under $s\longmapsto 1-s$, and
\begin{align}
\langle f,g\rangle_{\Gamma}=(4\pi)^{-l}w_{\Gamma}^{l+1}\Gamma(l)C^{-1}\mathrm{Res}_{s=l}L(s;f,g)\label{eqn:psp-formula}
\end{align}
where $C=\mathrm{Res}_{s=1}E_{\Gamma}(z,s)=\mathrm{Res}_{s=1}\xi(s)$, $\xi(s)$ being so that the constant term with respect to $x$, of $E_{\Gamma}(z,s)$ is $y^{s}+\xi(s)y^{1-s}$.
\end{thm}
\begin{remark} The method used in the following proof  is  found in Chiera \cite{Chiera} in which Theorem\ref{thm:psp} in the case $l=1/2$ is proved. He also uses the method to compute the scalar products of Eisenstein series of integral weight.
\end{remark}
\begin{proof} Let $l\ge3/2$. Let $F(z)$ be an automorphic form so that $F$ and $\Delta(F)$ are both integrable on $\mathfrak{F}(\Gamma)$. We compute the integration of $\Delta(F)$ over the fundamental domain as Kubota \cite{Kubota} Section 2.3. Then
\begin{align*}
&\int_{\mathfrak{F}(\Gamma)}\Delta(F)\tfrac{dxdy}{y^{2}}=\int_{\mathfrak{F}(\Gamma)}(\tfrac{\partial^{2}}{\partial x^{2}}+\tfrac{\partial^{2}}{\partial y^{2}})F(z)dxdy=\int_{\partial(\mathfrak{F}(\Gamma))}y\tfrac{\partial}{\partial n}F(z)\tfrac{dl}{y}
\end{align*}
by Green's theorem where $\partial(\mathfrak{F}(\Gamma))$ is the boundary of $\mathfrak{F}(\Gamma)$, $\partial/\partial n$ is the outer normal derivative, and $dl$ is the euclidean arc length. We note that $y\partial/\partial n$ and $dl/y$ are the invariant under $\mathrm{SL}_{2}(\mathbf{R})$. But on the lines of the boundary, arcs or pieces of arcs which are paired under $\Gamma$, the integral cancels, and $\int_{\mathfrak{F}(\Gamma)}\Delta(F)\tfrac{dxdy}{y^{2}}$ vanishes. We apply this to $F(z):=y^{l}f(z)\overline{g}(z)-\sum_{r}a_{0}^{(r)}\overline{b}_{0}^{(r)}E_{\Gamma,r}(z,l)$, and employ the method developed by Chiera \cite{Chiera}. 

We have equality $\Delta(y^{l}f\overline{g})=Q(f,g)+l(l-1)y^{l}f\overline{g}$ with
\begin{align}
Q(f,g):=4y^{l+2}\tfrac{\partial f}{\partial z}\overline{\tfrac{\partial g}{\partial z}}+2\sqrt{-1}ly^{l+1}(\tfrac{\partial f}{\partial z}\overline{g}-f\overline{\tfrac{\partial g}{\partial z}})\label{eqn:defQ}
\end{align}
(\cite{Chiera} Proposition 2.1). Then $\Delta(F)=Q(f,g)+l(l-1)F$, and hence
\begin{align}
\langle f,g\rangle_{\Gamma}=-\frac{1}{l(l-1)}\int_{\mathfrak{F}(\Gamma)}Q(f,g)\tfrac{dxdy}{y^{2}}\label{eqn:psp-Q}
\end{align}
by the above result. Since $Q(f,g)$ is an automorphic function rapidly decreasing at each cusp, to evaluate the integral we follows the argument due to Petersson \cite{Petersson} in which the Rankin-Selberg method (Rankin \cite{Rankin}, Selberg \cite{Selberg}) is made use of.

By a standard unfolding trick, we have for $\Re s>1$,
\begin{align*}
&\int_{\mathfrak{F}(\Gamma)}Q(f,g)E_{\Gamma}(z,s)\tfrac{dxdy}{y^{2}}=\int_{0}^{\infty}y^{s-2}\int_{0}^{w_{\Gamma}}Q(f,g)dxdy\\
=&w_{\Gamma}\int_{0}^{\infty}\left[\sum_{n=1}^{\infty}a_{n}\overline{b}_{n}e^{-4\pi ny/w_{\Gamma}}\{(4\pi/w_{\Gamma})^{2}n^{2}y^{l+s}-(8\pi l/w_{\Gamma})ny^{l-1+s}\}\right]dy\\
=&-(4\pi)^{-l+1-s}w_{\Gamma}^{l+s}(l-s)\Gamma(l+s)L(l-1+s;f,g).
\end{align*}
The extreme left hand side has a  meromorphic continuation to whole the $s$-plane since $E_{\Gamma}(z,s)$ has a meromorphic continuation and $Q(f,g)$ is rapidly decreasing at cusps. Hence $L(l-1+s;f,g)$ has also, and the functional equation of $E_{\Gamma}(z,s)$ gives that of $L(l-1+s;f,g)$. We note that this part holds true also for $l=1$. Evaluating the residues at $s=1$, we have $-Cl(l-1)\langle f,g\rangle_{\Gamma}=-(4\pi)^{-l}w_{\Gamma}^{l+1}(l-1)\Gamma(l+1)\mathrm{Res}_{s=1}L(l-1+s;f,g)$, which gives (\ref{eqn:psp-formula}) since $l\ne1$.
\end{proof}

When $l=1$, the equality (\ref{eqn:psp-Q}) does not make sense, and  (\ref{eqn:psp-formula}) does not hold in general (see Remark \ref{rem:lseries} later).

\section{Eisenstein series of integral weight}
To the end we concentrate ourselves to modular forms on $\Gamma_{0}(N)$. In this section we study analytic property of real analytic Eisenstein series for $\Gamma_{0}(N)$ of weight $(k+s,s)$ with $k\in\mathbf{Z},\ge0$, where $y^{s}$ times they are of integral weight $k$. Some generalization of Theorem \ref{thm:psp} is given by using Eisenstein series of weight $0$.

Let $l\in\frac{1}{2}\mathbf{Z},\ge0$. For a real number $m$ and for $\mathfrak{s}\in\mathbf{C}$ with $l+2\Re s>1$ we put
\begin{align*}
w_{-m}(y,l,s):=\mathbf{e}(mx)\int_{-\infty}^{\infty}\frac{\mathbf{e}(mt)}{(z+t)^{l}|z+t|^{2s}}\,dt
\end{align*}
with $z=x+\sqrt{{-}1}y\in\mathfrak{H}$, which satisfies $w_{m}(ny,l,s)=n^{-l+1-2s}w_{nm}(y,l,s)\ (n\in\mathbf{N})$. Then $w_{-m}(y,l,s)$ is equal to
\begin{alignat*}{2}
&\mathbf{e}(-\tfrac{l}{4})\cdot2\pi\cdot(2y)^{-l+1-2s}\Gamma(l-1+2s)\Gamma(s)^{-1}\Gamma(l+s)^{-1}&&(m=0),\\
&\mathbf{e}(-\tfrac{l}{4})m^{-1}(m\pi y^{-1})^{l/2+s}\Gamma(s)^{-1}W_{-\frac{l}{2},-\frac{l}{2}+\frac{1}{2}-s}(4\pi my)&&(m>0),\\
&\mathbf{e}(-\tfrac{l}{4})|m|^{-1}(|m|\pi y^{-1})^{l/2+s}\Gamma(l+s)^{-1}W_{\frac{l}{2},-\frac{l}{2}+\frac{1}{2}-s}(4\pi|m|y)&\quad&(m<0)
\end{alignat*}
which meromorphically extend to the whole complex $s$-plane, where $W_{\pm\frac{l}{2},-\frac{l}{2}+\frac{1}{2}-s}(y)$ denote the Whittaker functions of $y$.  We have $w_{-m}(y,l,0)=0$ for $m\ge0$ except for $(l,m)=(1,0)$. If $m<0$, then $w_{-m}(y,0,0)=0$, and $w_{-m}(y,l,0)=(2\pi)^{l}\mathbf{e}(-\frac{l}{4})\Gamma(l)^{-1}|m|^{l-1}e^{2\pi my}$ for $l>0$. By the Poisson summation formula, we have the Fourier expansion
\begin{align}
\sum_{m=-\infty}^{\infty}(z+m)^{-l}|z+m|^{-2s}=\sum_{n=-\infty}^{\infty}w_{n}(y,l,s)\mathbf{e}(nx).\label{eqn:fex}
\end{align}
The right hand side extends meromorphically to the whole complex $s$-plane.

Let $k\in\mathbf{N}$, $N\ge3$ and $0\le c_{0},d_{0}<N$. Put $g_{k}(z;c_{0},d_{0};N;s)=\sum_{c\equiv c_{0}(N)\atop d\equiv d_{0}(N)}(cz+d)^{-k}|cz+d|^{-2s}$. This series converges if $k+2\Re s>2$ and has Fourier expansion
\begin{align}
&\delta_{c_{0},0}\sum_{d\equiv d_{0}(N)}d^{-k}|d|^{-2s}+N^{-1}\sum_{c\equiv c_{0}(N),\ne0}c^{-k}|c|^{1-2s}w_{0}(y,k,s)\nonumber\\
&+N^{-1}\sum_{n\in\mathbf{Z},\ne0}\left(\rule{0cm}{1.2em}\right.\sum_{c\equiv c_{0}(N),\ne0\atop c|n}c^{-k}|c|^{1-2s}\mathbf{e}(\tfrac{nd_{0}}{Nc})\left.\rule{0cm}{1.2em}\right)w_{n/N}(y,k,s)\mathbf{e}(\tfrac{nx}{N}),\label{feg}
\end{align}
$\delta$ denoting the Kronecker delta. The infinite series appearing in the constant term with respect to $x$ are essentially the Hurwitz zeta functions, which extend meromorphically to the whole $s$ plane. Hence $g_{k}(z;c_{0},d_{0};N;s)$ is a meromorphic function on the whole complex plane as a function of $s$.

Let $\chi$ be a Dirichlet character, and let $\widetilde{\chi}$ be the primitive character associated with $\chi$. Let $\mathcal{I}_{\mathbf{Z}}$ denote the characteristic function of $\mathbf{Z}$ on $\mathbf{Q}$. The Gauss sum $\tau(\widetilde{\chi})$ is defined to be $\sum_{i:\mathbf{Z}/\mathfrak{f}_{\chi}}\widetilde{\chi}(i)\mathbf{e}(i/\mathfrak{f}_{\chi})$ where $i:\mathbf{Z}/\mathfrak{f}_{\chi}$ implies that $i$ runs the set of representatives of $\mathbf{Z}$ modulo $\mathfrak{f}_{\chi}$. Then for $m\in\mathbf{Z}$, we have the formula for the Gauss sum of a Dirichlet character not necessarily primitive,
\begin{align*}
\sum_{i:\mathbf{Z}/N}\chi(i)\mathbf{e}(im/N)=\tau(\widetilde{\chi})\sum_{0<R|\mathfrak{e}_{\chi}\mathfrak{f}_{\chi}^{-1}}\tfrac{\mu(R)\varphi(N)}{\varphi(\mathfrak{f}_{\chi}R)}\widetilde{\chi}(R)(\overline{\widetilde{\chi}}\mathbf{1}_{R}\mathcal{I}_{\mathbf{Z}})(m\mathfrak{f}_{\chi}RN^{-1}),
\end{align*}
with $\mathfrak{e}_{\chi}$ as in the introduction where in the summation of the right hand side, at most one term survives for each $m$ in $\mathbf{Z}$.

 As a set of representatives of cusps of $\Gamma_{0}(N)$, we take 
\begin{align}\label{eqn:cusps}
\mathcal{C}_{0}(N):=\{i/M\mid0<M\le N,\ M|N,\ (i,M)=1,\ 0\le i\le (M,N/M)\},
\end{align}
where $0\in\mathcal{C}_{0}(N)$ is considered to be $0/1$. Each rational number $r$ is equivalent only one element of $\mathcal{C}_{0}(N)$ under the action of $\Gamma_{0}(N)$. We note that $1/N$ is equivalent to $\sqrt{-1}\infty$, and we denote it also by $\sqrt{-1}\infty$ as a cusp. The width of a cusp $i/M$ in $\mathcal{C}_{0}(N)$ is given by
\begin{align}\label{eqn:w-cusp}
w^{(i/M)}=N/(M^{2},N).
\end{align}

Let $k\in\mathbf{Z},\ge0$. We assume that $N\ge3$ if $k$ is odd.  Let $M$ be a fixed positive divisor of $N$, and let $c_{0},d_{0}\in\mathbf{Z}$ with $(c_{0},N/M)=1,\,(d_{0},M)=1$. For $s\in\mathbf{C}$, we define Eisenstein series for $k+2\Re s>2$, by
\begin{align*}
G_{k}(z;c_{0},d_{0};M,N;s)&:=\mathop{{\sum}'}_{c\equiv c_{0}\,(N/M)\atop d\in M^{-1}d_{0}+\mathbf{Z}}(cz+d)^{-k}|cz+d|^{-2s},\\
E_{k}(z;c_{0},d_{0};M,N;s)&:=M^{-k-2s}\mathop{{\sum}'}_{{c\equiv c_{0}\,(N/M)\atop d\in M^{-1}d_{0}+\mathbf{Z}}\atop(Mc,Md)=1}(cz+d)^{-k}|cz+d|^{-2s}
\end{align*}
where ${\sum}'$ implies that the term corresponding to $c=d=0$ is omitted in the summation. 

 Let $\rho\in(\mathbf{Z}/M)^{\ast},\rho' \in(\mathbf{Z}/(N/M))^{\ast}$ with $N/M=\mathfrak{e}_{\rho'}$ so that $\rho\rho'$ has the same parity as $k$. We define Eisenstein series
\begin{align*}
G_{k,\rho,M}^{\rho'}(z,s)&:=\tfrac{\Gamma(k+s)}{(-2\sqrt{{-}1}\pi)^{k}\tau(\overline{\widetilde{\rho}})}\sum_{c_{0}:(\mathbf{Z}/(N/M))^{\times}}\sum_{d_{0}:(\mathbf{Z}/M)^{\times}}\overline{\rho}(d_{0})\rho'(c_{0})G_{k}(z;c_{0},d_{0};M,N;s),\\
E_{k,\rho,M}^{\rho'}(z,s)&:=2^{-1}M^{-k-2s}\sum_{c_{0}:(\mathbf{Z}/(N/M))^{\times}}\sum_{d_{0}:(\mathbf{Z}/M)^{\times}}\overline{\rho}(d_{0})\rho'(c_{0})E_{k}(z;c_{0},d_{0};M,N;s)\\
&=2^{-1}\sum_{(c,d)=1\atop c\equiv0(\mathrm{mod}\,M)}\overline{\rho}(d)\rho'(c/M)(cz+d)^{-k}|cz+d|^{-2s}
\end{align*}
for $\Gamma_{0}(N)$ with character $\rho\rho'$, where $(\mathbf{Z}/M)^{\times}$ denotes the reduced residue class modulo $M$ and $d_{0}:(\mathbf{Z}/M)^{\times}$ implies that $d_{0}$ runs over the complete set of representatives. We omit $M$ from $G_{k,\rho,M}^{\rho'}$ and $E_{k,\rho,M}^{\rho'}$ when $M=\mathfrak{e}_{\rho}$.  We also omit $\rho$ or  $\rho'$ if $\rho=\mathbf{1}$ or $\rho'=\mathbf{1}$. There holds equalities.
\begin{align}
&G_{k,\rho,M}^{\rho'}(z,s)\nonumber\\
=&(\sqrt{{-}1}\pi)^{-k}\rho'(-1)2^{-k+1}M^{k+2s}\mathfrak{f}_{\rho}^{-1}\tau(\widetilde{\rho})\Gamma(k{+}s)L(k{+}2s,\overline{\rho}\rho')E_{k,\rho,M}^{\rho'}(z,s),\label{eqn:G-E}
\end{align}
and if $M=\mathfrak{e}_{\rho}$, then $M^{k+2s}E_{k,\rho}^{\rho'}(z,s)|_{\left({\,0\ -1\atop N\ 0\,}\right)} = \rho'(-1)(N/M)^{k+2s}E_{k,\overline{\rho}'}^{\overline{\rho}}(z,s)$.

\begin{lem}\label{lem:eisi} (i) The Eisenstein series $G_{k,\rho,M}^{\rho'}(z,s),E_{k,\rho,M}^{\rho'}(z,s)$ as functions of $s$ extend meromorphically to the whole complex plane.

(ii) We fix $s$ so that the Eisenstein series are holomorphic at $s$. The constant terms of $G_{k,\rho,M}^{\rho'}(z,s),E_{k,\rho,M}^{\rho'}(z,s)$ with respect to $x$ at each cusp are linear combination of $1$ and $y^{-k+1-2s}$, and the Eisenstein series minus the constant terms are rapidly decreasing as $y\longrightarrow\infty$.
\end{lem}
\begin{proof}(i) The Eisenstein series are written as finite linear combinations of functions in the form $g_{k}(c_{0},d_{0};s)$. This shows the assertion.

(ii) Let $r$ be a cusp, and $A_{r}$ be as in (\ref{defAr}). Then $G_{k,\rho,M}^{\rho'}(z,s)|_{A_{r}},E_{k,\rho,M}^{\rho'}(z,s)|_{A_{r}}$ are also written as combinations of functions in the form $g_{k}(c_{0},d_{0};s)$. Then our assertion follows from the Fourier expansion (\ref{feg}).
\end{proof}
The Fourier expansion of $G_{k,\rho,M}^{\rho'}(z,s)$ is given by
\begin{align}
&G_{k,\rho,M}^{\rho'}(z,s)\nonumber\\
=&\delta_{M,N}(\sqrt{{-}1}\pi)^{-k}2^{-k+1}N^{k+2s}\mathfrak{f}_{\rho}^{-1}\tau(\widetilde{\rho})\Gamma(k+s)L(k+2s,\overline{\rho})\nonumber\\
&+\delta_{\mathfrak{f}_{\rho},1}\pi^{-k{+}1}2^{-2k+3-2s}\varphi(M)\tfrac{\Gamma(k-1+2s)}{\Gamma(s)}L(k{-}1{+}2s,\rho')y^{-k+1-2s}\nonumber\\
&+\tfrac{2\Gamma(k+s)}{(-2\sqrt{{-}1}\pi)^{k}}\sum_{-\infty<n<\infty\atop n\ne 0}n^{-k}|n|^{1-2s}\sum_{0<R|\mathfrak{e}_{\rho}\mathfrak{f}_{\rho}^{-1}}\tfrac{\mu(R)\varphi(M)}{\varphi(\mathfrak{f}_{\rho}R)}\nonumber\\
&\mbox{\hspace{3em}}\times\sum_{0<d|n}\overline{\widetilde{\rho}}(R)(\widetilde{\rho}\mathbf{1}_{R}\mathcal{I}_{\mathbf{Z}})(d\mathfrak{f}_{\rho}RM^{-1})\rho'(n/d)d^{k-1+2s}w_{n}(y,k,s)\mathbf{e}(nx) \label{eqn:fe}
\end{align}
where $\delta_{M,N},\delta_{\mathfrak{f}_{\rho},1}$ denote the Kronecker delta. The Fourier expansion of $E_{k,\rho,M}^{\rho'}(z,s)$ is obtained from (\ref{eqn:G-E}) and (\ref{eqn:fe}). 

For later use we write down the constant term of the Fourier expansion with respect to $x$, of $y^{s}E_{k,\mathbf{1}_{N},N}(z,s)$ for even $k$ at each cusp $i/M\in\mathcal{C}_{0}(N)$. If we denote the constant term at $\sqrt{-1}\infty$ by $y^{s}+\xi^{(1/N)}(s)y^{-k+1-s}$, and the constant term at $i/M\ (M\ne N)$ by $\xi^{(i/M)}(s)y^{-k+1-s}$, then
\begin{align}
\xi^{(i/M)}(s)=\tfrac{(-1)^{k/2}\pi\Gamma(k-1+2s)}{2^{k-2+2s}\Gamma(s)\Gamma(k+s)}\tfrac{\varphi(N)}{NM^{k-1+2s}\varphi(N/M)}\tfrac{\zeta(k-1+2s)\prod_{p|(N/M)}(1-p^{-k+1-2s})}{\zeta(k+2s)\prod_{p|N}(1-p^{-k-2s})}. \label{eqn:ctk}
\end{align}
In particular if $k=0$, then
\begin{align}
\xi^{(i/M)}(s)=\tfrac{\pi^{1/2}\Gamma(s-1/2)}{\Gamma(s)}
\tfrac{M^{1-2s}\varphi(N)}{N\varphi(N/M)}\tfrac{\zeta(-1+2s)\prod_{p|(N/M)}(1-p^{1-2s})}{\zeta(2s)\prod_{p|N}(1-p^{-2s})}\tfrac{}{}, \label{eqn:ct0}
\end{align}
all of which have poles of order $1$ at $s=1$ with same residue $3\pi^{-1}[\Gamma_{0}(1):\Gamma_{0}(N)]^{-1}$ where $[\Gamma_{0}(1):\Gamma_{0}(N)]=N^{-1}\prod_{p|N}(1+p^{-1})^{-1}$. There holds an equality $\mathrm{tr}_{\Gamma_{0}(1)/\Gamma_{0}(N)}(y^{s}E_{0,\mathbf{1}_{N},N}(z,s))=y^{s}E_{0}(z,s)=E_{\Gamma_{0}(1)}(z,s)$. 

For a square free $S\in\mathbf{N}$ and for a character $\chi$, let us define an operator on the function on $\mathfrak{H}$ by 
\begin{align*}
\Lambda_{S,s,\chi}f(z)&:=\sum_{0<R|S}\mu(S/R)\chi(S/R)R^{s}f(Rz).
\end{align*}
Then 
\begin{align}
G_{k,\rho,M}^{\rho'}(z,s)=(\tfrac{M}{\mathfrak{e}_{\rho}})^{k+2s}\Lambda_{\mathfrak{e}_{\rho}\mathfrak{f}_{\rho}^{-1},k+2s,\overline{\widetilde{\rho}}}G_{k,\widetilde{\rho}}^{\rho'}(\tfrac{M}{\mathfrak{e}_{\rho}}z,s).\label{eqn:ses}
\end{align}
Now we assume that $\rho'$ is primitive and $N/M=\mathfrak{f}_{\rho'}$. If $\rho$ is also primitive, then it follows from (\ref{eqn:fe}), the functional equation $y^{-k+1-s}G_{k,\rho}^{\rho'}(z,{-}k{+}1{-}s)=\pi^{-k+1-2s}y^{s}G_{k,\rho'}^{\rho}(z,s)$. Then for $\rho$ not necessarily primitive, we have the functional equations by (\ref{eqn:ses}) as
\begin{align}
&y^{-k+1-s}G_{k,\rho,M}^{\rho'}(z,{-}k{+}1{-}s)=\pi^{-k+1-2s}\tfrac{M}{\mathfrak{e}_{\rho}}y^{s}\Lambda_{\mathfrak{e}_{\rho}\mathfrak{f}_{\rho}^{-1},1,\overline{\widetilde{\rho}}}G_{k,\rho'}^{\widetilde{\rho}}(\tfrac{M}{\mathfrak{e}_{\rho}}z,s),\nonumber\\
&y^{-k+1-s}E_{k,\rho,M}^{\rho'}(z,{-}k{+}1{-}s)\nonumber\\
=&(-1)^{k}\pi^{-k+1-2s}\tfrac{(M\mathfrak{f}_{\rho'})^{k-1+2s}\tau(\rho')\Gamma(k+s)L(k+2s,\widetilde{\rho}\overline{\rho}')}{\mathfrak{e}_{\rho}\mathfrak{f}_{\rho}^{-1}\tau(\widetilde{\rho})\Gamma(1-s)L(-k+2-2s,\overline{\rho}\rho')}y^{s}\Lambda_{\mathfrak{e}_{\rho}\mathfrak{f}_{\rho}^{-1},1,\overline{\widetilde{\rho}}}E_{k,\rho'}^{\widetilde{\rho}}(\tfrac{M}{\mathfrak{e}_{\rho}}z,s). \label{eqn:fe2}
\end{align}
In particular for $\rho'=\mathbf{1}$ we obtain form (\ref{eqn:fe2}),
\begin{align}
&y^{-k+1-s}E_{k,\rho,N}(z,{-}k{+}1{-}s)\nonumber\\
=&(-1)^{k}\pi^{-k+1-2s}\tfrac{N^{k-1+2s}\Gamma(k+s)L(k+2s,\widetilde{\rho})}{\mathfrak{e}_{\rho}\mathfrak{f}_{\rho}^{-1}\tau(\widetilde{\rho})\Gamma(1-s)L(-k+2-2s,\overline{\rho})}y^{s}\Lambda_{\mathfrak{e}_{\rho}\mathfrak{f}_{\rho}^{-1},1,\overline{\widetilde{\rho}}}E_{k}^{\widetilde{\rho}}(\tfrac{N}{\mathfrak{e}_{\rho}}z,s)\nonumber\\
=&\pi^{-1/2}z^{-k}U_{k,\overline{\rho}}(s)\sum_{0<P|\mathfrak{e}_{\rho}\mathfrak{f}_{\rho}^{-1}}U_{k,\overline{\rho},P}(s)(\Im(-\tfrac{1}{Nz})^{s}E_{k,\overline{\widetilde{\rho}}\mathbf{1}_{P},\mathfrak{f}_{\rho}P}(-\tfrac{1}{Nz},s) \label{eqn:felu}
\end{align}
for $\rho\in(\mathbf{Z}/N)^{\ast}$ with the same parity as $k$ and with
\begin{align}
U_{k,\rho}(s)&:=\tfrac{(\sqrt{{-}1})^{k}N^{-1+s}\mathfrak{e}_{\rho}^{-1}\mathfrak{f}_{\rho}^{2}\Gamma(s)\Gamma(k{+}s)L(k{+}2s,\overline{\widetilde{\rho}})}{\Gamma((k{-}1)/2{+}s)\Gamma(k/2{+}s)L(k{-}1{+}2s,\overline{\widetilde{\rho}})\prod_{p|\mathfrak{e}_{\rho}\mathfrak{f}_{\rho}^{-1}}(1{-}\widetilde{\rho}(p)p^{-k+2-2s})},\label{eqn:Uk1}\\
U_{k,\rho,P}(s)&:=\prod_{p|P}(1{-}\widetilde{\rho}(p)p^{k+2s})\varphi(\mathfrak{e}_{\rho}\mathfrak{f}_{\rho}^{-1}P^{-1}).\label{eqn:Uk2}
\end{align}

Since $E_{\Gamma_{0}(N)}(z,s)=y^{s}E_{0,\mathbf{1}_{N},N}(z,s)$, we have $E_{\Gamma_{0}(N)}(z,1{-}s)=\pi^{-1/2}U_{0}(s)\times$ $\sum\limits_{0<P|\mathfrak{e}_{\mathbf{1}_{N}}}U_{0,P}(s)E_{\Gamma_{0}(P)}(-\tfrac{1}{Nz},s) $ with $\mathfrak{e}_{\mathbf{1}_{N}}=\prod_{p|N}p$ and with
\begin{align}\hspace*{-.3em}
U_{0}(s)&:=\tfrac{N^{{-}1{+}s}\mathfrak{e}_{\mathbf{1}_{N}}^{-1}\Gamma(s)\zeta(2s)}{\Gamma({-}1/2{+}s)\zeta({-}1{+}2s)\prod_{p|\mathfrak{e}_{\mathbf{1}_{N}}}(1{-}p^{2{-}2s})},\hspace{.2em}U_{0,P}(s):=\prod\limits_{p|P}(1{-}p^{2s})\varphi(\mathfrak{e}_{\mathbf{1}_{N}}P^{{-}1}).\label{eqn:U}
\end{align}

We use these result in the later sections. The Eisenstein series $y^{s}E_{0,\rho,N}(z,s)$ of weight $0$ is an eigenfunction of the Laplacian $\Delta$ unlike the Eisenstein series with $k>0$. By the similar argument as in Section \ref{sect:PSP}, we obtain the following two theorems, making use of the Eisenstein series of weight $0$.

\begin{thm}\label{thm:lseries}
Let $f,g$ be holomorphic modulars forms for $\Gamma_{0}(N)$ of weight $l\in\frac{1}{2}\mathbf{Z},l\ge1/2$ with the same character. Then $L(s;f,g)$ extends meromorphically to the whole $s$-plane, and
\begin{align}
\langle f,g\rangle_{\Gamma_{0}(N)}=3^{-1}4^{-l}\pi^{-l+1}\Gamma(l)N\prod_{p|N}(1+p^{-1})\mathrm{Res}_{s=l}L(s;f,g).\label{eqn:psp-formula2}
\end{align}
except for $l=1$. Let $\widetilde{f\overline{g}}(z):=(f\overline{g})|_{S_{N}}(z)=|N^{1/2}z|^{-2l}(f\overline{g})(-1/(Nz))\in\mathcal{M}_{l.l}(N)$ with $S_{N}=\mbox{\tiny$\left(\begin{array}{@{}c@{\,}c@{}}0&-1/\sqrt{N}\\\sqrt{N}&0\end{array}\right)$}$. If $\mathrm{tr}_{\Gamma_{0}(N)/\Gamma_{0}(M)}(\widetilde{f\overline{g}})$ for $M|N$ has $\sum_{n=0}^{\infty}c_{n}^{(M)}e^{-4\pi ny}$ as the constant term of its Fourier expansion with respect to $x$, then we put \linebreak$L(s;\mathrm{tr}_{\Gamma_{0}(N)/\Gamma_{0}(M)}(\widetilde{f\overline{g}})):=\sum_{n=1}^{\infty}c_{n}^{(M)}n^{-s}$. Then we have a functional equation
\begin{align*}
L(l-s;f,g)=&\frac{2^{2-4s}\pi^{1/2-2s}\Gamma(l{-}1{+}s)U_{0}(s)}{\Gamma(l{-}s)}{\textstyle\sum\limits_{P|\mathfrak{e}_{\mathbf{1}_{N}}}}U_{0,P}(s)L(l{-}1{+}s;\mathrm{tr}_{\Gamma_{0}(N)/\Gamma_{0}(P)}(\widetilde{f\overline{g}})) 
\end{align*}
with $U_{0}(s),U_{0,P}(s)$ as in (\ref{eqn:U}).
\end{thm}
\begin{proof}We prove only the functional equation. Let $Q(f,g)$ be as in the proof of Theorem \ref{thm:psp}.  Then
\begin{align*}
&-(4\pi)^{-l+s}(l-1+s)\Gamma(l+1-s)L(l-s;f,g)\\
=&\int_{\mathfrak{F}(N)}Q(f,g)(z)E_{\Gamma_{0}(N)}(z,1-s)\tfrac{dxdy}{y^{2}}\\
=&\pi^{-1/2}U_{0}(s){\textstyle\sum\limits_{P|\mathfrak{e}_{\mathbf{1}_{N}}}\prod\limits_{p|P}(1{-}p^{2s})\varphi(\mathfrak{e}_{\rho}P^{-1})}\int_{\mathfrak{F}(N)}Q(f,g)(z)E_{\Gamma_{0}(P)}(-\tfrac{1}{Nz},s)\tfrac{dxdy}{y^{2}}\\
=&\pi^{-1/2}U_{0}(s){\textstyle\sum\limits_{P|\mathfrak{e}_{\mathbf{1}_{N}}}\prod\limits_{p|P}(1{-}p^{2s})\varphi(\mathfrak{e}_{\rho}P^{-1})}\int_{\mathfrak{F}(N)}Q(f|_{S_{N}},g|_{S_{N}}(z)E_{\Gamma_{0}(P)}(z,s)\tfrac{dxdy}{y^{2}}\\
=&-\pi^{-1/2}U_{0}(s){\textstyle\sum\limits_{P|\mathfrak{e}_{\mathbf{1}_{N}}}\prod\limits_{p|P}(1{-}p^{2s})\varphi(\mathfrak{e}_{\rho}P^{-1})}\\
&\hspace{5em}\times(4\pi)^{-l+1-s}(l-s)\Gamma(l+s)L(l{-}1{+}s;\mathrm{tr}_{\Gamma_{0}(N)/\Gamma_{0}(P)}(\widetilde{f\overline{g}})),
\end{align*}
which shows the functional equation.
\end{proof}

\begin{thm}\label{thm:lseries2} Let $f,g$ be holomorphic modulars forms for $\Gamma_{0}(N)$ of weight $l\in\frac{1}{2}\mathbf{Z},l\ge1/2$ with characters. We assume that $f\overline{g}\in\mathcal{M}_{l,l}(N,\rho)$ with $\rho\in(\mathbf{Z}/N)^{\ast},\ne\mathbf{1}_{N}$. Then $L(s;f,g)$ converges for $s$ with $\Re s>\max\{2l{-}1,1/2\}$, and extends meromorphically to the whole $s$-plane. Let $\widetilde{f\overline{g}}(z):=(f\overline{g})|_{S_{N}}(z)\in\mathcal{M}_{l.l}(N,\overline{\rho})$ with $S_{N}:=\mbox{\tiny$\left(\begin{array}{@{}c@{\,}c@{}}0&-N^{-1/2}\\N^{1/2}&0\end{array}\right)$}$. For $M\in\mathbf{N}$ with $\mathfrak{f}_{\rho}|M|N$, let $\mathrm{tr}_{\Gamma_{0}(N)/\Gamma_{0}(M),\rho}(\widetilde{f\overline{g}}):=\sum_{A:\Gamma_{0}(N)/\Gamma_{0}(M)}$ $\rho(d)\widetilde{f\overline{g}}|_{A}(z)$, $d$ being the $(2,2)$ entry of $A$. Then we have a functional equation
\begin{align*}
&L(l-s;f,g)\\
=&\frac{2^{2-4s}\pi^{1/2-2s}N^{-1+s}\mathfrak{e}_{\rho}^{-1}\mathfrak{f}_{\rho}^{2}\Gamma(l{-}1{+}s)\Gamma(s)L(2s,\overline{\widetilde{\rho}})}
{\Gamma(l-s)\Gamma({-}1/2{+}s)L({-}1{+}2s,\overline{\widetilde{\rho}})\prod_{p|\mathfrak{e}_{\rho}\mathfrak{f}_{\rho}^{-1}}(1{-}\widetilde{\rho}(p)p^{2-2s})}\\
&\hspace{1em}\times\sum_{P|\mathfrak{e}_{\rho}\mathfrak{f}_{\rho}^{-1}}\prod_{p|P}(1{-}\widetilde{\rho}(p)p^{2s})\varphi(\mathfrak{e}_{\rho}\mathfrak{f}_{\rho}^{-1}P^{-1})L(l{-}1{+}s;\mathrm{tr}_{\Gamma_{0}(N)/\Gamma_{0}(P\mathfrak{f}_{\rho}),\overline{\rho}}(\widetilde{f\overline{g}})).
\end{align*}
\end{thm}
\begin{proof}Let $Q(f,g)$ be as in (\ref{eqn:defQ}). Then $Q(f,g)$ is a real analytic automorphic form with character $\rho$, and $Q(f,g)y^{s}E_{0,\overline{\rho},N}(z,s)$ is an automorphic form with trivial character. Then $\int_{\mathfrak{F}(\Gamma)}Q(f,g)y^{s}E_{0,\overline{\rho},N}(z,s)\frac{dxdy}{y^{2}}$ is well-defined, and it has a  meromorphic continuation to whole the $s$-plane since $Q(f,g)$ is rapidly decreasing at cusps. As in the proof of Theorem \ref{thm:psp}, it is shown to be equal to $-(4\pi)^{-l+1-s}(l-s)\Gamma(l+s)L(l-1+s;f,g)$ by a standard unfolding trick. The functional equation is proved in the same manner as in the proof of Theorem \ref{thm:lseries}.
\end{proof}

We note that the function $E_{0,\rho,N}(z,s)$ of $s$ is holomorphic on the real axis with $s\ge1$ for even $\rho\ne\mathbf{1}_{N}$. So the formula of type (\ref{eqn:psp-formula2}) is not obtained in this case.

Theorem \ref{thm:lseries} and Theorem \ref{thm:lseries2} are generalized in Corollary \ref{cor:lseries} and in Corollary \ref{cor:lseries2}. 

The scalar product defined in Petersson \cite{Petersson} satisfies the equality $\langle f,g\rangle_{\Gamma_{0}(N)}=\langle f|_{S_{N}},g|_{S_{N}}\rangle_{\Gamma_{0}(N)}$ for holomorphic cusp forms $f,g$  for $\Gamma_{0}(N)$ of same weight and with same character where $S_{N}$ is as in Theorem \ref{thm:lseries2}. Let $f,g$ be real analytic modular forms satisfying (\ref{eqn:dfP}). If $Q_{y^{l}f\overline{g}}^{(r)}(T)$ for all cusps $r$ have only terms given by the definite integrals form $0$ to $\infty$, namely, no $Q_{y^{l}f\overline{g}}^{(r)}(T)$ have a term containing $\log T$, then the equality also holds for the scalar product (\ref{eqn:psp}). However it does not hold in general.
\begin{lem}\label{lem:psp-sup} Let $f,g$ be as above. Then
$\langle f|_{S_{N}},g|_{S_{N}}\rangle_{\Gamma_{0}(N)}=\langle f,g\rangle_{\Gamma_{0}(N)}-\sum_{i/M\in\mathcal{C}_{0}(N)}$ $w^{(i/M)}c_{1}^{(i/M)}\log(N/M^{2})$ with $w^{(i/M)}$ in (\ref{eqn:w-cusp}) and with $c_{1}^{(i/M)}$ in (\ref{eqn:dfP}).
\end{lem}
\begin{proof} Let $\mathbf{T}=(T^{(r)})_{r\in\mathcal{C}_{0}(N)}$ with $T^{(r)}>N$, and let $\mathfrak{F}_{\mathbf{T}}(N)$ denote the domain obtained from $\mathfrak{F}(N)$ by cutting off neighborhoods of cusps $r$ along the lines $\Im(A_{r}^{-1}z)=T^{(r)}\ (r\in\mathcal{C}_{0}(N))$. As easily seen, the equality
\begin{align*}
\langle f,g\rangle_{\Gamma}=\lim_{T^{(r)}\to\infty\atop r\in\mathcal{C}_{0}(N)}\left(\rule{0cm}{1.1em}\right.\int_{\mathfrak{F}_{\mathbf{T}}(N)}y^{l}f(z)\overline{g}(z)\tfrac{dxdy}{y^{2}}{-}\sum_{r}Q_{y^{l}f\overline{g}}^{(r)}(T^{(r)})\left.\rule{0cm}{1.1em}\right)
\end{align*}
holds in the notation of (\ref{eqn:psp}).

The matrix $S_{N}$ maps $\mathfrak{F}(N)$ onto the fundamental domain of $\Gamma_{0}(N)$, and hence $S_{N}\mathfrak{F}(N)$ can be decomposed into a finite number of pieces so that the union of their suitable translations by matrices in $\Gamma_{0}(N)$, is equal to $\mathfrak{F}(N)$. Let $\phi_{S_{N}}$ denote the map of $\mathfrak{F}(N)$ onto itself obtained in this manner. Then $\phi_{S_{N}}^{2}$ is the identity map of $\mathfrak{F}(N)$. The map $\phi_{S_{N}}$ can be naturally extended to $\mathfrak{F}(N)\cup\mathcal{C}_{0}(N)$, and then a cusp in the form $i/M\ (M|N)$ in $\mathcal{C}_{0}(N)$ is mapped to a cusp in the form $j/(N/M)$ and vice versa. If we take $\mathbf{T}$ so that $T^{(i/M)}=(N/M^{2})T$, then $\phi_{S_{N}}(\mathfrak{F}_{T}(N))=\mathfrak{F}_{\mathbf{T}}(N)$. Hence $\int_{\mathfrak{F}_{T}(N)}(y^{l}f\overline{g})|_{S_{N}}(z)\tfrac{dxdy}{y^{2}}=\int_{\mathfrak{F}_{\mathbf{T}}(N)}(y^{l}f\overline{g})(z)\tfrac{dxdy}{y^{2}}$ for sufficiently large $T$.

Since $((y^{l}f\overline{g})|_{S_{N}})|_{A_{j/(N/M})}(z)=((y^{l}f\overline{g})|_{S_{N}})|_{A_{i/M}}((N/M^{2})z)$, there is an equality $P_{(y^{l}f\overline{g})|_{S_{N}}}^{(j/(N/M))}(y)=P_{y^{l}f\overline{g}}^{(i/M)}((N/M^{2})y)$. Let $\widetilde{P}_{y^{l}f\overline{g}}^{(i/M)}(y):=\sum_{\Re\nu_{j}^{(i/M)}\ge1,\nu_{j}^{(i/M)}\ne1}c_{\nu_{j}^{(i/M)}}^{(i/M)}$ $\times y^{\nu_{j}^{(i/M)}}$. Then, noticing that $(N/M^{2})w^{(j/(N/M)}=w^{(i/M)}$, we have
\begin{align*}
Q_{(y^{l}f\overline{g})|_{S_{N}}}^{(j/(N/M))}(T)&=w^{(j/(N/M))}\left\{\rule{0cm}{.9em}\right.c_{1}^{(i/M)}(N/M^{2})\log T+\int_{0}^{T}\widetilde{P}_{y^{l}f\overline{g}}^{(i/M)}((N/M^{2})y)y^{-2}dy\left.\rule{0cm}{.9em}\right\}\\
&=w^{(i/M)}\left\{\rule{0cm}{.9em}\right.c_{1}^{(i/M)}\log T+\int_{0}^{(N/M^{2})T}\widetilde{P}_{y^{l}f\overline{g}}^{(i/M)}(y)y^{-2}dy\left.\rule{0cm}{.9em}\right\}\\
&=Q_{y^{l}f\overline{g}}^{(i/M)}((N/M^{2})T)-w^{(i/M)}c_{1}^{(i/M)}\log(N/M^{2}).
\end{align*}
It follows an equality
\begin{align*}
&\int_{\mathfrak{F}_{T}(N)}(y^{l}f\overline{g})_{S_{N}}(z)\tfrac{dxdy}{y^{2}}{-}\sum_{r}Q_{y^{l}f\overline{g}}^{(r)}(T^{(r)})\\
=&\int_{\mathfrak{F}_{\mathbf{T}}(N)}(y^{l}f\overline{g})(z)\tfrac{dxdy}{y^{2}}{-}\sum_{r}Q_{y^{l}f\overline{g}}^{(r)}(T^{(r)})-\sum_{i/M\in\mathcal{C}_{0}(N)}w^{(i/M)}c_{1}^{(i/M)}\log(N/M^{2}),
\end{align*}
which shows the lemma.
\end{proof}
\section{Eisenstein series of half integral weight}\label{sect:ESHIW}
In this section, we study analytic property of real analytic Eisenstein series for $\Gamma_{0}(N)$ of weight $(k+1/2+s,s)$ with $k\in\mathbf{Z},\ge0$, or equivalently  $y^{s}$ times it of half integral weight $k+1/2$. The main purpose is to obtain the Fourier expansion of some specific Eisenstein series which is necessary to investigate the $L$-function (\ref{eqn:lseries}) through the unfolding trick.

For $k\in\mathbf{Z},\ge0$, $8|N$ and $0\le c_{0},d_{0}<N,(N,c_{0},d_{0})=1$, we put $g_{k+1/2}(z;c_{0},d_{0};N;s)$ $:=\sum_{0\le c\equiv c_{0}(N)\atop{d\equiv d_{0}(N)\atop(c,d)=1}}\chi_{c}(d)(cz+d)^{-k-1/2}|cz+d|^{-2s}$ with $2|c_{0}$, and put $g_{k+1/2}'(z;c_{0},d_{0};N;s):=\sum_{0<c\equiv c_{0}(N)\atop{d\equiv d_{0}(N)\atop(c,d)=1}}\chi_{c^{\vee}}(d)(cz+d)^{-k-1/2}|cz+d|^{-2s}$ with $2\nmid c_{0}$, where we mean $\chi_{0}(\pm1)$ as $1$. These series converge if $k+1/2+2\Re s>2$, and have a Fourier expansion in the form
\begin{align*}
a_{0}+a_{0}'(s)w_{0}(y,k+1/2,s)+\sum_{n\ne0}a_{n/N}(s)w_{n/N}(y,k+1/2,s)\mathbf{e}(nx/N).
\end{align*}
\begin{lem}The Eisenstein series $g_{k+1/2}(z;c_{0},d_{0};N;s),g_{k+1/2}'(z;c_{0},d_{0};N;s)$ as functions of $s$ extend meromorphically to the whole complex plane.
\end{lem}
\begin{proof} At first we note that 
\begin{align}
\sum_{m\equiv a(N),>0\atop(m,Q)=1}\frac{\varphi(m)}{m^{s}},\sum_{m\equiv a(N)\atop{(n,Q)=1\atop m:\mathrm{square\,free}}} \frac{\rho(m)}{m^{s}}\label{dseries}
\end{align}for $(a,N)=1,Q\in\mathbf{N},\rho\in(\mathbf{Z}/N)^{\ast}$ extend meromorphically to the whole complex plane. Indeed the first one is equal to $\frac{1}{\varphi(N)}\sum_{\chi\in(\mathbf{Z}/N)^{\ast}}\overline{\chi}(a)\frac{L(s-1,\chi\mathbf{1}_{Q})}{L(s,\chi\mathbf{1}_{Q})}$, and the second is equal to $\frac{1}{\varphi(N)}\sum_{\chi\in(\mathbf{Z}/N)^{\ast}}\frac{\overline{\chi}(a)}{L(s,\chi\rho\mathbf{1}_{Q})}$.

Let $P:=\prod_{p|N,(p,N/(c_{0},N))=1}p$. Let $e_{p}\ (p|P)$ be the order of $p$ modulo $N/(c_{0},N)$. As for $g_{k+1/2}(z;c_{0},d_{0};N;s)$, we have $a_{0}=\delta_{c_{0},0}\delta_{d_{0},1}$ and $a_{0}'(s)=(c_{0},N)^{{-}k{-}1/2{-}2s}$\linebreak$\times\varphi((c_{0},N))\prod_{p|P}(1{-}p^{{-}2e_{p}(k{+}1/2{+}2s)})^{{-}1}\sum_{p|P,0\le j_{p}<2e_{p}}\chi_{(\prod_{p}p^{j_{p}})(c_{0},N)}(d_{0})$\linebreak$\times\prod_{p|P}p^{-j_{p}(k+1/2+2s)}\sum_{n>0,(m,N)=1\atop (\prod_{p}p^{j_{p}})n^{2}\equiv c_{0}/(c_{0},N)(\mathrm{mod}\,N/(c_{0},N))}$ $\varphi(m)m^{-2k-4s}$. The last summation is $0$ or a finite linear combination of Dirichlet series  in the form of the first series in (\ref{dseries}). As for $g_{k+1/2}'(z;c_{0},d_{0};N;s)$, we have we have $a_{0}=0$ and $a_{0}'(s)=(c_{0},N)^{-k-1/2-2s}\varphi((c_{0},N))\prod_{p|P}(1-p^{-2e_{p}(k+1/2+2s)})^{-1}\sum_{p|P,\,0\le j_{p}<2e_{p}}$\\$\chi_{\{(\prod_{p}p^{j_{p}})(c_{0},N)\}^{\vee}}(d_{0})\prod_{p|P}p^{-j_{p}(k+1/2+2s)}\sum_{m>0,(m,N)=1\atop (\prod_{p}p^{j_{p}})n^{2}\equiv c_{0}/(c_{0},N)(\mathrm{mod}\,N/(c_{0},N))}\varphi(m)$\linebreak$\times m^{-2k-4s}$. In either case, the constant term extends meromorphically to the whole $s$ plane.

For $n\ne0$, we consider $a_{n/N}(s)$ of $g_{k+1/2}(z;c_{0},d_{0};N;s)$. Let $n=(c_{0},N)n_{P}n'\in\mathbf{Z}$ where $|n'|$ is the maximal divisor of $n$ coprime to $N$, and $n_{P}>0$ is the divisor of $n$ with $\mathrm{rad}(n_{P})|P$. Then it is checked that $a_{n/N}(s)$ is $0$ if $n$ is not in this form. Put $t_{n'}(p^{i}):=\varphi(p^{i})p^{-i(k+1/2+2s)}$ if $2|i\ge0,i\le v_{p}(n')$, $t_{n'}(p^{i}):=-p^{-(v_{p}(n')+1)(k+1/2+2s)}$ if $2|i\ge0,i=v_{p}(n)+1$, $t_{n'}(p^{i}):=p^{1/2+i-1}$ if $2\nmid i\ge0,i=v_{p}(n)+1$, and $t_{n'}(p^{i}):=0$ if otherwise.  Then $a_{n/N}(s)=\sum_{M|n_{P}}\chi_{(c_{0},N)M}(d_{0})$ $((c_{0},N)M)^{-k+1/2-2s}\sum_{M'|n'}\prod_{p|M'}$ $t_{n'}(p^{v_{p}(M')})\sum_{(m,Nn)=1\atop{m\equiv\overline{MM'}c_{0}/(c_{0},N)(\mathrm{mod}\,N/(c_{0},N))\atop m:\mathrm{square\,free}}}\iota_{M'm}$\\$\times\chi_{-4}(m)^{(d_{0}-1)/2}\mathbf{e}(\tfrac{nd_{0}\overline{c}'}{N(c_{0},N)M})\chi_{(c_{0},N)M}(m)m^{-k-1/2-2s}$ where $\overline{M},\overline{M}'$ are inverses of $M,M'$ modulo $N/(c_{0},N)$ respectively, and $\overline{c}'$ is an inverse of $c'$ modulo $N$.  The summations and the product of the right hand side of this equation are all finite except for the last summation, and the last summation is written as a finite linear combination of the series in the form of the second series in (\ref{dseries}). Then $a_{n/N}(s)$ extends meromorphically to the whole $s$ plane.

The coefficient  $a_{n/N}(s)$ of $g_{k+1/2}'(z;c_{0},d_{0};N;s)$ also vanishes if $n$ is not in the form $n=(c_{0},N)n_{P}n'$. For such $n$,  $a_{n/N}(s)=\sum_{M|n_{P}}\chi_{\{(c_{0},N)M\}^{\vee}}(d_{0})$\linebreak$\times((c_{0},N)M)^{-k+1/2-2s} \sum_{M'|n'}\prod_{p|M'}t_{n'}(p^{v_{p}(M')})\sum_{(m,Nn)=1\atop{m\equiv\overline{MM'}c_{0}/(c_{0},N)(\mathrm{mod}\,N/(c_{0},N))\atop m:\mathrm{square\,free}}}\iota_{M'm}$\linebreak$\times\mathbf{e}(\tfrac{nd_{0}\overline{c}'}{N(c_{0},N)M})\chi_{\{(c_{0},N)M\}^{\vee}}(m)m^{-k-1/2-2s}$, which extends meromorphically to the whole complex plane. 
\end{proof}

Let $N,M\in\mathbf{N}$ with $4|N, M|N$ so that $4|M$ or $4|(N/M)$. Let $\rho\in(\mathbf{Z}/M)^{\ast},\rho'\in(\mathbf{Z}/(N/M))^{\ast}$ where
\begin{align}
&2|v_{p}(M\mathfrak{f}_{\rho}^{-1})\hspace{1em}(p|M\mbox{ for which }\{\rho\}_{p}\mbox{ is trivial or real}),\label{cond:M}\\
&N/M=\mathfrak{e}_{\rho'}'.\nonumber
\end{align}
$\mathfrak{e}_{\rho'}'$ being as in the introduction. For $k\in\mathbf{Z},\ge0$ with the same parity as $\rho\rho'$, we define an Eisenstein series of weight $(k+1/2+s,s)$ for $\Gamma_{0}(N)$ with character $\rho\rho'$ as follows. If $4|M$, then 
\begin{align}
E_{k+1/2,\rho,M}^{\rho'}(z,s):=\sum_{c,d} \overline{\rho}(d)\rho'(c/M)\chi_{c}(d)\iota_{d}(cz+d)^{-k-1/2}|cz+d|^{-2s},\label{eqn:eshiwe}
\end{align}
and if $2\nmid M$, then
\begin{align}
E_{k+1/2,\rho,M}^{\rho'}(z,s):=\sum_{c,d} \overline{\rho}(d)\rho'(c/M)\chi_{c^{\vee}}(d)\iota_{c}^{-1}(cz+d)^{-k-1/2}|cz+d|^{-2s}\label{eqn:eshiwo}
\end{align}
where $c,d$ run over the set of the second rows of matrices in $A_{1/M}^{-1}\Gamma_{0}(N)$ with $c>0$, or with $c=0$ and $d>0$, $A_{1/M}$ being as in (\ref{defAr}). More precisely $c,d$ satisfy the condition that $c>0,M|c,(c/M,N/M)=1, d\in\mathbf{Z},(d,M)=1, (c,d)=1$ where $c=0,d=1$ is added if $M=N$. We drop the notation $M$ from $E_{k+1/2,\rho,M}^{\rho'}(z,s)$ if $M=\mathfrak{e}_{\rho}'$, and drop $\rho$ or $\rho'$  if $\rho=\mathbf{1}$ or $\rho'=\mathbf{1}$.

\begin{lem}\label{lem:eishi} (i) The Eisenstein series $E_{k+1/2,\rho,M}^{\rho'}(z,s)$ as functions of $s$ extends meromorphically to the whole complex plane.

(ii) We fix $s$ so that the Eisenstein series is holomorphic at $s$. Then constant term of $E_{k+1/2,\rho,M}^{\rho'}(z,s)$ with respect to $x$ at each cusp is a linear combination of $1$ and $y^{-k+1-2s}$, and the Eisenstein series minus the constant term is rapidly decreasing as $y\longrightarrow\infty$.
\end{lem}
\begin{proof} The Eisenstein series $E_{k+1/2,\rho,M}^{\rho'}(z,s)$ is written as a linear combination of  Eisenstein series of the form $g_{k+1/2}(z;c_{0},d_{0};N;s)$ or $g_{k+1/2}'(z;c_{0},d_{0};N;s)$. The rest of the proof is parallel to that of Lemma \ref{lem:eisi}
\end{proof}

Let $\widetilde{\rho}$ be a primitive Dirichlet character. We put 
\begin{align}
\rho:=\widetilde{\rho}\mathbf{1}_{2},\ \ N:=\mathrm{lcm}(4,\mathfrak{f}_{\rho}) \label{eqn:N}
\end{align}
and we consider $\rho$ as a character in $(\mathbf{Z}/N)^{\ast}$. If $2|\mathfrak{f}_{\widetilde{\rho}}$, then the equality $\rho=\widetilde{\rho}$ holds. Obviously $\rho,N$ such as (\ref{eqn:N}) satisfy the condition (\ref{cond:M}) with $M=N$. We closely compute the Fourier expansion of Eisenstein series
\begin{align}
E_{k+1/2,\rho}(z,s):=1+\sum_{c\equiv0(\mathrm{mod}N),c>0\atop(c,d)=1}(\overline{\rho}\chi_{c})(d)\iota_{d}(cz+d)^{-k-1/2}|cz+d|^{-2s}.\label{eqn:eshiw}
\end{align}
For $k\ge2$, $E_{k+1/2,\rho}(z,0)$ is a holomorphic in $z$ and, it is in $\mathbf{M}_{k+1/2}(N,\rho)$. The Eisenstein series has the Fourier expansion for $s\in\mathbf{C}$ with $2\Re s+k+1/2>2$,
\begin{align}
1+c_{k,s,\rho,N}(0)w_{0}(y,k{+}1/2,s)+\sum_{n\ne0}c_{k,s,\rho,N}(n)w_{n}(y,k{+}1/2,s)\mathbf{e}(nx)\label{feoe}
\end{align}with $c_{k,s,\rho,N}(n):=\sum_{m\equiv0(N),m>0}\,m^{-k-1/2-2s}\sum_{i:(\mathbf{Z}/m)^{\times}}(\overline{\rho}\chi_{m})(i)\iota_{i}\mathbf{e}(ni/m)\ (n\in\mathbf{Z})$ by (\ref{eqn:fex}).  Let $\rho_{2}:=\{\rho\}_{2}$, and let $\rho_{\mathbf{c}}$ be the product of complex $\{\rho\}_{p}\ (2\ne p|N)$, and $\rho_{\mathbf{r}}$ be the product  of real $\{\rho\}_{p}\ (2\ne p|N)$, so that
\begin{align}
\rho=\rho_{2}\rho_{\mathbf{c}}\rho_{\mathbf{r}}.\label{eqn:dcmp-rho}
\end{align}
By definition, $\mathfrak{f}_{\rho_{\mathbf{r}}}$ is an odd natural number. We put $\rho_{2\mathbf{c}}:=\rho_{2}\rho_{\mathbf{c}},\rho_{2\mathbf{r}}:=\rho_{2}\rho_{\mathbf{r}},\rho_{\mathbf{c}\mathbf{r}}:=\rho_{\mathbf{c}}\rho_{\mathbf{r}}$. The primitive character $\widetilde{\rho^{2}}$ is equal to $\rho_{2\mathbf{c}}^{2}$ or $\rho_{\mathbf{c}}^{2}$ according as $\rho_{2}$ is complex or not. Put $c_{k,s,\rho,N}'(n):=\sum_{2\nmid m\equiv0(2^{-v_{2}(N)}N),m>0}$ $\overline{\rho}_{2}(m)m^{-k-1/2-2s}\iota_{m}^{-1}\sum_{i:(\mathbf{Z}/m)^{\times}}$ $(\overline{\rho}_{\mathbf{cr}}\chi_{m^{\vee}})(i)\mathbf{e}(ni/m)\ (n\in\mathbf{Z})$, and
\begin{align*}
c_{k,s,\rho,N}^{(2)}(n):&=2^{-1}(1{+}\sqrt{{-}1})\sum_{l=v_{2}(N)}^{\infty}\overline{\rho}_{\mathbf{cr}}(2^{l})2^{-l(k+1/2+2s)}\times\\
&\{\sum_{i:(\mathbf{Z}/2^{l})^{\times}}(\overline{\rho}_{2}\chi_{2^{l}})(i)\mathbf{e}(ni/2^{l}){-}\sqrt{{-}1}\sum_{i:(\mathbf{Z}/2^{l})^{\times}}(\overline{\rho}_{2}\chi_{-2^{l}})(i)\mathbf{e}(ni/2^{l})\},
\end{align*}
where the summation of $l=v_{2}(N)\ge2$ to $\infty$ is actually finite for $n\ne0$. We have the decomposition
\begin{align}
c_{k,s,\rho,N}(n)=c_{k,s,\rho,N}^{(2)}(n)c_{k,s,\rho,N}'(n). \label{eqn:dcmp-c}
\end{align}for $n\in\mathbf{Z}$. Put  $c_{k,s,\rho,N}''(n):=\sum_{(m,N)=1,m>0}$ $\overline{\rho}(m)m^{-k-1/2-2s}\iota_{m}^{-1}\sum_{i:(\mathbf{Z}/m)^{\times}}\chi_{m^{\vee}}(i)$ $\times\mathbf{e}(ni/m)\ (n\in\mathbf{Z})$, and $c_{k,s,\rho,N}^{(\mathbf{c})}(n):=\sum_{m}\rho_{2\mathbf{r}}(m)m^{-k-1/2-2s}\iota_{m}^{-1}\sum_{i:(\mathbf{Z}/m)^\times}$\linebreak$(\overline{\rho}_{\mathbf{c}}\chi_{m^{\vee}})(i)\mathbf{e}(ni/m)$ where $m$ runs over the set of all multiples of  $\mathfrak{f}_{\rho_{\mathbf{c}}}$ whose radicals equal that of $\mathfrak{f}_{\rho_{\mathbf{c}}}$, and $c_{k,s,\rho,N}^{(\mathbf{r})}(n):=\sum_{m}\overline{\rho}_{2\mathbf{c}}(m)m^{-k-1/2-2s}\iota_{m}^{-1}\sum_{i:(\mathbf{Z}/m)^\times}(\rho_{\mathbf{r}}\chi_{m^{\vee}})(i)$ $\times\mathbf{e}(ni/m)$ where $m$ runs the set of all positive integers whose radicals are $\mathfrak{f}_{\rho_{\mathbf{r}}}$. Then there holds the decomposition
\begin{align}
c_{k,s,\rho,N}'(n)=c_{k,s,\rho,N}^{(\mathbf{c})}(n)c_{k,s,\rho,N}^{(\mathbf{r})}(n)c_{k,s,\rho,N}''(n)\hspace{1.5em}(n\in\mathbf{Z}). \label{eqn:dcmp-cprime}
\end{align}

Let $n=0$. Then
\begin{align}
c_{k,s,\rho,N}''(0)=\tfrac{L(2k-1+4s,\overline{\rho}^2)}{L(2k+4s,\overline{\rho}^{2})}=\tfrac{L(2k-1+4s,\overline{\widetilde{\rho^{2}}})}{L(2k+4s,\overline{\widetilde{\rho^{2}}})}\prod_{p|2\mathfrak{f}_{\rho_{\mathbf{r}}}}\tfrac{1-\overline{\widetilde{\rho^{2}}}(p)p^{-2k+1-4s}}{1-\overline{\widetilde{\rho^{2}}}(p)p^{-2k-4s}}.\label{eqn:c2prime0}
\end{align}
If $\rho_{2}$ is complex, then $c_{k,s,\rho,N}^{(2)}(0)$ vanishes since $\overline{\rho}_{2}\chi_{\pm2^{l}}$ is nontrivial for any $l$ and $\sum_{i:(\mathbf{Z}/2^{l})^{\times}}(\overline{\rho}_{2}\chi_{2^{l}})(i)=0\ (l\ge4)$. If $\rho_{\mathbf{c}}$ is nontrivial, then $c_{k,s,\rho,N}'(0)$ vanishes by the same reason.  Hence $c_{k,s,\rho,N}(0)=0$ if $\rho$ is complex. If $\rho$ is real,  then putting $t_{\rho_{2}}=1$ or $2^{-k+1/2-2s}$ according as $\mathfrak{f}_{\rho_{2}}\le4$ or  $\mathfrak{f}_{\rho_{2}}=8$, we have
\begin{align}
c_{k,s,\rho,N}^{(2)}(0)&=(1{+}\rho_{2}({-}1)\sqrt{{-}1})\tfrac{2^{-2k-1-4s}t_{\rho_{2}}}{1-2^{-2k+1-4s}},%c_{k,s,\rho,N}''(0)=\tfrac{\zeta(2k-1+4s)}{\zeta(2k+4s)}\prod_{p|2\mathfrak{f}_{\rho_{\mathbf{r}}}}\tfrac{1-p^{-2k+1-4s}}{1-p^{-2k-4s}},
%\nonumber\\
c_{k,s,\rho,N}^{(\mathbf{r})}(0)=\rho_{2}(\mathfrak{f}_{\rho_{\mathbf{r}}})\iota_{\mathfrak{f}_{\rho_{\mathbf{r}}}}^{-1}\prod_{p|\mathfrak{f}_{\rho_{\mathbf{r}}}}\tfrac{(p-1)p^{-k-1/2-2s}}{1-p^{-2k+1-4s}},\nonumber\\
c_{k,s,\rho,N}'(0)&=\rho_{2}(\mathfrak{f}_{\rho_{\mathbf{r}}})\iota_{\mathfrak{f}_{\rho_{\mathbf{r}}}}^{-1}\tfrac{\zeta(2k-1+4s)}{\zeta(2k+4s)}\tfrac{1-2^{-2k+1-4s}}{1-2^{-2k-4s}}\prod_{p|\mathfrak{f}_{\rho_{\mathbf{r}}}}\tfrac{(p-1)p^{-k-1/2-2s}}{1-p^{-2k-4s}},\nonumber\nonumber\\
c_{k,s,\rho,N}(0)&=(1{+}({-}1)^{k}\sqrt{{-}1})\tfrac{\zeta(2k-1+4s)}{\zeta(2k+4s)}\tfrac{2^{-2k-1-4s}t_{\rho_{2}}}{1-2^{-2k-4s}}\prod_{p|\mathfrak{f}_{\rho_{\mathbf{r}}}}\tfrac{(p-1)p^{-k-1/2-2s}}{1-p^{-2k-4s}}.\nonumber
\end{align}

Let $n\in\mathbf{Z},\ne0$. We put $n_{\mathbf{c}}:=\prod_{p|\mathfrak{f}_{\rho_{\mathbf{c}}}}p^{v_{p}(n)},n_{\mathbf{r}}:=\prod_{p|N,p\nmid2\mathfrak{f}_{\rho_{\mathbf{c}}}}p^{v_{p}(n)},n'=\prod_{p\nmid N}p^{v_{p}(n)}$, so that $n=\mathrm{sgn}(n)2^{v_{2}(n)}n_{\mathbf{c}}n_{\mathbf{r}}n'$. Then there holds an equality
\begin{align*}
c_{k,s,\rho,N}^{(\mathbf{c})}(n)=\iota_{n_{\mathbf{c}}\mathfrak{f}_{\rho_{\mathbf{c}}}}^{-1}\overline{\rho}_{2\mathbf{r}}(n_{\mathbf{c}}\mathfrak{f}_{\rho_{\mathbf{c}}})(\rho_{\mathbf{c}}\chi_{(n_{\mathbf{c}}\mathfrak{f}_{\rho_{\mathbf{c}}})^{\vee}})(n/n_{\mathbf{c}})\tau(\overline{\rho}_{\mathbf{c}}\chi_{(n_{\mathbf{c}}\mathfrak{f}_{\rho_{\mathbf{c}}})^{\vee}})\mathfrak{f}_{\rho_{\mathbf{c}}}^{-k-1/2-2s}n_{\mathbf{c}}^{-k+1/2-2s},
\end{align*}and if $\rho_{2}$ is complex, then $c_{k,s,\rho,N}^{(2)}(n)=2^{-1}(1{+}\sqrt{{-}1})$ $\overline{\rho}_{\mathbf{cr}}(2^{v_{2}(n)}$ $\mathfrak{f}_{\rho_{2}})(\rho_{2}\chi_{2^{v_{2}(n)}\mathfrak{f}_{\rho_{2}}})(n2^{-v_{2}(n)})$ $\{1{-}\chi_{-4}(n2^{-v_{2}(n)})(\overline{\rho}_{2}\chi_{2^{v_{2}(n)}})(1{+}2^{-2}\mathfrak{f}_{\rho_{2}})\}$ $\tau(\overline{\rho}_{2}\chi_{2^{v_{2}(n)}\mathfrak{f}_{{\rho}_{2}}})$ $\mathfrak{f}_{{\rho}_{2}}^{-k-1/2-2s}2^{-v_{2}(n)(k-1/2+2s)}$. 

We put
\begin{align}
\psi_{n}:=(\overline{\rho}\chi_{n})^{\sim}=(\overline{\rho}\chi_{2^{v_{2}(n)}}\chi_{(2^{-v_{2}(n))}|n|)^{\vee}}\chi_{-4}^{(2^{-v_{2}(n))}|n|-\mathrm{sgn}(n))/2})^{\sim}. \label{eqn:psin}
\end{align}
Then $\widetilde{\psi_{n}^{2}}=\overline{\rho}_{\mathbf{c}}^{2}$ if $\rho_{2}$ is real, and $\widetilde{\psi_{n}^{2}}=\overline{\rho}_{2\mathbf{c}}^{2}$ if $\rho_{2}$ is complex. The conductors $\mathfrak{f}_{\rho_{\mathrm{c}}}$ and $\mathfrak{f}_{\rho_{\mathrm{c}}^{2}}$ have the same prime factors.  We define $f_{k,s,\rho}(n,p)$ to be
\begin{align}
&f_{k,s,\rho}(n,p)\nonumber\\
:=&\begin{cases}
\sum\limits_{0\le l\le v_{p}(n)/2}\widetilde{\overline{\rho}^{2}}(p)^{l}p^{-2l(k-1/2+2s)}\\
\hspace{2.5em}-\psi_n(p)p^{-k-2s}\sum\limits_{0\le l\le(v_{p}(n)-2)/2}\widetilde{\overline{\rho}^{2}}(p)^{l}p^{-2l(k-1/2+2s)}&(2| v_{p}(n)),\\
\sum\limits_{0\le l \le(v_{p}(n)-1)/2}\widetilde{\overline{\rho}^{2}}(p)^{l}p^{-2l(k-1/2+2s)}&(2\nmid v_{p}(n)),
\end{cases}\label{deffksrho}
\end{align}
where in the case $2|v_{p}(n)$, the second summation is $0$ if $v_{p}(n)=0$. As a function of $s$, $f_{k,s,\rho}(n,p)$ is holomorphic.

Let $\rho_{2}$ be real.  If  $\{\psi_{n}\}_{2}=\mathbf{1}_{2}$, then we put $f_{k,s,\rho}^{(\mathbf{r})}(n,2):=2^{-1}(1{+}\rho_{2}({-}1)\sqrt{{-}1})\{-(1+\psi_{n}(2)2^{-k-2s})^{-1}+f_{k,s,\rho}(4n,2)\}$ for $v_{2}(n)$ even, and $f_{k,s,\rho}^{(\mathbf{r})}(n,2):=\overline{\rho}_{\mathbf{cr}}(2)2^{-k-1/2-2s}(1+\rho_{2}({-}1)\sqrt{{-}1})\{-(1+\psi_{n}(2)2^{-k-2s})^{-1}+f_{k,s,\rho}(2n,2)\}$ for  $v_{2}(n)$ odd.  If $\{\psi_{n}\}_{2}=\chi_{-4}$, then we put  $f_{k,s,\rho}^{(\mathbf{r})}(n,2):=2^{-1}(1{+}\rho_{2}({-}1)\sqrt{{-}1})\{{-}(1-\widetilde{\overline{\rho}^{2}}(2)2^{-2k-4s})^{-1}+f_{k,s,\rho}(2n,2)\}$ for $v_{2}(n)$ even, and  $f_{k,s,\rho}^{(\mathbf{r})}(n,2):=\overline{\rho}_{\mathbf{cr}}(2)2^{-k-1/2-2s}(1+\rho_{2}({-}1)\sqrt{{-}1})\{-(1-\widetilde{\overline{\rho}^{2}}(2)2^{-2k-4s})^{-1}+f_{k,s,\rho}(n,2)\}$ for $v_{2}(n)$ odd.  If $\{\psi_{n}\}_{2}=\chi_{\pm8}$, then we put  $f_{k,s,\rho}^{(\mathbf{r})}(n,2):=-\overline{\rho}_{\mathbf{cr}}(2)2^{-k-1/2-2s}(1+\rho_{2}({-}1)\sqrt{{-}1})(1-\widetilde{\overline{\rho}^{2}}(2)2^{-2k-4s})^{-1}
$ for $v_{2}(n)=0$, $f_{k,s,\rho}^{(\mathbf{r})}(n,2):=\overline{\rho}_{\mathbf{cr}}(2)2^{-k-1/2-2s}(1{+}\rho_{2}({-}1)\sqrt{{-}1})\{-(1-\widetilde{\overline{\rho}^{2}}(2)2^{-2k-4s})^{-1}+f_{k,s,\rho}(n/2,2)$ for  $v_{2}(n)>0$ even, and $f_{k,s,\rho}^{(\mathbf{r})}(n,2):=2^{-1}(1{+}\rho_{2}({-}1)\sqrt{{-}1})\{-(1-\widetilde{\overline{\rho}^{2}}(2)2^{-2k-4s})^{-1}+f_{k,s,\rho}(n,2)\}$ for  $v_{2}(n)$ odd. Then   $c_{k,s,\rho,N}^{(2)}(n)=(1+\psi_{n}(2)2^{-k-2s})f_{k,s,\rho}^{(\mathbf{r})}(n,2)$ if $\{\psi_{n}\}_{2}=\mathbf{1}_{2}$, and $c_{k,s,\rho,N}^{(2)}(n)=(1-\widetilde{\overline{\rho}^{2}}(2)2^{-2k-4s})f_{k,s,\rho}^{(\mathbf{r})}(n,2)$ if otherwise. The factor $c_{k,s,\rho,N}^{(2)}(n)\ (n\ne0)$ is holomorphic in $s$ at least if $k+2s\ne0$.
%
%-------- If $\rho_{2}=\mathbf{1}_{2}$ or $\chi_{-4}$, then $c_{k,s,\rho,N}^{(2)}(n)$ is equal to $2^{-2}(1{+}\rho_{2}({-}1)\sqrt{{-}1})$ $\{\sum_{l=1}^{v_{2}(n)/2}$ $\overline{\rho}_{\mathbf{c}}(2)^{2l}2^{-2l(k-1/2+2s)}{+}\rho_{2}(-1)\overline{\rho}_{\mathbf{c}}(2)^{v_{2}(n)+2}\chi_{-4}(n2^{-v_{2}(n)})2^{-(v_{2}(n)+2)(k-1/2+2s)}+$\linebreak$(1{+}\rho_{2}(-1)\chi_{-4}(n2^{-v_{2}(n)}))\chi_{8}(n2^{-v_{2}(n)})$ $\overline{\rho}_{\mathbf{cr}}(2)^{v_{2}(n)+3}2^{-(v_{2}(n)+3)(k-1/2+2s)-1/2}\}$ or\linebreak $2^{-2}(1{+}\rho_{2}({-}1)\sqrt{{-}1})\{\sum_{l=1}^{(v_{2}(n)-1)/2}\overline{\rho}_{\mathbf{c}}(2)^{2l}2^{-2l(k-1/2+2s)}-\overline{\rho}_{\mathbf{c}}(2)^{v_{2}(n)+1}2^{-(v_{2}(n)+1)(k-1/2+2s)}\}$ according as $v_{2}(n)$ is even or odd. If $\rho_{2}=\chi_{\pm8}$, then it is $2^{-2}(1{+}\rho_{2}({-}1)\sqrt{{-}1})\{\sum_{l=1}^{(v_{2}(n)-2)/2}$ $\overline{\rho}_{\mathrm{cr}}(2)^{2l+1}2^{-(2l+1)(k-1/2+2s)}-\overline{\rho}_{\mathrm{cr}}(2)^{v_{2}(n)+1}2^{-(v_{2}(n)+1)(k-1/2+2s)}\}$ or $2^{-2}(1{+}\rho_{2}({-}1)\sqrt{{-}1})$ $\{\sum_{l=1}^{(v_{2}(n)-1)/2}\overline{\rho}_{\mathrm{cr}}(2)^{2l+1}2^{-(2l+1)(k-1/2+2s)}{+}\rho_{2}({-}1)\chi_{-4}(n2^{-v_{2}(n)})\overline{\rho}_{\mathrm{cr}}(2)^{v_{2}(n)+2}$ \linebreak $\times2^{-(v_{2}(n)+2)(k-1/2+2s)}{+}(1{+}\rho_{2}({-}1)\chi_{-4}(n2^{-v_{2}(n)}))\chi_{8}(n2^{-v_{2}(n)})\overline{\rho}_{\mathrm{c}}(2)^{v_{2}(n)+3}\times$\linebreak $2^{-(v_{2}(n)+3)(k-1/2+2s)-1/2}\}$ according as $v_{2}(n)$ is even or odd.----------------

For an odd $p|\mathfrak{f}_{\rho_{\mathbf{r}}}$, the $p$-factor of $c_{k,s,\rho,N}^{(\mathbf{r})}(n)$, namely $\sum_{l=1}^{\infty}(\overline{\rho}_{2\mathbf{c}}\chi_{(\mathfrak{f}_{\mathbf{r}}p^{-1})^{\vee}})(p)^{l}$\linebreak$\times p^{-l(k+1/2+2s)}\iota_{p^{l}}^{-1}\sum_{i:(\mathbf{Z}/p^{l})^{\times}}(\chi_{{p^{l-1}}^{\vee}})(i)\mathbf{e}(ni/p^{l})$ is given by $(1+\psi_{n}(p)p^{-k-2s})$\linebreak$\times f_{k,s,\rho}^{(\mathbf{r})}(n,p)$ for $v_{p}(n)$ odd, or $(1-\overline{\rho}(p)^{2}p^{-2k-4s})f_{k,s,\rho}^{(\mathbf{r})}(n,p)$ for $v_{p}(n)$ even where $f_{k,s,\rho}^{(\mathbf{r})}(n,p)$ denotes
\begin{alignat*}{2}
&\iota_{p}\widetilde{\rho\chi_{4p}}(p)p^{k-1/2+2s}\{-(1+\psi_{n}(p)p^{-k-2s})^{-1}+f_{k,s,\rho}(pn,p)\}&\hspace{1em}&(2\nmid v_{p}(n)),\\
&\iota_{p}\widetilde{\rho\chi_{4p}}(p)p^{k-1/2+2s}\{-(1-\widetilde{\overline{\rho}^{2}}(p)p^{-2k-4s})^{-1}+f_{k,s,\rho}(pn,p)\}&&(2|v_{p}(n)).
\end{alignat*}
We have $c_{k,s,\rho,N}^{(\mathbf{r})}(n)=\prod_{p|\mathfrak{f}_{\rho_{\mathbf{r}}},2\nmid v_{p}(n)}(1+\psi_{n}(p)p^{-k-2s})\prod_{p|\mathfrak{f}_{\rho_{\mathbf{r}}},2|v_{p}(n)}(1-\widetilde{\overline{\rho}^{2}}(p)p^{-2k-4s})$ $\times\prod_{p|\mathfrak{f}_{\rho_{\mathbf{r}}}}f_{k,s,\rho}^{(\mathbf{r})}(n,p)$.

For $p\nmid N$, the $p$-factor of $c_{k,s,\rho,N}''(n)$, namely $\sum_{l=0}^{\infty}\overline{\rho}(p)^{l}p^{-l(k+1/2+2s)}\iota_{p^{l}}^{-1}\sum_{i:(\mathbf{Z}/p^{l})^{\times}}$ $\chi_{{p^{l}}^{\vee}}(i)\mathbf{e}(ni/p^{l})$, is equal to $(1+\psi_{n}(p)p^{-k-2s})f_{k,s,\rho}(n,p)$ if $2|v_{p}(n)$, and to $(1-\widetilde{\overline{\rho}^{2}}(p)p^{-k-2s})f_{k,s,\rho}(n,p)$ if $2\nmid v_{p}(n)$. Then we have for $n\ne0$,
\begin{align}
c_{k,s,\rho,N}''(n)=&\tfrac{L(k{+}2s,\psi_{n})}{L(2k{+}4s,\widetilde{\overline{\rho}^{2}})}\prod_{p|2\mathfrak{f}_{\rho_{\mathbf{r}}}}\tfrac{1{-}\psi_{n}(p)p^{-k-2s}}{1{-}\widetilde{\overline{\rho}^{2}}(p)p^{-2k-4s}}\prod_{p\nmid N,p|n}f_{k,s,\rho}(n,p).\label{eqn:c2prime}
\end{align}
Then
\begin{align*}c_{k,s,\rho,N}(n)=c_{k,s,\rho,N}^{(2)}(n)c_{k,s,\rho,N}^{(\mathbf{c})}(n)\tfrac{L(k{+}2s,\psi_{n})}{L(2k{+}4s,\widetilde{\overline{\rho}^{2}})}\prod_{p|\mathfrak{f}_{\rho_{\mathbf{r}}}}f_{k,s,\rho}^{(\mathbf{r})}(n,p)\prod_{p\nmid N,p|n}f_{k,s,\rho}(n,p)
\end{align*}
and in particular if $\rho_{2}$ is real, then 
\begin{align*}
c_{k,s,\rho,N}(n)=c_{k,s,\rho,N}^{(\mathbf{c})}(n)\tfrac{L(k{+}2s,\psi_{n})}{L(2k{+}4s,\widetilde{\overline{\rho}^{2}})}\prod_{p|2\mathfrak{f}_{\rho_{\mathbf{r}}}}f_{k,s,\rho}^{(\mathbf{r})}(n,p)\prod_{p\nmid N, p|n}f_{k,s,\rho}(n,p).
\end{align*}
Thus the Fourier coefficients of (\ref{eqn:eshiw}) is obtained.

\section{Constant terms of Eisenstein series of half integral weight}\label{sect:CTEHIW}
We compute the constant terms of Fourier expansions with respect to $x$ at cusps in $\mathcal{C}_{0}(N)$ of (\ref{eqn:cusps}), of some specific Eisenstein series of half integral weight. They are useful to investigate the functional equations of Eisenstein series.

Let $4|N,\rho\in(\mathbf{Z}/N)^{\ast}$ be as in (\ref{cond:M}) with $M=N$. Then
\begin{alignat*}{2}
E_{k+1/2,\rho,2^{2m}N}(z,s)&=E_{k+1/2,\rho,N}(2^{2m}z,s)&\mbox{\hspace{2em}}&(m\ge0),\\
E_{k+1/2,\rho\chi_{8},2^{2m+1}N}(z,s)&=E_{k+1/2,\rho,N}(2^{2m+1}z,s)&&(m\ge0),\\
E_{k+1/2,\rho\mathbf{1}_{p},p^{2}N}(z,s)&=E_{k+1/2,\rho\chi_{p},pN}(pz,s)&&(p\nmid N),\\
E_{k+1/2,\rho,p^{2}N}(z,s)&=E_{k+1/2,\rho,N}(p^{2}z,s)&&(p|N).
\end{alignat*}
For $\rho'\in(\mathbf{Z}/N')$ satisfying (\ref{cond:M}) with $\rho=\rho'$ and $M=N'$, the above equalities imply that Eisenstein series $E_{k+1/2,\rho',N'}(z,s)$ is written in the form $E_{k+1/2,\rho',N'}(z,s)=E_{k+1/2,\rho,N}(mz,s)$ for some natural number $m$ and for $\rho,N$ so that
\begin{align}
\rho=\widetilde{\rho}\mathbf{1}_{2},\ N=\mathrm{lcm}(4,\mathfrak{f}_{\rho})\mbox{, and }\rho_{2}=\mathbf{1}_{2},\chi_{-4}\mbox{ or }\rho_{2}\mbox{  is complex} \label{eqn:N2}
\end{align}
where $\widetilde{\rho}$ is a primitive character and where $\rho_{2}$ is as in (\ref{eqn:dcmp-rho}).

The constant terms of Eisenstein series (\ref{eqn:eshiwe}),(\ref{eqn:eshiwo}) at cusps are in the form
\begin{align}\label{cte}
c_{0}+\xi_{0}(s)y^{-k-1/2-2s} 
\end{align}
with constants $c_{0}$ and with functions $\xi_{0}(s)$ of $s$ by (\ref{eqn:fex}), where $c_{0}$ and $\xi_{0}(s)$ could be $0$.

\begin{lem}\label{lem:vac0} Let $\rho,\widetilde{\rho},N$ be as in (\ref{eqn:N2}), and let $\rho=\rho_{2}\rho_{\mathbf{c}}\rho_{\mathbf{r}}=\rho_{2\mathbf{c}}\rho_{\mathbf{r}}$ be as in (\ref{eqn:dcmp-rho}). Let $k$ be a nonnegative integer with the same parity as $\rho$.

(i) Suppose that $\rho_{2}$ is real. Then $N=4\mathfrak{f}_{\rho_{\mathbf{c}}}\mathfrak{f}_{\rho_{\mathbf{r}}}$, and $E_{k+1/2,\rho}(z,s)$ vanishes at a cusp $i/(2m)\ ((i,2m)=1)$ for any odd $m$.

(ii) Suppose that $\rho_{2}$ is complex, which is irreducible by our assumption. Then $E_{k+1/2,\rho}(z,s)$ vanishes at a cusp $i/(2^{-1}\mathfrak{f}_{\rho_{2}}m)\ ((i,2\mathfrak{f}_{\rho_{2}}m)=1)$ for any odd $m$.

(iii) Let $M|N$. Suppose that $M$ has a prime divisor $p$ so that $\{\rho\}_{p}$ is complex. Then $E_{k+1/2,\rho}(z,s)$ has the constant term (\ref{cte}) at a cusp $i/M\ ((i,M)=1)$ with $\xi_{0}(s)=0$.
\end{lem}
\begin{proof}
(i) Put $B_{n}=\mbox{\footnotesize$\begin{pmatrix}{-}1{+}2imn&{-}i^{2}n\\4m^{2}n&{-}1{-}2imn\end{pmatrix}$}\in\mathrm{SL}_{2}(\mathbf{Z})$, which stabilizes the cusp $i/(2m)$. It is checked that the automorphy factor of $E_{k+1/2,\rho}(z,s)$ has the value in $\pm\sqrt{-1}$ at $z=i/(2m)$ for $B_{\mathfrak{f}_{\rho_{\mathbf{c}}}\mathfrak{f}_{\rho_{\mathbf{r}}}}\in\Gamma_{0}(N)$ noting that ${-}1{-}2im\mathfrak{f}_{\rho_{\mathbf{c}}}\mathfrak{f}_{\rho_{\mathbf{r}}}\equiv1\pmod{4}$, namely $\iota_{({-}1{-}2im\mathfrak{f}_{\rho_{\mathbf{c}}}\mathfrak{f}_{\rho_{\mathbf{r}}})}=1$, and $\lim_{z\to i/(2m)\atop z\in\mathfrak{H}}(4m^{2}\mathfrak{f}_{\rho_{\mathbf{c}}}\mathfrak{f}_{\rho_{\mathbf{r}}}z{-}1{-}2im\mathfrak{f}_{\rho_{\mathbf{c}}}\mathfrak{f}_{\rho_{\mathbf{r}}})^{1/2}=\sqrt{-1}$. Then the constant term of the Fourier expansion of $E_{k+1/2,\rho}(z,s)$ at the cusp $i/(2m)$ vanishes.

(ii) We note that $16|\mathfrak{f}_{\rho_{2}}$. Put $B_{n}=\mbox{\footnotesize$\begin{pmatrix}{-}1{+}2^{-1}\mathfrak{f}_{\rho_{2}}imn&{-}i^{2}n\\2^{-2}\mathfrak{f}_{\rho_{2}}^{2}m^{2}n&{-}1{-}2^{-1}\mathfrak{f}_{\rho_{2}}imn\end{pmatrix}$}\in\mathrm{SL}_{2}(\mathbf{Z})$, which stabilizes the cusp $i/(2^{-1}\mathfrak{f}_{\rho_{2}}m)$. Take $\mathfrak{f}_{\rho_{\mathbf{c}}}\mathfrak{f}_{\rho_{\mathbf{r}}}(>0)$ as $n$. Then $\lim_{z\to i/(2m)\atop z\in\mathfrak{H}}$ $(4m^{2}nz{-}1{-}2imn)^{k+1/2}=(-1)^{k}\sqrt{-1}$, $\chi_{(2^{-2}\mathfrak{f}_{\rho_{2}}^{2}m^{2}n)}({-}1{-}2^{-1}\mathfrak{f}_{\rho_{2}}imn)=$\linebreak$\chi_{n}({-}1{-}2^{-1}\mathfrak{f}_{\rho_{2}}imn)=\chi_{n}({-}1)=1$, and ${-}1{-}2^{-1}\mathfrak{f}_{\rho_{2}}imn\equiv3\pmod{4}$. Hence the automorphy factor of $E_{k+1/2,\rho}(z,s)$ has the value $(-1)^{k}\rho({-}1{-}2^{-1}\mathfrak{f}_{\rho_{2}}imn)$ at $z=i/(2^{-2}\mathfrak{f}_{\rho_{2}}^{2}m^{2}n)$ for $B_{n}\in\Gamma_{0}(N)$. Then $(-1)^{k}\rho({-}1{-}2^{-1}\mathfrak{f}_{\rho_{2}}imn)=\rho_{2}(1{+}2^{-1}\mathfrak{f}_{\rho_{2}}imn)$ $=-1$. Then the constant term of the Fourier expansion at the cusp $i/(2^{-1}\mathfrak{f}_{\rho_{2}}m)$ vanishes. Though Lemma makes no mention of the case $\rho_{2}=\chi_{\pm8}$, this proof is effective.

(iii) At first we assume that there is a prime divisor $p_{0}$ of $M$ with $1\le v_{p_{0}}(M)<v_{p_{0}}(N)$ so that $\{\rho\}_{p_{0}}$ is complex. When $p_{0}=2$, we may assume  by the assertion (ii), that $v_{2}(M)\le v_{2}(N)-2$. Put $B_{n}=\mbox{\footnotesize$\begin{pmatrix}{-}1{+}2iMn&{-}i^{2}n\\M^{2}n&{-}1{-}iMn\end{pmatrix}$}\in\mathrm{SL}_{2}(\mathbf{Z})$, which stabilizes the cusp $i/M$.  We take $n>0$ so that $N|M^{2}n,\ v_{p}(Mn)\ge v_{p}(N)\ (p|N,p\ne p_{0})$ and that $v_{p_{0}}(Mn)<v_{p_{0}}(N)$ if $p_{0}\ne2$, and $v_{2}(Mn)<v_{2}(N)-1$ if $p_{0}=2$. Then the automorphy factor of $E_{k+1/2,\rho}(z,s)$ does not takes a real value at $i/M$ for $B_{n}\in\Gamma_{0}(N)$, in particular it does not take the value $1$ and hence the constant term of the Fourier expansion at  $i/M$ vanishes.

Now we may assume that $(N,N/M)=1$ and the cusp $i/M$ is $1/M$ (see (\ref{eqn:cusps})). Let $A_{1/M}=\left({\, 1\ \ b_{0}\atop M\ d_{0}}\right)\in\mathrm{SL}_{2}(\mathbf{Z})$. Put $u(c,d):=(cz+d)^{k-1/2}|cz+d|^{-2s}$ for short. Then
\begin{align}
&E_{k+1/2,\chi}(z,s)|_{A_{i/M}}\nonumber\\
=&u(M,d_{0})+\sum_{N|c>0,(c,d)=1\atop c+dM>0}(\overline{\rho}\chi_{c})(d)\iota_{d}u(c+dM,cb_{0}+dd_{0})\nonumber\\
&-(-1)^{k}\sqrt{-1}\sum_{N|c>0,(c,d)=1\atop c+dM<0}(\overline{\rho}\chi_{c})(d)\iota_{d}u((-(c+dM),-cb_{0}-dd_{0})\nonumber\\
=&\sum_{N|c>0,(c,d)=1\atop c+dM>0}(\overline{\rho}\chi_{c})(d)\iota_{d}u(c+dM,cb_{0}+dd_{0}).\label{eqn:cfc}
\end{align}
We fix $c+dM>0$. Then (\ref{eqn:cfc}) has the partial sum $\sum_{n=-\infty}^{\infty}\overline{\rho}(d-\tfrac{N}{M}n)\chi_{c+Nn}(d-\tfrac{N}{M}n)$ $\times\iota_{d-\tfrac{N}{M}n}u(c+dM,cb_{0}+dd_{0}-\frac{N}{M}n)$, whose constant term containing $y^{-k-1/2-2s}$ is equal to
\begin{align*}
&2^{-1}(1+\sqrt{-1})\sum_{n=0}^{M}\{(\overline{\rho}\chi_{c+Md})(d-\tfrac{N}{M}n)-(\overline{\rho}\chi_{(c+Md)}\chi_{-4})(d-\tfrac{N}{M}n)\sqrt{-1}\}\\
&\times w(N^{-1}(c+dM),k+1/2,s).
\end{align*}
This is $0$, and hence the Fourier expansion of (\ref{eqn:cfc}) does not have the constant term containing $y^{-k-1/2-2s}$.
\end{proof}

The nonzero constant term of the Fourier expansion at  each cusp , of the Eisenstein series (\ref{eqn:eshiw}) is obtained similarly as in the preceding section. We state them as a lemma.
\begin{lem}\label{lem:vac} Let $\rho,\widetilde{\rho},N$ be as in (\ref{eqn:N2}), and let $\rho=\rho_{2}\rho_{\mathbf{c}}\rho_{\mathbf{r}}=\rho_{2\mathbf{c}}\rho_{\mathbf{r}}$ be as in (\ref{eqn:dcmp-rho}). Let $k\in\mathbf{Z},\ge0$  be so that $k$ and $\rho$ have the same parity. Put
\begin{align}\hspace*{-.4em}
U_{k+1/2,\rho}(s):=\,&\mathbf{e}({-}\tfrac{k{+}1/2}{4})2^{{-}k{-}1/2{-}2s}(\mathfrak{f}_{\rho_{\mathbf{c}}}\mathfrak{f}_{\rho_{\mathbf{r}}})^{-1}\pi\tfrac{\Gamma(k{-}1/2{+}2s)}{\Gamma(s)\Gamma(k+1/2+s)}\tfrac{L(2k-1+4s,\overline{\rho}^{2}\mathbf{1}_{2})}{L(2k+4s,\overline{\rho}^{2}\mathbf{1}_{2})},\label{eqn:cffe}
\end{align}
and put for $R|4\mathfrak{f}_{\rho_{\mathbf{r}}}$,
\begin{align}
U_{k+1/2,\rho,R}(s):=R^{-k+1/2-2s}\prod_{p|R}\{(p-1)(1-\overline{\rho}_{\mathbf{c}}(p)^{2}p^{-2k+1-4s})^{-1}\}.\label{eqn:cffe2}
\end{align}

(i) Suppose that $\rho_{2}$ is real. Then $N=4\mathfrak{f}_{\rho_{\mathbf{c}}}\mathfrak{f}_{\rho_{\mathbf{r}}}$, and $E_{k+1/2,\rho}(z,s)\in$ $\mathbf{M}_{k{+}1/2{+}s,s}(N,\rho)$ has $0$ as the constant terms of the Fourier expansion with respect to $x$ at cusps in $\mathcal{C}_{0}(N)$ except $1/N,1/P,1/(4P)$ with $P|\mathfrak{f}_{\rho_{\mathbf{r}}}$. The constant term  of $E_{k+1/2,\rho}(z,s)$ at a cusp $1/P$  is $U_{k+1/2,\rho}(s)U_{k+1/2,\rho,P}(s)y^{-k+1/2-2s}$, and the constant term at a cusp $1/(4P)$ is $\{1{+}\rho_{2}(-1)\chi_{-4}(P)\sqrt{-1}\}U_{k+1/2,\rho}(s) U_{k+1/2,\rho,4P}(s)$ $y^{-k+1/2-2s}$ where there is the additional term $1$ if $\mathfrak{f}_{\rho_{\mathbf{c}}}=1$ and $P=\mathfrak{f}_{\rho_{\mathbf{r}}}$. If $\mathfrak{f}_{\rho_{\mathbf{c}}}>1$, then the constant term at a cusp $1/N$ is $1$.

(ii) Suppose that $\rho_{2}$ is complex. Then $N=\mathfrak{f}_{\rho_{2\mathbf{c}}}\mathfrak{f}_{\rho_{\mathbf{r}}}$, and $E_{k+1/2,\rho}(z,s)$ has $0$ as the constant terms at cusps in $\mathcal{C}_{0}(N)$ except $1/N,1/P$ with $P|\mathfrak{f}_{\rho_{\mathbf{r}}}$. The constant term at a cusp $1/N$ of $E_{k+1/2,\rho}(z,s)$ is $1$, and the constant term at a cusp $1/P$ is $2^{2}\mathfrak{f}_{\rho_{2}}^{-1}U_{k+1/2,\rho}(s)U_{k+1/2,\rho,P}(s)y^{-k+1/2-2s}$ .
\end{lem}

For $\rho,N$ as in (\ref{eqn:N2}), the Eisenstein series $E_{k+1/2}^{\rho}(z,s)=\sum_{c>0,(c,N)=1 \atop(c,d)=1}\rho(c)\chi_{c^{\vee}}(d)\iota_{c}^{-1}$ $\times(cz+d)^{-k-1/2}|cz+d|^{-2s}$ has the Fourier expansion
\begin{align}
\sum_{n=-\infty}^{\infty}c_{k,s,\overline{\rho},N}''(n)w_{n}(y,k+1/2,s)\mathbf{e}(nx) \label{feoe2}
\end{align}
where $c_{k,s,\rho,N}''(0)$ is as in (\ref{eqn:c2prime0}), and where $c_{k,s,\rho,N}''(n)$ is as in (\ref{eqn:c2prime}) for $n\ne0$. We have $E_{k+1/2,\rho}(z,s)|_{\mbox{\tiny$\left(\begin{array}{@{}c@{\,}c@{}}0&-1/N\\1&0\end{array}\right)$}}=z^{-k-1/2}|z|^{-2s}E_{k+1/2,\rho}(-\tfrac{1}{Nz},s)=E_{k+1/2}^{\overline{\rho}}(z,s)$, and $E_{k+1/2}^{\overline{\rho}}(z,s)|_{\mbox{\tiny$\left(\begin{array}{@{}c@{\,}c@{}}0&-1/N\\1&0\end{array}\right)$}}=(-1)^{k-1}\sqrt{-1}E_{k+1/2,\rho}(z,s)$.

\begin{lem}\label{lem:vac2}  Let $\rho,N,k$ be as in Lemma \ref{lem:vac}. 

(i) Let $P|\mathfrak{f}_{\rho_{\mathbf{r}}}$ and let $\varepsilon\in\{\pm 1\}$ be so that $\rho=\rho_{2\mathbf{c}}\chi_{P^{\vee}}\chi_{(\varepsilon4\mathfrak{f}_{\rho_{\mathbf{r}}}/P)}$, namely $\varepsilon=\chi_{{-}4}(\mathfrak{f}_{\rho_{\mathbf{r}}}/P)$. Then the constant term of $E_{k+1/2,\chi_{P^{\vee}}}^{\rho_{2\mathbf{c}}\chi_{(\varepsilon4\mathfrak{f}_{\rho_{\mathbf{r}}}/P)}}(z,s)\in\mathbf{M}_{k+1/2+s,s}(N,\rho)$ at a cusp $1/P$ is in the form (\ref{cte}) with $c_{0}=(-1)^{k-1}\sqrt{-1}\chi_{-4}(P)\iota_{P}$, and the constant terms at other cusps in $\mathcal{C}_{0}(N)$ are in the form (\ref{cte}) with $c_{0}=0$. The the constant terms (\ref{cte}) with $\xi_{0}(s)\ne0$ possibly appear only at cusps $i/M\in\mathcal{C}_{0}(N)\ ((i,M)=1)$ where $\mathfrak{f}_{\rho_{\mathbf{c}}}|M$ if $\rho_{2}$ is real, and $\mathfrak{f}_{\rho_{2\mathbf{c}}}|M$ if $\rho_{2}$ is complex.

(ii) Suppose that $\rho_{2}=\mathbf{1}_{2},\chi_{-4}$. Let $\rho=\rho_{2\mathbf{c}}\chi_{\varepsilon4P}\chi_{(\mathfrak{f}_{\rho_{\mathbf{r}}}/P)^{\vee}}$, namely $\varepsilon=\chi_{-4}(P)$. Then the constant term of $E_{k+1/2,\rho_{2}\chi_{\varepsilon 4P}}^{\rho_{\mathbf{c}}\chi_{(\mathfrak{f}_{\rho_{\mathbf{r}}}/P)^{\vee}}}(z,s)\in\mathbf{M}_{k+1/2+s,s}(N,\rho)$ at a cusp $1/(4P)$ is in the form (\ref{cte}) with $c_{0}=(-1)^{k}\rho_{2}(-1)\chi_{-4}(P)$, and the constant terms at other cusps in $\mathcal{C}_{0}(N)$ are in the form (\ref{cte}) with $c_{0}=0$. The the constant terms (\ref{cte}) with $\xi_{0}(s)\ne0$ possibly appear only at cusps $i/M\in\mathcal{C}_{0}(N)\ ((i,M)=1)$ with $\mathfrak{f}_{\rho_{\mathbf{c}}}|M$.
\end{lem}

\begin{proof} (i) Let $b_{0},d_{0}$ be so that $A_{1/P}=\left(\,1\ \,b_{0}\atop P\ \,d_{0}\right)\in\mathrm{SL}_{2}(\mathbf{Z})$. Put $Q:=\mathfrak{f}_{\rho_{\mathbf{c}}}/P$. Then 
\begin{align}
&E_{k+1/2,\chi_{P^{\vee}}}^{\rho_{2\mathbf{c}}\chi_{\varepsilon4Q}}(z,s)|_{A_{1/P}}=(Pz+d_{0})^{-k-1/2}|Pz+d_{0}|^{-2s}E_{k+1/2,\chi_{P^{\vee}}}^{\rho_{2\mathbf{c}}\chi_{\varepsilon4Q}}(\tfrac{z+b_{0}}{Pz+d_{0}},s)\nonumber\\
=&(-1)^{k-1}\iota_{-P}+\sum_{P|c\in\mathbf{Z},(c/P,N/P)=1\atop{d\in\mathbf{Z},(c,d)=1\atop c+dP>0}}\chi_{P^{\vee}}(d)(\rho_{2\mathbf{c}}\chi_{\varepsilon4Q})(c/P)\left(\tfrac{d}{|c|}\right)\iota_{c}^{-1}\nonumber\\
&\hspace{5em}\times((c+dP)z+cb_{0}+dd_{0})^{-k-1/2}|(c+dP)z+cb_{0}+dd_{0}|^{-2s},\label{esc}
\end{align}
which implies that $c_{0}=(-1)^{k-1}\iota_{-P}$ in (\ref{cte}) at the cusp $1/P$.  In the series \linebreak$E_{k+1/2,\chi_{P^{\vee}}}^{\rho_{2\mathbf{c}}\chi_{\varepsilon4Q}}(z,s)|_{A_{i/M}}$ corresponding to (\ref{esc}) at all other cusps $i/M\in\mathcal{C}_{0}(N)$, the coefficients $c$ of $z$ in $(cz+d)^{-k-1/2}$ are always nonzero, and hence $c_{0}=0$.

The similar argument as in the proof of Lemma \ref{lem:vac0} (iii) shows that $\xi_{0}(s)$ in the constant term (\ref{cte}) of the Fourier expansion at $i/M$ is $0$ unless $M$ does not satisfy the condition of the lemma.

The assertion (ii) is proved similarly.
\end{proof}

\section{Theta series as Eisenstein series}
In this section, we mainly consider Eisenstein series of weight less than $3/2$ with character $\rho$. Their relations between theta series are given when $N=4,\rho=\chi_{-4}^{k}$.
\begin{prop} Assume that  $k\in\mathbf{Z},\ge1,\,\rho\in(\mathbf{Z}/N)^{\ast}$ have the same parity. If $k\ge2$, or if $k=1$ and $\rho$ is a complex character, then $E_{k+1/2,\rho,N}(z,0)$ is in $\mathbf{M}_{k+1/2}(N,\rho)$.
\end{prop}
\begin{proof} If $k\ge2$, then the series $E_{k+1/2,\rho,N}(z,0)$ converges absolutely and uniformly on any compact subset of $\mathfrak{H}$, and hence it is holomorphic and it is in $\mathbf{M}_{k+1/2}(N,\rho)$.

Let $k=1$, and  $\rho$ be complex.  We may assume that $\rho,N$ satisfy (\ref{eqn:N}), since for  $\rho'\in(\mathbf{Z}/N')^{\ast}$ satisfying (\ref{cond:M}) with $\rho=\rho'$ and $M=N'$, $E_{k+1/2,\rho',N'}(z,0)$ is written as $E_{k+1/2,\rho,N}(mz,0)$  for some $m\in\mathbf{N}$ as stated in the beginning of Section \ref{sect:CTEHIW}. Then $c_{1,s,\rho,N}(0)=0$, and hence the constant term of the Fourier expansion (\ref{feoe}) is $1$. Since $\rho$ is complex, the character $\psi_{n}$ of (\ref{eqn:psin}) is nontrivial for any $n\in\mathbf{Z},\ne0$. Then $L(1+2s,\psi_{n})$ appearing in (\ref{eqn:c2prime}), is finite at $s=0$, and hence $c_{1,s,\rho,N}(n)$ is finite for all $n\in\mathbf{Z}$. As $s\longrightarrow0$, $w_{n}(y,k+1/2,s)$ tends to $0$ for $n\le0$, the expansion (\ref{feoe}) gives a holomorphic function in $z$ at $s=0$.
\end{proof}

In the rest of this section, we consider the Eisenstein series with  $N$ of (\ref{eqn:N}) exclusively in the case that $\rho$ is real. Then $N=\mathfrak{e}_{\rho}'$. Though our main objective here is the case $k\le1$, we do not restrict $k$ to be $\le1$.  When $k\le1$, $E_{k+1/2,\rho}(z,s)$ does not give a holomorphic Eisenstein series at $s=0$. Let us put
\begin{align*}
\mathscr{E}_{k+1/2,\rho}(z,s):=E_{k+1/2,\rho}(z,s)-c_{k,0,\rho,N}^{(2)}(0)c_{k,0,\rho,N}^{(\mathbf{r})}(0)E_{k+1/2}^{\rho}(z,s). 
\end{align*}
Then
\begin{align*}
&\mathscr{E}_{k+1/2,\rho}(z,s)\\
=&1+\sum_{n=-\infty}^{\infty}\{c_{k,s,\rho,N}^{(2)}(n)c_{k,s,\rho,N}^{(\mathbf{r})}(n)-c_{k,0,\rho,N}^{(2)}(0)c_{k,0,\rho,N}^{(\mathbf{r})}(0)\}c_{k,s,\rho,N}''(n)w_{n}(y,k,s)\mathbf{e}(nx).
\end{align*}

Let $k=1$.  If $n$ is so that the character $\psi_{n}$ of (\ref{eqn:psin}) is not trivial, then $c_{1,s,\rho,N}''(n)$, $c_{1,s,\rho,N}(n)$ are finite at $s=0$ by (\ref{eqn:c2prime}) , (\ref{eqn:dcmp-cprime}) and (\ref{eqn:dcmp-c}).  The terms $c_{1,s,\rho,N}''(n),c_{1,s,\rho,N}(n)$ have a pole at $s=0$ only for $n$ in the form $n=-\mathfrak{f}_{\rho}m^{2}\ (m\in\mathbf{Z})$. The pole is of order one. It is check that $c_{1,0,\rho,N}^{(2)}(0)=c_{1,0,\rho,N}^{(2)}(-\mathfrak{f}_{\rho}m^{2})$ for any $m\in\mathbf{Z}$, and that $f_{1,0,\rho,N}^{(\mathbf{r})}(-\mathfrak{f}_{\rho}m^{2},p)=(\rho_{2}\chi_{(\mathfrak{f}_{\mathbf{r}}/p)^{\vee}})(p)\iota_{p}^{-1}p^{-1/2}=f_{1,0,\rho,N}^{(\mathbf{r})}(0,p)$ for $p|\mathfrak{f}_{\rho_{\mathbf{r}}}$. Hence $c_{1,0,\rho,N}^{(\mathbf{r})}(-\mathfrak{f}_{\rho}m^{2})=c_{1,0,\rho,N}^{(\mathbf{r})}(0)$. On the other hand, $c_{1,s,\rho,N}^{(2)}(n)c_{1,s,\rho,N}^{(\mathbf{r})}(n)-c_{1,0,\rho,N}^{(2)}(0)c_{1,0,\rho,N}^{(\mathbf{r})}(0)$ has zero at $s=0$ for such $n$.  Thus $\mathscr{E}_{3/2,\rho}(z,0)$ is holomorphic since $w_{-m}(y,3/2,0)=0$ for $m\ge0$, and gives a modular form in $\mathbf{M}_{3/2}(N,\rho)$. The computation of the case $k=1$ and $\rho$ is real, is found in Pei \cite{Pei}. 

We take $4$ as $N$, and $\chi_{-4}^{k}$ as $\rho$. Then $\rho_{\mathbf{r}}=\rho_{\mathbf{c}}=\mathbf{1}$. From Section \ref{sect:ESHIW} and from (\ref{feoe2}), we have 
\begin{align}
&E_{k+1/2,\chi_{-4}^{k}}(z,s)\nonumber\\
=&1+\tfrac{(-1)^{k(k+1)/2}2^{-k+1-2s}}{2^{2k+4s}-1}\tfrac{\pi\Gamma(k-1/2+2s)}{\Gamma(s)\Gamma(k+1/2+s)}\tfrac{\zeta(2k-1+4s)}{\zeta(2k+4s)}y^{-k+1/2-2s}\nonumber\\
&+\sum_{n\ne0}\tfrac{L(k{+}2s,\widetilde{\chi}_{(-1)^{k}n})}{\zeta(2k{+}4s)}f_{k,s,\rho}^{(\mathbf{r})}(n,2)\prod_{2\ne p|n}f_{k,s,\rho}(n,p)w_{n}(y,k+1/2,s)\mathbf{e}(nx),\label{eqn:feehiw1}\\
&E_{k+1/2}^{\chi_{-4}^{k}}(z,s)\nonumber\\
=&\tfrac{({-}\sqrt{{-}1})^{k}(1{-}\sqrt{{-}1})(2^{2k-1+4s}-1)}{2^{k-2+2s}(2^{2k+4s}-1)}\tfrac{\pi\Gamma(k-1/2+2s)}{\Gamma(s)\Gamma(k+1/2+s)}\tfrac{\zeta(2k-1+4s)}{\zeta(2k+4s)}y^{-k+1/2-2s}+\sum_{n\ne0}\tfrac{L(k{+}2s,\widetilde{\chi}_{(-1)^{k}n})}{\zeta(2k{+}4s)}\nonumber\\
&\hspace{4em}\times\tfrac{1{-}\widetilde{\chi}_{(-1)^{k}n}(2)2^{-k-2s}}{1{-}2^{-2k-4s}}\prod_{2\ne p|n}f_{k,s,\rho}(n,p)w_{n}(y,k+1/2,s)\mathbf{e}(nx),\label{eqn:feehiw2}
\end{align}
where $f_{k,s,\rho}$ is as in (\ref{deffksrho}) and $f_{k,s,\rho}^{(\mathbf{r})}$ is defined below (\ref{deffksrho}). Equations $E_{k+1/2,\chi_{-4}^{k}}(z,s)|_{\left(\,0\ -1\atop 1\ \ 0\right)}=2^{-2k-1-4s}E_{k+1/2}^{\chi_{-4}^{k}}(z/4,s),\ E_{k+1/2}^{\chi_{-4}^{k}}(z,s)|_{\left(\,0\ -1\atop 1\ \ 0\right)}=(-1)^{k-1}\sqrt{-1}E_{k+1/2,\chi_{-4}^{k}}(z/4,s)$ hold. By (\ref{eqn:feehiw1}), a simple computation gives the following lemma.

\begin{lem} Let $\alpha_{k}(n,2):=-(1+\widetilde{\chi}_{n}(2)2^{-k})^{-1}+(1-2^{-2k+1})^{-1}\{1-\widetilde{\chi}_{n}(2)2^{-k}+2^{(v_{2}(n)/2+1)(-2k+1)-k}(1-\widetilde{\chi}_{n}(2)2^{-k+1})\}$, or $-(1-2^{-2k})^{-1}+(1-2^{-2k+1})^{-1}$\linebreak $\times (1-2^{(v_{2}(n)/2+1)(-2k+1)})$, or $-(1-2^{-2k})^{-1}+(1-2^{-2k+1})^{-1}(1-2^{(v_{2}(n)-1)(-2k+1)/2})$ according as $2|v_{2}(n)$ and $2^{-v_{2}(n)}n\equiv(-1)^{k}\pmod{4}$, or $2|v_{2}(n)$ and $2^{-v_{2}(n)}n\equiv(-1)^{k}3\pmod{4}$, or $2\nmid v_{2}(n)$. Then we have for $k\ge1$
\begin{align}
E_{k+1/2,\chi_{-4}^{k}}(z,0)=&1+\tfrac{(-1)^{k}2^{2k-1}\pi^{2k+1/2}}{(k-1)!\Gamma(k+1/2)\zeta(2k)}\sum_{n>0}\alpha_{k}(n,2) L(1{-}k,\widetilde{\chi}_{(-1)^{k}n})\nonumber\\
&\hspace{3em}\times(\mathfrak{f}_{\chi_{(-1)^{k}n}}^{-1}n)^{k-1/2}\prod_{2\ne p|n}f_{k,0,\chi_{-4}^{k}}(n,p)\mathbf{e}(nz),\label{eqn:feehiw1s0}
\end{align}
where there is the additional term $-\pi^{-1}y^{-1/2}-2\pi^{-1/2}\sum_{m=1}^{\infty}m(4\pi m^{2}y)^{-3/4}$\linebreak$\times W_{-3/4,1/4}(4\pi m^{2}y)\mathbf{e}(-m^{2}x)$ when $k=1$. For $k=0$, we have
\begin{align}
&E_{1/2,\mathbf{1}_{2}}(z,0)\nonumber\\
=&-\tfrac{\pi}{6\log 2}y^{1/2}+1+2\sum_{m=1}^{\infty}(1{-}2^{-v_{2}(m)-2})\mathbf{e}(m^{2}z)\nonumber\\
&+(2\log2)^{-1}\sum_{n>0}(1{-}\widetilde{\chi}_{n}(2))L(1,\widetilde{\chi}_{n})(\mathfrak{f}_{\chi_{n}}n^{-1})^{1/2}\prod_{2\ne p|n}f_{0,0,\mathbf{1}_{2}}(n,p)\mathbf{e}(nz)\nonumber\\
&+(2\log2)^{-1}\pi^{1/2}\sum_{n<0}(1{-}\widetilde{\chi}_{n}(2))L(0,\widetilde{\chi}_{n})|n|^{-1/2}\prod_{2\ne p|n}f_{0,0,\mathbf{1}_{2}}(n,p)\nonumber\\
&\mbox{\hspace{11em}}\times(4\pi|n|y)^{-1/4}W_{-1/4,1/4}(4\pi|n|y)\mathbf{e}(nx).\label{eqn:feehiw2s0}
\end{align}
\end{lem}
We describe theta series $\theta(z),\theta(z)^{3},\theta(z)^{5},\theta(z)^{7}$ as Eisenstein series. We have $\mathscr{E}_{k+1/2,\chi_{-4}^{k}}(z,s)=1+\sum_{n=-\infty}^{\infty}\{c_{k,s,\chi_{-4}^{k},4}^{(2)}(n)-c_{k,0,\chi_{-4}^{k},4}^{(2)}(0)\}c_{k,s,\chi_{-4}^{k},4}''(n)w_{n}(y,k,s)$ $\times\mathbf{e}(nx)$. Let  $k=0$.  Then the direct computation shows that $\mathscr{E}_{1/2,\mathbf{1}_{2}}(z,0)$ is equal to $\theta(z)$. As is shown,  $\mathscr{E}_{k+1/2,\chi_{-4}^{k}}(z,s)$ is a holomorphic modular form for $k\ge1$, and it is in $\mathbf{M}_{k+1/2}(4,\chi_{-4}^{k})$ with the value $1$ at $\sqrt{{-}1}\infty$ and with the value $2^{-k-1/2}\mathbf{e}(\tfrac{6k-1}{8})$ at $1$. Since the subspaces of $\mathbf{M}_{k+1/2}(4,\chi_{-4}^{k})$ consisting of the such modular forms are one dimensional for $k=0,1,2,3$ and since $\theta(z)^{2k+1}$'s satisfy the same condition, we have the following proposition.

\begin{prop}(1) Let $N,\rho$ be as in (\ref{eqn:N}). Then $\mathscr{E}_{k+1/2,\rho,N}(z,0)$ is a holomorphic modular form in $\mathbf{M}_{k+1/2}(N,\rho)$ for $k\ge1$ with the same parity as $\rho$.

(2) If $k=0,1,2,3$, then
\begin{align}
\theta(z)^{2k+1}=\mathscr{E}_{k+1/2,\chi_{-4}^{k}}(z,0). \label{eqn:t-e}
\end{align}
\end{prop}

\begin{remark} The equality (\ref{eqn:t-e}) holds for any $k\ge0$ up to cusp forms. As for even powers of $\theta(z)$, the following equalities holds up to cusp forms; $\theta(z)^{2}=E_{1,\chi_{-4}}(z,0)$, and $\theta(z)^{2k}=E_{k,\chi_{-4}}(z,0)+2^{-k}\sqrt{{-}1}^{k}E_{k}^{\chi_{-4}}(z,0)$ for $k>1$ odd. For $k\ge2$, even, $\theta(z)^{2k}=\{(-1)^{k/2}2^{-k}E_{k}(z,0)-(-1)^{k/2}2^{-k}E_{k,\mathbf{1}_{2},2}(z,0)+E_{k,\mathbf{1}_{2},4}(z,0)\}$.
%$\theta(z)^{4}=-\tfrac{1}{3}\{E_{2}(z,0)-4E_{2}(4z,0)\},\theta(z)^{6}=E_{3,\chi_{-4}}(z,0)-2^{-3}\sqrt{-1}$ $\times E_{3}^{\chi_{-4}}(z,0)$.
These equalities are exact if the weight is less than or equal to $4$.
\end{remark}
\begin{cor} All the holomorphic Eisenstein series of weight $1/2$ on $\Gamma_{1}(N)$ are linear combinations of $\sum_{i:(\mathbf{Z}/\mathfrak{f}_{\omega})^{\times}}\omega(i)\mathscr{E}_{1/2,\mathbf{1}_{2},4}(tz+i/\mathfrak{f}_{\omega},0)$ where $\omega$ are primitive complex characters whose squares are also primitive, and where $t$ are natural numbers with $4(\mathfrak{f}_{\omega^{2}})^{2}t|N$.
\end{cor}
\begin{proof} Let $\chi$ be a square of $\omega$. If both of  $\omega$ and $\chi$ are primitive, then an equality $\mathfrak{f}_{\chi}=\mathfrak{f}_{\omega}/(2,\mathfrak{f}_{\omega})$ holds. By Serre and Stark \cite{Serre-Stark} Corollary 1 to Theorem A, the subspace in $\mathbf{M}_{1/2}(\Gamma_{1}(N))$ consisting of Eisenstein series is spanned by $\theta_{\chi}(tz):=\sum_{n=-\infty}^{\infty}\chi(n)\mathbf{e}(nz)$ with primitive characters $\chi$ which are squares, and with $4\mathfrak{f}_{\chi}^{2}t|N$.  From (\ref{eqn:t-e}) we see that $\theta_{\chi}(z)=\tau(\overline{\omega})^{-1}\sum_{i:(\mathbf{Z}/\mathfrak{f}_{\omega})^{\times}}\overline{\omega}(i)\theta(z+i/\mathfrak{f}_{\omega})=\tau(\overline{\omega})^{-1}\sum_{i:(\mathbf{Z}/\mathfrak{f}_{\omega})^{\times}}$ $\overline{\omega}(i)\mathscr{E}_{1/2,\mathbf{1}_{2},4}(z+i/\mathfrak{f}_{\omega},0)$, which shows our assertion.
\end{proof}

\section{Rankin Selberg method}
We show that the functions of $s$ obtained from the constant terms of some real analytic modular forms in $\mathcal{M}_{l,l'}(N,\rho)$ with $l,l'\in\tfrac{1}{2}\mathbf{Z},\ge0$ with $\rho\in(\mathbf{Z}/N)^{\ast}$ together with a real analytic Eisenstein series, have analytic continuation to the whole complex $s$-plane by using the unfolding trick. From this, the analytic continuation of $L$-function (\ref{eqn:lseries})  is proved in the case $l-l'\in\mathbf{Z}$. Further some special value of the $L$-function is written in terms of the scalar product, and the functional equation between the $L$-function (\ref{eqn:lseries}) and the $L$-function defined similarly is derived.

Let $l\ge l'\ge0$.  Let $N\in\mathbf{N}$ be so that $4|N$ if at least one of $l,l'$ is not integral. Put $k:=l-l'$ if $l-l'\in\mathbf{Z}$, and put $k:=l-l'-1/2$ if otherwise, and let $\rho$ be a Dirichlet character modulo $N$ with the same parity as $k$. We assume that $\rho,N$ satisfy (\ref{cond:M}) with $M=N$ if $l-l'$ is not integral. We consider a real analytic modular form $F(z)$ in $\mathcal{M}_{l,l'}(\Gamma_{1}(N))$ which satisfies $F(Az)=\overline{\rho}(d)(cz+d)^{l}(\overline{cz+d})^{l-k}F(z)$ for $A=\left(a\ \,b\atop c\ \,d\right)\in\Gamma_{0}(N)$ if $l-l'\in\mathbf{Z}$, or $F(Az)=\overline{\rho}(d)(cz+d)^{l}(\overline{cz+d})^{l-k}\overline{j}(A,z)^{-1}F(z)$ if $l-l'\not\in\mathbf{Z}$, $j(A,z)$ being the automorphy factor of $\theta(z)$ defined in the introduction.  We assume that $F(z)$ has the Fourier expansion in the form
\begin{align}\label{eqn:fearc}
F|_{A_{r}}(z)=P_{F}^{(r)}(y)+\sum_{n=-\infty}^{\infty}u_{n}^{(r)}(y)\mathbf{e}(nx/w^{(r)})\mbox{ \ with \ }P_{F}^{(r)}(y)=\sum_{j} a_{j}^{(r)}y^{q_{j}^{(r)}}
\end{align}
at each cusp $r\in\mathcal{C}_{0}(N)$, $A_{r}$ being as in (\ref{defAr}), where the summation $P_{F}^{(r)}(y)$ is a finite sum with $a_{j}^{(r)},q_{j}^{(r)}\in\mathbf{C}$ and all  $u_{n}^{(r)}(y)$ are rapidly decreasing as $y\longrightarrow\infty$. We drop the notation $(r)$ from $P_{F}^{(r)},a_{j}^{(r)},q_{j}^{(r)},u_{n}^{(r)}$ when $r=1/N$ or equivalently $r=\sqrt{-1}\infty$. We define the Rankin-Selberg transform 
\begin{align}
R(F,s):=\int_{0}^{\infty}y^{l+s-2}u_{0}(y)dy,\label{eqn:rst}
\end{align}
following Zagier \cite{Zagier}. In Theorem \ref{thm:analytic continuation}, we show that the integral converges for $s$ with sufficiently large $\Re s$.
\begin{figure}[htbp]
\centering
\includegraphics{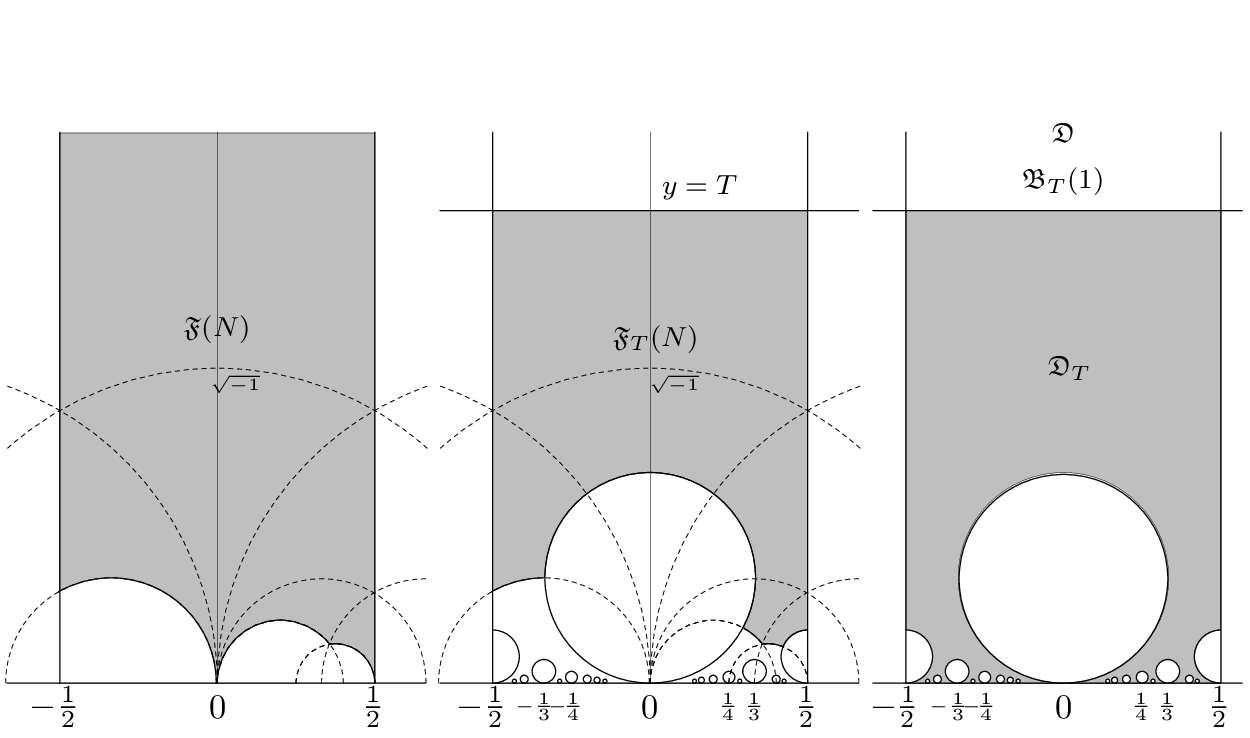}
\caption{ }
\label{fig:fdomain}
\end{figure}

We denote by $Q(i/M)$, the set of rational numbers in $(-1/2,1/2]$ equivalent to $i/M$ under $\Gamma_{0}(N)$.  
\begin{lem}\label{lem:ub}
Let $q$ be the maximum of $0,\Re q_{j}^{(r)}\ (r\in\mathcal{C}_{0}(N))$. Then  $|F(z)|=$\linebreak$O(y^{-q-(l+l')/2})$ as $y\longrightarrow+0$.
\end{lem}
\begin{proof}
 The function $\psi(z):=y^{(l+l')/2}\mathrm{tr}_{\Gamma_{0}(N)/\mathrm{SL}_{2}(\mathbf{Z})}(|F(z)|)$ is $\mathrm{SL}_{2}(\mathbf{Z})$ invariant function. Then $y^{-q-(l+l')/2}\psi(z)$ is bounded on the fundamental domain $\mathfrak{F}$ of $\mathrm{SL}_{2}(\mathbf{Z})$, and hence
$$(\max_{A\in\mathrm{SL}_{2}(\mathbf{Z})}\Im(Az))^{-q-(l+l')/2}\psi(z)$$
is bounded on $\mathfrak{H}$. Since $\max_{A\in\mathrm{SL}_{2}(\mathbf{Z})}\Im(Az)\le \max\{y,y^{-1}\}$, we have $|F(z)|=O(y^{-q-(l+l')/2})$ as $y\longrightarrow+0$. 
\end{proof}
We fix the notation for the rest of the section. We denote by $y^{s}+\xi^{(\sqrt{-1}\infty)}(s)$ $\times y^{-(l-l')+1-s}$ and $\xi^{(r)}(s)y^{-(l-l')+1-s}$, the constant terms of Fourier expansions with respect to $x$, of $y^{s}E_{l-l',\rho,N}(z,s)$ at $\sqrt{-1}\infty$ and at $r\in\mathcal{C}_{0}(N),\ne 1/N$ respectively. The function $\xi^{(\sqrt{-1}\infty)}(s)$ is also denoted by $\xi^{(1/N)}(s)$. For $T>0$,
\begin{align}
&h_{+}(T,s)=h_{+}^{(\sqrt{-1}\infty)}(T,s):=\sum_{j}\tfrac{a_{j}T^{q_{j}+l-1+s}}{q_{j}+l-1+s}\left(\rule{0cm}{1.em}\right.\!\!=\int_{0}^{T}y^{l+s}P_{F}^{(\sqrt{-1}\infty)}(y)dy\ (\Re s\gg 0)\left)\rule{0cm}{1.em}\right.,\nonumber\\
&h_{-}(T,s)=h_{-}^{(\sqrt{-1}\infty)}(T,s):=\sum_{j}\tfrac{a_{j}T^{q_{j}+l'-s}}{q_{j}+l'-s}\left(\rule{0cm}{1.em}\right.\!\!=-\int_{T}^{\infty}y^{l'+1-s}P_{F}^{(\sqrt{-1}\infty)}(y)dy\ (\Re s\gg 0)\left)\rule{0cm}{1.em}\right.,\nonumber\\
&h^{(r)}(T,s):=w^{(r)}\sum_{j}\tfrac{a_{j}^{(r)}T^{q_{j}^{(r)}+l'-s}}{q_{j}^{(r)}+l'-s}\left(\rule{0cm}{1.em}\right.\!\!=-w^{(r)}\int_{T}^{\infty}y^{l'+1-s}P_{F}^{(r)}(y)dy\ (\Re s\gg 0)\left)\rule{0cm}{1.em}\right.(r\ne\tfrac{1}{N}),\nonumber\\
&h(T,s)=h_{{+}}(T,s)+\overline{\xi^{(1/N)}(\overline{s})}h_{{-}}(T,s)+\sum_{r\in\mathcal{C}_{0}(N)-\{1/N\}}\overline{\xi^{(r)}(\overline{s})}h^{(r)}(T,s)  \label{eqn:defHTS}
\end{align}
where $\Re s\gg0$ implies that $\Re s$ is sufficiently large.  If $\Re s\ll0$ (sufficiently small), then $h_{+}(T,s)=-\int_{T}^{\infty}y^{l+s}P_{F}^{(1/N)}(y)dy,h_{-}(T,s)=\int_{0}^{T}y^{l'+1-s}P_{F}^{(1/N)}(y)dy,h^{(r)}(T,s)=w^{(r)}\int_{0}^{T}y^{l'+1-s}P_{F}^{(r)}(y)dy$.
\begin{thm}\label{thm:analytic continuation} Let $l,l',k,,\xi^{(r)},h_{+},h_{-},h^{(r)},F,R(F,s)$ be as above.  Then for $s$ with $\Re s$ sufficiently large, the integral (\ref{eqn:rst}) converges and $R(F,s)$ is defined. Further $R(F,s)$ extends meromorphically to the whole $s$ plane, and its possible poles are poles of the Eisenstein series $E_{l-l',\rho,N}(z,s)$ and $s=q_{j}^{(r)}+l'\ (r\in\mathcal{C}_{0}(N))$,\ $s=-q_{j}-l+1$. We have
\begin{align}
R(F,s)&=\lim_{T\to\infty}\left(\rule{0cm}{1.5em}\right.\int_{\mathfrak{F}_{T}(N)}y^{l+s}F(z)\overline{E_{l-l',\rho,N}(z,\overline{s})}\tfrac{dxdy}{y^2}-h_{+}(T,s)\left.\rule{0cm}{1.5em}\right) \label{eqn:RFs}
\end{align}
for $s$ with sufficiently large $\Re s$. For $s$ with sufficiently small $\Re s$ at which $E_{l-l',\rho,N}(z,s)$ is holomorphic, we have
\begin{align}
R(F,s)&=\lim_{T\to\infty}\left(\rule{0cm}{1.5em}\right.
\int_{\mathfrak{F}_{T}(N)}y^{l+s}F(z)\overline{E_{l-l',\rho,N}(z,\overline{s})}\tfrac{dxdy}{y^2}\nonumber\\
&\hspace{4em}-\overline{\xi^{(1/N)}(\overline{s})}h_{{-}}(T,s)-\sum_{r\in\mathcal{C}_{0}(N)-\{1/N\}}\overline{\xi^{(r)}(\overline{s})}h^{(r)}(T,s)\left.\rule{0cm}{1.5em}\right). \label{eqn:RFs2}
\end{align}
\end{thm}
\begin{proof}
 Let 
\begin{align}
\mathfrak{D}:=\{z\in\mathfrak{H}\mid|x|\le1/2\}.\label{eqn:defD}
\end{align}
We denote by $\mathfrak{F}_{T}\ (T>1)$, the truncated fundamental domain $\{z \in\mathfrak{D}\mid |z|\ge1,\, y\le T\}$, and by $\mathfrak{F}_{T}(N)$, the union $\bigcup_{A} A\mathfrak{F}_{T}\,(\subset\mathfrak{D})$ where $A$ runs over the representatives of $\mathrm{SL}_{2}(\mathbf{Z})$ modulo $\Gamma_{0}(N)$.

Let $a/c$ be a rational number with $-1/2<a/c\le 1/2,\,c>0,\,(a,c)=1$. For $T>1$, let  $S_{a/c,T}$ be the open disk of radius $(2c^{2}T)^{-1}$ tangent to the real axis at $a/c$ (Fig. \ref{fig:fdomain}) where $S_{1/2,T}$ is exceptionally the union of the left half of the disk tangent to the real axis at $1/2$ and the right half of the disk tangent to the real axis at $-1/2$. The domain $\{z\in\mathfrak{H}\mid y>T\}$ and all $S_{a/c,T}$ but $S_{1/2,T}$ are mapped onto each other by matrices of $\mathrm{SL}_{2}(\mathbf{Z})$. 

Put $\mathfrak{B}_{T}(n):=\{z\in\mathfrak{H}\mid -1/2\le x\le n-1/2,\,y>T\}$ for $n\in\mathbf{N}$. For  $r\in\mathcal{C}_{0}(N)$, let $A_{r}$ be as in (\ref{defAr}) so that it sends $\mathfrak{B}_{T}(w^{(r)})$ onto $S_{r,T}\cap\mathfrak{F}(N)$. 
By Lemma \ref{lem:ub}, the integral $\int_{0}^{T}\int_{-1/2}^{1/2}y^{l+s}F(z)\frac{dxdy}{y^{2}}$ converges absolutely and uniformly for $s$ with  $\Re s\gg0$. Let $\mathfrak{D}_{T}=\{z\in\mathfrak{D}\mid\Im z\le T\}-\cup_{r\in\mathcal{C}_{0}(N)}\cup_{a/c\in Q(r)}S_{a/b,T}$. Then by taking suitable representatives of $\Gamma_{0}(N)$ modulo $\Gamma_{\infty}=\{\pm\left(1\ \,n\atop 0\ \,1\right)\mid n\in\mathbf{Z}\}$, we have $\mathfrak{D}=\cup_{A:\Gamma_{\infty}\backslash\Gamma_{0}(N)}A\mathfrak{F}(N)$ and $\mathfrak{D}_{T}=\cup_{A:\Gamma_{\infty}\backslash\Gamma_{0}(N)}A\mathfrak{F}_{T}(N)$. Then
\begin{align}
\int_{\mathfrak{F}_{T}(N)}y^{l+s}F(z)\overline{E_{l-l',\rho,N}(z,\overline{s})}\frac{dxdy}{y^2}=\int_{\mathfrak{D}_{T}}y^{l+s}F(z)\frac{dxdy}{y^2}. \label{eqn:int-on-FT}
\end{align}
Indeed if $l-l'$ is integral, then the left hand side of (\ref{eqn:int-on-FT}) is equal to 
\begin{align*}
&\int_{\mathfrak{F}_{T}(N)}y^{l+s}F(z)\{1+\sum_{(c,d)=1,c>0\atop c\equiv0(\mathrm{mod}\,N)}\rho(d)(\overline{cz+d})^{-k}|cz+d|^{-2s}\}\frac{dxdy}{y^2}\\
=&\int_{\mathfrak{F}_{T}(N)}[y^{l+s}F(z)+\sum_{(c,d)=1,c>0\atop c\equiv0(\mathrm{mod}\,N)}\{\overline{\rho}(d)(\overline{cz+d})^{k}|cz+d|^{2s}\}^{-1}y^{l+s}F(z)]\frac{dxdy}{y^2}\\
=&\int_{\cup_{A:\Gamma_{\infty}\backslash\Gamma_{0}(N)}A\mathfrak{F}_{T}(N)}y^{l+s}F(z)\frac{dxdy}{y^2}=\int_{\mathfrak{D}_{T}}y^{l+s}F(z)\frac{dxdy}{y^2}. 
\end{align*}
The similar argument holds also in the case that $l-l'$ is a half integer. Since $h_{+}(T,s)=\int_{0}^{T}\int_{-1/2}^{1/2}y^{l+s}P_{F}(y)\tfrac{dxdy}{y^{2}}$, from (\ref{eqn:int-on-FT}) we obtain
\begin{align}
&\int_{\mathfrak{F}_{T}(N)}y^{l+s}F(z)\overline{E_{l-l',\rho,N}(z,\overline{s})}\tfrac{dxdy}{y^2}\nonumber\\
=&\int_{\mathfrak{D}-\mathfrak{B}_{T}(1)}y^{l+s}F(z)\tfrac{dxdy}{y^2}-\int_{\cup_{r\in\mathcal{C}_{0}(N)}\cup_{a/c\in Q(r)}S_{a/b,T}}y^{l+s}F(z)\tfrac{dxdy}{y^2}\nonumber\\
=&\int_{0}^{T}y^{l+s-2}u_{0}(y)dy+h_{{+}}(T,s){-}\sum\limits_{r\in\mathcal{C}_{0}(N)}\sum\limits_{a/c\in Q(r)}\int_{S_{a/b,T}}y^{l+s}F(z)\tfrac{dxdy}{y^2},\label{eqn:int-cal}
\end{align}
and for $r\in\mathcal{C}_{0}(N),\ne 1/N$,
\begin{align}
&\int_{\mathfrak{F}(N)\cap S_{r,T}}y^{l+s}F(z)\overline{E_{l-l',\rho,N}(z,\overline{s})}\tfrac{dxdy}{y^2}=\int_{\cup_{a/c\in Q(r)}S_{a/b,T}}y^{l+s}F(z)\tfrac{dxdy}{y^2}\nonumber\\
=&\sum_{a/c\in Q(r)}\int_{S_{a/b,T}}y^{l+s}F(z)\tfrac{dxdy}{y^2}.\label{eqn:int-on-S}
\end{align}
We have $E_{l-l',\rho,N}(z,s)|_{A_{r}}=O(y^{1-2s})$ as $y\longrightarrow\infty$ for $r\in\mathcal{C}_{0}(N),\ne 1/N$, and  $E_{l-l',\rho,N}(z,s)-1=O(y^{1-2s})$. We take $s$ so that  $\Re s>\max\{q+l'+1,q+l\}$. Then $y^{l+s}F|_{A{r}}(z)\overline{E_{l-l',\rho,N}(z,\overline{s})}|_{A_{r}}$ is integrable on $\mathfrak{B}_{T}(w^{(r)})$, and $y^{l+s}F(z)$ $\times \{\overline{E_{l-l',\rho,N}(z,\overline{s})}-1\}$ is integrable on $\mathfrak{B}_{T}(1)$.  Then the integral (\ref{eqn:int-on-S}) is equal to $\int_{\mathfrak{B}_{T}(w^{(r)})}y^{l+s}F|_{A_{r}}(z)\overline{E_{l-l',\rho,N}(z,\overline{s})}|_{A_{r}}\frac{dxdy}{y^2}$. Since there holds an equality\linebreak $\cup_{A:\Gamma_{\infty}\backslash(\Gamma_{0}(N)-\Gamma_{\infty})}A\mathfrak{B}_{T}(1)=\cup_{a/c\in Q(1/N)}S_{a/c,T}$ for a suitable choice of representatives, we have
\begin{align*}
&\int_{\mathfrak{B}_{T}(1)}y^{l+s}F(z)\{\overline{E_{l-l',\rho,N}(z,\overline{s})}-1\}\tfrac{dxdy}{y^2}\\
=&\int_{\mathfrak{B}_{T}(1)}y^{l+s}F(z)\sum_{(c,d)=1,c>0\atop c\equiv0(\mathrm{mod}\,N)}\rho(d)(\overline{cz+d})^{-k}|cz+d|^{-2s}\frac{dxdy}{y^2}\\
=&\sum_{a/c\in Q(1/N)}\int_{S_{a/c,T}}y^{l+s}F(z)\tfrac{dxdy}{y^2}.
\end{align*}
If $\Re s\gg0$, then we obtain from (\ref{eqn:int-cal}),
\begin{align}
&R(F,s)-\int_{T}^{\infty}y^{l+s-2}u_{0}(y)dy=\int_{0}^{T}y^{l+s-2}u_{0}(y)dy\nonumber\\
=&\int_{\mathfrak{F}_{T}(N)}y^{l+s}F(z)\overline{E_{l-l',\rho,N}(z,\overline{s})}\tfrac{dxdy}{y^2}+\int_{\mathfrak{B}_{T}(1)}y^{l+s}F(z)\{\overline{E_{l-l',\rho,N}(z,\overline{s})}-1\}\tfrac{dxdy}{y^2}\nonumber\\
&+\sum_{r\in\mathcal{C}_{0}(N)-\{1/N\}}\int_{\mathfrak{B}_{T}(w^{(r)})}y^{l+s}F|_{A_{r}}(z)\overline{E_{l-l',\rho,N}(z,\overline{s})}|_{A_{r}}\tfrac{dxdy}{y^2}-h_{+}(T,s). \label{eqn:RFs3}
\end{align}
Thus integral $\int_{0}^{T}y^{l+s-2}u_{0}(y)dy$ converges, and since $y^{l+s-2}u_{0}(y)$ is rapidly decreasing as $y\longrightarrow\infty$ by our assumption, the integral (\ref{eqn:rst}) converges for $\Re s\gg0$.

Noting that $\int_{T}^{\infty}y^{l+s-2}u_{0}(y)dy+\int_{\mathfrak{B}_{T}(1)}y^{l+s}F(z)\{\overline{E_{l-l',\rho,N}(z,\overline{s})}-1\}\tfrac{dxdy}{y^2}=$\linebreak$\int_{\mathfrak{B}_{T}(1)}\{y^{l+s}F(z)\overline{E_{l-l',\rho,N}(z,\overline{s})}-(y^{l+s}+\overline{\xi^{(1/N)}(\overline{s})}y^{l'+1-s})P_{F}(y)\}\tfrac{dxdy}{y^2}$, we obtain from (\ref{eqn:RFs3}),
\begin{align}
R(F,s)=&\int_{\mathfrak{F}_{T}(N)}y^{l+s}F(z)\overline{E_{l-l',\rho,N}(z,\overline{s})}\tfrac{dxdy}{y^2}\nonumber\\
&+\int_{T}^{\infty}\!\!\int_{-1/2}^{1/2}\{y^{l+s}F(z)\overline{E_{l-l',\rho,N}(z,\overline{s})}-(y^{l+s}+\overline{\xi^{(1/N)}(\overline{s})}y^{l'+1-s})P_{F}(y)\}\tfrac{dxdy}{y^2}\nonumber\\
&+\sum_{r\in\mathcal{C}_{0}(N)-\{1/N\}}\int_{T}^{\infty}\!\!\int_{-1/2}^{w^{(r)}-1/2}\{y^{l+s}F|_{A_{r}}(z)\overline{E_{l-l',\rho,N}(z,\overline{s})}|_{A_{r}}\nonumber\\
&\hspace{6em}-\overline{\xi^{(r)}(\overline{s})}y^{l'{+}1{-}s}P_{F}^{(r)}(y)\}\tfrac{dxdy}{y^2}-h(T,s) \label{eqn:RFs4}
\end{align}
with $h(T,s)$ of (\ref{eqn:defHTS}) since $h(T,s)=h_{{+}}(T,s)-\int_{T}^{\infty}\!\int_{-1/2}^{1/2}\overline{\xi^{(1/N)}(\overline{s})}y^{l'+1-s}P_{F}(y)\frac{dxdy}{y^{2}}-\sum_{r\in\mathcal{C}_{0}(N)-\{1/N\}}$ $\int_{T}^{\infty}\!\int_{-1/2}^{w^{(r)}-1/2}\overline{\xi^{(r)}(\overline{s})}y^{l'{+}1{-}s}P_{F}^{(r)}(y)\frac{dxdy}{y^{2}}$.

By our assumption that the second term of (\ref{eqn:fearc}) is rapidly decreasing as $y\longrightarrow\infty$, and by Lemma \ref{lem:eisi} and Lemma \ref{lem:eishi}, the integrands of the second and third terms of (\ref{eqn:RFs4}) are rapidly decreasing. The equality (\ref{eqn:RFs4}) is proved for $s$ with $\Re s\gg0$. However the right hand side is a meromorphic function on the whole $s$ plane since  the first integral is over a compact set, and the integrands of other integrals are rapidly decreasing, and $h(T,s)$ is a meromorphic function on the $s$ plane. Thus $R(F,s)$ is a meromorphic on the $s$ plane, and $R(F,s)$ is holomorphic at $s$ if both of $E_{l-l',\rho,N}(z,s)$ and $h(T,s)$ are holomorphic at $s$. This shows the second statement of the theorem.

In the right hand side of (\ref{eqn:RFs4}), the second and the third terms tend to $0$ as $T\longrightarrow\infty$, and $h(T,s)-h_{+}(T,s)$ also tend to $0$ as $T\longrightarrow\infty$ for $s$ with $\Re s\gg0$. This shows (\ref{eqn:RFs}). The equality (\ref{eqn:RFs2}) is proved similarly.
\end{proof}
\begin{cor}\label{cor:int-eis} (i) For $s$ with sufficient large $\Re s$, there holds equalities
\begin{align}
R(F,s)&=\int_{\mathfrak{F}(N)}\{y^{l+s}F(z)\overline{E_{l-l',\rho,N}(z,\overline{s})}-\mathbf{E}_{0}(z,s)\}\tfrac{dxdy}{y^2} \label{eqn:RFs5}
\end{align}
for $\mathbf{E}_{0}(z,s):=\sum_{j}a_{j}E_{\Gamma_{0}(N)}(z,l+q_{j}+s)\in\mathcal{M}_{0,0}(N)$.

(ii) Let the notations be the same as in the theorem. For $s$ with $\Re s\ll0$ at which $E_{l-l',\rho,N}(z,s)$ is holomorphic, there holds
\begin{align*}
R(F,s)&=\int_{\mathfrak{F}(N)}\{y^{l+s}F(z)\overline{E_{l-l',\rho,N}(z,\overline{s})}\tfrac{dxdy}{y^2}-\mathbf{E}_{0}(z,s)\}\tfrac{dxdy}{y^{2}}
\end{align*}
for $\mathbf{E}_{0}(z,s):=\overline{\xi^{(1/N)}(\overline{s})}\sum_{j}a_{j}E_{\Gamma_{0}(N)}(z,q_{j}{+}l'{+}3{-}s)+\sum_{r\in\mathcal{C}_{0}(N){-}\{1/N\}}\overline{\xi^{(r)}(\overline{s})}$\linebreak$\times E_{\Gamma_{0}(N),r}(z,q_{j}^{(r)}{+}l'{+}3{-}s)\in\mathcal{M}_{0,0}(N)$.

(iii) Assume that ${E_{l-l',\rho,N}(z,s)}$ is holomorphic in $s$ at $s=s_{0}$.  Let $\widetilde{h}(T,s)$ denote the sum of terms of (\ref{eqn:defHTS}) so that for $s=s_{0}$ the powers of $T$ are nonzero and their real parts are nonnegative. Let $\tfrac{C_{0}}{s-s_{0}}T^{s-s_{0}}-\tfrac{\alpha(s)}{s-s_{0}}T^{-s+s_{0}}$ be the sum of terms of (\ref{eqn:defHTS}) so that the powers of $T$ for $s=s_{0}$ are $0$, where $C_{0}$ is a constant. Then
\begin{align}
&\left(\rule{0cm}{1.2em}\right.R(F,s){+}\frac{C_{0}{-}\alpha(s_{0})}{s-s_{0}}-\frac{d}{ds}\alpha(s_{0})\left)\rule{0cm}{1.2em}\right.|_{s=s_{0}}\nonumber\\
\hspace*{-.3em}=&\lim_{T\to\infty}\left(\rule{0cm}{1.2em}\right.\int_{\mathfrak{F}_{T}(N)}y^{l{+}s}F(z)\overline{E_{l{-}l',\rho,N}(z,\overline{s})}\tfrac{dxdy}{y^2}-(C_{0}{+}\alpha(s_{0}))\log T{-}\widetilde{h}(T,s_{0})\left)\rule{0cm}{1.2em}\right..\label{eqn:r-i-rel}
\end{align}
If $C_{0}=\alpha(s)=0$, then $R(F,s)$ is holomorphic at $s=s_{0}$.
\end{cor}
\begin{proof} (i) We put $H(z,s):=y^{l+s}F(z)\overline{E_{l-l',\rho,N}(z,\overline{s})}-\sum_{j}a_{j}E_{\Gamma_{0}(N)}(z,l+q_{j}+s)$. Then $H(z,s)=O(y^{l'+1+q-s}),\ H(z,s)|A_{r}=O(y^{l'+1+q-s})$ as $y\longrightarrow\infty$ where $q$ is as in Lemma \ref{lem:ub}. Hence the integral of the right hand side of (\ref{eqn:RFs5}) converges for $\Re s\gg0$. We apply the argument in the theorem to $E_{\Gamma_{0}(N)}(z,l+q_{j}+s)$. Since the term $u_{0}(y)$ of $E_{\Gamma_{0}(N)}(z,l+q_{j}+s)$ is $0$, we obtain from (\ref{eqn:RFs3}),
\begin{align*}
0=&\int_{\mathfrak{F}_{T}(N)}E_{\Gamma_{0}(N)}(z,l{+}q_{j}{+}s)\tfrac{dxdy}{y^2}+\int_{\mathfrak{B}_{T}(1)}\{E_{\Gamma_{0}(N)}(z,l{+}q_{j}{+}s)-y^{l+q_{j}+s}\}\tfrac{dxdy}{y^2}\\
&+\sum_{r\in\mathcal{C}_{0}(N)-\{1/N\}}\int_{\mathfrak{B}_{T}(w^{(r)})}E_{\Gamma_{0}(N)}(z,l{+}q_{j}{+}s)|_{A_{r}}\tfrac{dxdy}{y^2}-\tfrac{T^{q_{j}+l-1+s}}{q_{j}+l-1+s}.
\end{align*}
Then by this and by the equality (\ref{eqn:RFs3}), $R(F,s)$ is equal to
\begin{align}
&\int_{\mathfrak{F}_{T}(N)}H(z,s)\tfrac{dxdy}{y^2}+\int_{T}^{\infty}\!\!\int_{-1/2}^{1/2}\{H(z,s)+Q(y,s)\}\tfrac{dxdy}{y^2}\nonumber\\
&+\sum_{r\in\mathcal{C}_{0}(N)-\{1/N\}}\int_{T}^{\infty}\!\!\int_{-1/2}^{w^{(r)}-1/2}H(z,s)|_{A_{r}}\tfrac{dxdy}{y^2}+Q'(T,s)\label{eqn:st}
\end{align}
where $Q(y,s)$ is a linear combination of powers of $y$, and $Q'(T,s)$ is a linear combination of powers of $T$. It is easily checked that the real parts of powers of terms  in $Q(y,s)$ or in $Q'(T,s)$ are all negative if $\Re s\gg0$. The first term of (\ref{eqn:st}) tends to the right hand side of (\ref{eqn:RFs5}) as $T\longrightarrow\infty$ and the other terms tend to $0$. This shows the equality (\ref{eqn:RFs5}).

The assertion (ii) is proved similarly.

(iii) In (\ref{eqn:RFs4}), the first integral is holomorphic in $s$ at $s=s_{0}$ since it is over the compact set, and the other integrals are also since integrands are rapidly decreasing. The coefficients of $h(T,s)$ of (\ref{eqn:RFs4}) in $T$ as well as $\alpha(s)$ are holomorphic at $s=s_{0}$ from our assumption that ${E_{l-l',\rho,N}(z,s)}$ is holomorphic at $s=s_{0}$. If $C_{0}=\alpha(s)=0$, then $R(F,s)$ is obviously holomorphic at $s=s_{0}$. By (\ref{eqn:RFs4}), we can write $R(F,s)$ as $R(F,s)=\int_{\mathfrak{F}_{T}(N)}y^{l+s}F(z)\overline{E_{l-l',\rho,N}(z,\overline{s})}\tfrac{dxdy}{y^2}+g_{T}(z,s)-h(T,s)$ for a holomorphic function $g_{T}(z,s)$ of $s$ rapidly decreasing as $T\longrightarrow\infty$ which is a sum of integrals in (\ref{eqn:RFs4}) other than the first one.  If we put $n(T,s):=h(T,s)-\frac{C_{0}T^{s-s_{0}}}{s-s_{0}}+\frac{\alpha(s)T^{-s+s_{0}}}{s-s_{0}}-\widetilde{h}(T,s)$, then $\lim_{T\to\infty}n(T,s_{0})=0$ since the powers of $T$ of terms in $n(T,s_{0})$ have negative real parts. By the Taylor expansions $T^{s-s_{0}}=1+(s-s_{0})\log T+O((s-s_{0})^{2}),\,T^{-s+s_{0}}=1-(s-s_{0})\log T+O((s-s_{0})^{2})$ at $s=s_{0}$, we have $\{R(F,s){+}\tfrac{C_{0}{-}\alpha(s_{0})}{s-s_{0}}-\tfrac{d}{ds}\alpha(s_{0})\}|_{s=s_{0}}=\int_{\mathfrak{F}_{T}(N)}y^{l+s_{0}}F(z)\overline{E_{l-l',\rho,N}(z,\overline{s}_{0})}\tfrac{dxdy}{y^2}+g_{T}(z,s_{0})-(C_{0}{+}\alpha(s_{0}))\log T-\widetilde{h}(T,s_{0})-n(T,s_{0})$. Taking the limit as $T\longrightarrow\infty$ of the right hand side, the equality (\ref{eqn:r-i-rel}) follows.
\end{proof}

\begin{thm}\label{thm:functional equation} Assume that $F(z)$ of Theorem \ref{thm:analytic continuation} is in $\mathcal{M}_{l,l-k}(N,\rho)$ with $k\in\mathbf{Z},\ge0$ and with $\rho\in(\mathbf{Z}/N)^{\ast}$ where $\rho$ and $k$ have the same parity. Put $\widetilde{F}(z):=F|_{S_{N}}(z)=(N^{1/2}z)^{-k}|N^{1/2}z|^{-2(l-k)}F(-1/(Nz))\in \mathcal{M}_{l,l-k}(N,\overline{\rho})$ with $S_{N}:=\mbox{\tiny$\left(\begin{array}{@{}c@{\,}c@{}}0&-N^{-1/2}\\N^{1/2}&0\end{array}\right)$}$. Then $R(F,s)$ defined in (\ref{eqn:rst}) has the meromorphic continuation to the whole complex plane, and satisfies the functional equation
\begin{align*}
R(F,-k+1-s)
=\pi^{-1/2}U_{k,\rho}(s)\sum_{0<P|\mathfrak{e}_{\rho}\mathfrak{f}_{\rho}^{-1}}U_{k,\rho,P}(s)R(\mathrm{tr}_{\Gamma_{0}(N)/\Gamma_{0}(\mathfrak{f}_{\rho}P),\rho}(\widetilde{F}),s),
\end{align*}
where $\mathrm{tr}_{\Gamma_{0}(N)/\Gamma_{0}(M),\rho}(\widetilde{F})$ denotes $\sum\limits_{A:\Gamma_{0}(N)/\Gamma_{0}(M)}\rho(d)\widetilde{F}|_{A}(z)$, $d$ being the $(2,2)$ entry of $A$, and where $U_{k,\rho}(s)$ and $U_{k,\rho,P}(s)$ are as in (\ref{eqn:Uk1}) and (\ref{eqn:Uk2}) respectively. 
\end{thm} 
\begin{proof} By (\ref{eqn:felu}) and by Corollary \ref{cor:int-eis} (ii), we have
\begin{align}
&(-1)^{k}\pi^{1/2}U_{k,\rho}(s)^{-1}R(F,-k+1-s)\nonumber\\
=&\int_{\mathfrak{F}(N)}\{y^{l}F(z)\overline{z}^{-k}\sum_{0<P|\mathfrak{e}_{\rho}\mathfrak{f}_{\rho}^{-1}}\prod_{p|P}(1-\widetilde{\rho}(p)p^{k+2s})\varphi(\mathfrak{e}_{\rho}\mathfrak{f}_{\rho}^{-1}P^{-1})\nonumber\\
&\hspace{10em}\times\Im(-\tfrac{1}{Nz})^{s}\overline{E_{k,\overline{\widetilde{\rho}}\mathbf{1}_{P},\mathfrak{f}_{\rho}P}(-\tfrac{1}{Nz},\overline{s})}-\mathbf{E}_{0}(z,s)\}\tfrac{dxdy}{y^2}\label{eqn:int-eis-h}
\end{align}
for a linear combination $\mathbf{E}_{0}(z,s)$ of Eisenstein series of weight $(0,0)$ if $\Re s\gg 0$ and $E_{k,\rho,N}(z,s)$ is holomorphic at $s$. The integrand of (\ref{eqn:int-eis-h}) takes the value $0$ at each cusp, and  hence its transformation by the matrix $S_{N}$ also takes the value $0$ at each cusp. Then (\ref{eqn:int-eis-h}) is equal to
\begin{align*}
&\int_{\mathfrak{F}(N)}\{(-1)^{k}y^{l+s}\widetilde{F}(z)\sum_{0<P|\mathfrak{e}_{\rho}\mathfrak{f}_{\rho}^{-1}}\prod_{p|P}(1-\widetilde{\rho}(p)p^{k+2s})\\
&\hspace{7em}\times\varphi(\mathfrak{e}_{\rho}\mathfrak{f}_{\rho}^{-1}P^{-1})\overline{E_{k,\overline{\widetilde{\rho}\mathbf{1}_{P}},\mathfrak{f}_{\rho}P}(z,\overline{s})}-\widetilde{\mathbf{E}}_{0}(z,s)(z,s)\}\tfrac{dxdy}{y^2}
\end{align*}
where $\widetilde{\mathbf{E}}_{0}(z,s)=\mathbf{E}_{0}(-\frac{1}{Nz},s)$. Hence
\begin{align*}
&(-1)^{k}\pi^{1/2}U_{k,\rho}(s)^{-1}R(F,-k+1-s)\nonumber\\
=&(-1)^{k}\sum_{0<P|\mathfrak{e}_{\rho}\mathfrak{f}_{\rho}^{-1}}\prod_{p|P}(1-\widetilde{\rho}(p)p^{k+2s})\varphi(\mathfrak{e}_{\rho}\mathfrak{f}_{\rho}^{-1}P^{-1})R(\mathrm{tr}_{\Gamma_{0}(N)/\Gamma_{0}(\mathfrak{f}_{\rho}P)}(\widetilde{F}),s),
\end{align*}
which shows our assertion.
\end{proof}

\begin{cor}\label{cor:lseries} Let $k\in\mathbf{Z},\ge0$ and let $l\in\frac{1}{2}\mathbf{Z},\, l>0,\,l\ge k$. Let $f,g$ be holomorphic modulars forms for $\Gamma_{0}(N)$ of weight $l$ and of weight $l-k$ with characters respectively. We assume that $f\overline{g}\in\mathcal{M}_{l,l-k}(N,\rho)$ for $\rho\in(\mathbf{Z}/N)^{\ast}$ with the same parity as $k$. 

 (i) Then $L(s;f,g)$ defined in (\ref{eqn:lseries}) converges at least if $\Re s>\max\{2l-k-1,1/2\}$, and extends meromorphically to the whole $s$-plane. Let $\widetilde{f\overline{g}}(z):=(f\overline{g})|_{S_{N}}(z)\in\mathcal{M}_{l,l-k}(N,\overline{\rho})$ with $S_{N}:=\mbox{\tiny$\left(\begin{array}{@{}c@{\,}c@{}}0&-N^{-1/2}\\N^{1/2}&0\end{array}\right)$}$. For $M\in\mathbf{N}$ with $\mathfrak{f}_{\rho}|M|N$, put \linebreak$\mathrm{tr}_{\Gamma_{0}(N)/\Gamma_{0}(M),\rho}(\widetilde{f\overline{g}}):=\sum_{A:\Gamma_{0}(N)/\Gamma_{0}(M)}$ $\rho(d)\widetilde{f\overline{g}}|_{A}(z)$, $d$ being the $(2,2)$ entry of $A$. Let $L(s;\mathrm{tr}_{\Gamma_{0}(N)/\Gamma_{0}(M),\rho}(\widetilde{f\overline{g}}))$ $:=\sum_{n=1}^{\infty}\frac{c_{n}^{(M)}}{n^{s}}$ if $\mathrm{tr}_{\Gamma_{0}(N)/\Gamma_{0}(M),\rho}(\widetilde{f\overline{g}})$ has $\sum_{n=0}^{\infty}c_{n}^{(M)}\times$ $e^{-4\pi ny}$ as the constant term of its Fourier expansion with respect to $x$. Then we have a functional equation
\begin{align}
L(l-k-s;f,g)=&\tfrac{2^{-2k+2-4s}\pi^{-k+1/2-2s}\Gamma(l{-}1{+}s)}{\Gamma(l{-}k{-}s)}U_{k,\rho}(s)\nonumber\\&\hspace{1em}\times\sum_{0<P|\mathfrak{e}_{\rho}\mathfrak{f}_{\rho}^{-1}}U_{k,\rho,P}(s)L(l-1+s;\mathrm{tr}_{\Gamma_{0}(N)/\Gamma_{0}(M),\rho}(\widetilde{f\overline{g}})).\label{eqn:fn-eq3}
\end{align}

(ii) Assume that $\rho$ is primitive with $N=\mathfrak{f}_{\rho}$, or $k\ge1$.  Let $P_{y^{l}f\overline{g}}^{(r)}(y)$ be as in (\ref{eqn:dfP}).  If $P_{y^{l}f\overline{g}}^{(\sqrt{-1}\infty)}(y)$  does not have a term containing $y$ to the power of $1$, and if $P_{y^{l}f\overline{g}}^{(r)}(y)$ does not have a term containing $y^{k}$ for any $r\in\mathcal{C}_{0}(N)$, then there holds
\begin{align}
\langle f(z),g(z)E_{k,\rho,N}(z,0) \rangle_{\Gamma_{0}(N)}=(4\pi)^{-l+1}\Gamma(l-1)L(l-1;f,g)\label{eqn:l-psp}
\end{align}
with the scalar product defined in (\ref{eqn:psp}). (Note that the right hand side is the value of the analytic continuation of $\Gamma(l-1+s)L(l-1+s;f,g)$ at $s=0$.) Suppose otherwise. If $C_{0}$ is a coefficient of $y$ in $P_{y^{l}f\overline{g}}^{(\sqrt{-1}\infty)}(y)$, and if $\alpha(s)$ is the coefficient of $y^{k}$ in $\sum_{r\in\mathcal{C}_{0}(N)}\overline{\xi^{(r)}(\overline{s})}w^{(r)}P_{y^{l}f\overline{g}}^{(r)}(y)$, then the equation (\ref{eqn:l-psp}) holds replacing the right hand side by $\{(4\pi)^{-l+1-s}\Gamma(l{-}1{+}s)L(l{-}1{+}s;f,g)+s^{-1}(C_{0}{-}\alpha(0))-\frac{d}{ds}\alpha(0)\}|_{s=0}$.

(iii) Let $k=0,\,\rho=\mathbf{1}_{N}$ . If $P_{y^{l}f\overline{g}}^{(\sqrt{-1}\infty)}(y)$ does not have a nonzero constant term, and if $P_{y^{l}f\overline{g}}^{(r)}(y)$ does not have a term containing $y$ to the power of $1$ for any $r\in\mathcal{C}_{0}(N)$, then  the equality (\ref{eqn:psp-formula2}) holds.

If $C_{0}$ is a constant term of $P_{y^{l}f\overline{g}}^{(\sqrt{-1}\infty)}(y)$, and if $\alpha(s)$ is the coefficient of $y$ in $\sum_{r\in\mathcal{C}_{0}(N)}\overline{\xi^{(r)}(\overline{s})}w^{(r)}P_{y^{l}f\overline{g}}^{(r)}(y)$, then an equality 
\begin{align*}
\langle f,g\rangle_{\Gamma_{0}(N)}-C_{0}+\tfrac{1}{2}\tfrac{d^{2}}{ds^{2}}(s{-}1)\alpha(s)|_{s=1}=\mathrm{Res}_{s=l}\tfrac{\Gamma(s)N\prod_{p|N}(1+\tfrac{1}{p})}{3\cdot 4^{s}\pi^{s-1}}L(s;f,g)
\end{align*}
holds.
\end{cor}
\begin{proof} (i) We take $f(z)\overline{g}(z)$ as $F(z)$ in the theorem and then $R(F,s)=(4\pi)^{-l+1-s}\Gamma(l-1+s)L(l-1+s;f,g)$. The other assertion follows from the theorem.  

(ii) Our assumption implies that the Eisenstein series $E_{k,\rho}(z,s)$ is holomorphic in $s$ at $s=0$.  Then the assertion follows from the identity (\ref{eqn:r-i-rel}) for $s_{0}=0$. Indeed $(C_{0}+\alpha(0))\log T+\widetilde{h}(T,0)$ of the right hand side of (\ref{eqn:r-i-rel}) is equal to $\sum_{r\in\mathcal{C}_{0}(N)}Q_{y^{l}f\overline{gE_{k,\rho}(z,\overline{0})}}^{(r)}(T)$, and by (\ref{eqn:psp}), the right hand side of (\ref{eqn:r-i-rel}) equals\linebreak $\langle f(z),g(z)E_{k,\rho,N}(z,0) \rangle_{\Gamma_{0}(N)}$.

(iii) The Eisenstein series $y^{s}E_{0,\mathbf{1}_{N},N}(z,s)$ has a simple pole at $s=1$ with residue $C=3(\pi N\prod_{p|N}(1+p^{-1}))^{-1}$, and $C=\mathrm{Res}_{s=1}\xi^{(r)}(s)$ for all cusps $r\in\mathcal{C}_{0}(N)$.  So $(s-1)\xi^{(r)}(s)$ tends to $C$ as $s\longrightarrow1$. By (\ref{eqn:RFs4}),
\begin{align}
&(s-1)R(f\overline{g},s)\nonumber\\
=&\int_{\mathfrak{F}_{T}(N)}y^{l+s}f(z)\overline{g(z)}(s-1)\overline{E_{0,\mathbf{1}_{N},N}(z,\overline{s})}\tfrac{dxdy}{y^2}+g_{T}(z,s)-(s-1)h(T,s) \label{eqn:RFs4-2}
\end{align}
where $g_{T}(z,s)$ is $s-1$ times the sum of all integrals but the first one in (\ref{eqn:RFs4}) . Then $g_{T}(z,s)$ is holomorphic in $s$ at $s=1$ and $g_{T}(z,1)$ is rapidly decreasing as $T\longrightarrow\infty$. The integral of the right side of (\ref{eqn:RFs4-2}) tends to $C\int_{\mathfrak{F}_{T}(N)}y^{l}(f\overline{g})(z)\tfrac{dxdy}{y^2}$ as $s\longrightarrow1$ since $(s-1)E_{0,\mathbf{1}_{N},N}(z,s)$ uniformly convergent to $C$ on the compact set $\mathfrak{F}_{T}(N)$. If $C_{0}=0$ and $\alpha(s)=0$, then we see from (\ref{eqn:defHTS}) that $(s-1)h(T,s)$ tends to $C\sum_{r}Q_{y^{l}f\overline{g}}^{(r)}(T)+n(T)$ where the powers of $T$ of terms in $n(T)$ have only negative real parts. Then the right hand side of (\ref{eqn:RFs4-2}) turns out to be $C\{\int_{\mathfrak{F}_{T}(N)}y^{l}(f\overline{g})(z)\tfrac{dxdy}{y^2}-\sum_{r}Q_{y^{l}f\overline{g}}^{(r)}(T)\}+g_{T}(z,1)+n(T)$. Taking the limit as $T\longrightarrow\infty$, we obtain (\ref{eqn:psp-formula2}). 

Suppose that  $C_{0}\ne0$ or $\alpha(s)\ne0$. Then $(s-1)h(T,s)$ has a term $C_{0}T^{-1+s}-\alpha(s)T^{1-s}$ additionally. Let $a$ be the coefficient of $y$ in $\sum_{r}P_{y^{l}f\overline{g}}^{(r)}(y)$. Then $\alpha(s)$ has $aC$ as the residue at $s=1$, since all $\xi^{(r)}(s)$ have the common residues $C$. Let $\alpha(s)=\frac{aC}{s-1}+c_{0}+O(s-1)$ be the Laurent expansion at $s=1$ with $c_{0}=\frac{1}{2}\frac{d^{2}}{ds^{2}}(s{-}1)\alpha(s)|_{s=1}$. Then $C_{0}T^{-1+s}-\alpha(s)T^{1-s}=-\tfrac{aC}{s-1}+C_{0}+aC\log T-c_{0}+O(s-1)$. Hence $(s-1)h(T,s)+\tfrac{aC}{s-1}-C_{0}+c_{0}$ tends to $C\sum_{r}Q_{y^{l}f\overline{g}}^{(r)}(T)+n(T)$ as $s\longrightarrow1$, and the same argument as above leads to an equality
\begin{align*}
\tfrac{\Gamma(l-1+s)N\prod_{p|N}(1+\tfrac{1}{p})}{3\cdot 4^{l-1+s}\pi^{l-2+s}}L(l{-}1{+}s;f,g)
=\tfrac{a}{(s-1)^2}+\tfrac{\langle f,g\rangle_{\Gamma_{0}(N)}-C_{0}+\frac{1}{2}\frac{d^{2}}{ds^{2}}(s{-}1)\alpha(s)|_{s=1}}{s-1}+O(1)
\end{align*}
which holds near $s=1$. The last assertion of (iii) follows from this.
\end{proof}
\begin{remark}\label{rem:lseries} Let $f,g\in\mathbf{M}_{1}(N,\rho)$, and let $a_{0}^{(r)},b_{0}^{(r)}$ be the $0$-th Fourier coefficients of $f,g$ at a cusp $r$ respectively. Let $a:=\sum_{r\in\mathcal{C}_{0}(N)}a_{0}^{(r)}b_{0}^{(r)}w^{(r)}$ and $\alpha(s):=\sum_{r\in\mathcal{C}_{0}(N)}a_{0}^{(r)}b_{0}^{(r)} w^{(r)}\xi^{(r)}(s)$ where $w^{(r)}$ is as in (\ref{eqn:w-cusp}), and $\xi^{(r)}(s)$ is as in (\ref{eqn:ct0}). Corollary \ref{cor:lseries} (iii) and its proof imply that
\begin{align*}
\tfrac{\Gamma(s)N\prod_{p|N}(1+\tfrac{1}{p})}{3\cdot 4^{s}\pi^{-1+s}}L(s;f,g)=\tfrac{a}{(s-1)^2}+\tfrac{\langle f,g\rangle_{\Gamma_{0}(N)}+\frac{1}{2}\frac{d^{2}}{ds^{2}}(s{-}1)\alpha(s)|_{s=1}}{s-1}+O(1)
\end{align*}
near $s=1$. Hence if $a\ne0$, then $L(s;f,g)$ has a pole of order $2$ at $s=1$. If $a=0$, then (\ref{eqn:psp-formula2}) holds by adding $\frac{1}{2}\frac{d^{2}}{ds^{2}}(s{-}1)\alpha(s)|_{s=1}$ to the left hand side, and if a product $fg$ is a cusp form, then $\alpha(s)=0$ and (\ref{eqn:psp-formula2}) holds.
\end{remark}

\section{The case of half integral weight}
In this section, the analytic continuation of $L$-function (\ref{eqn:lseries})  is proved in the case $l-l'\in\tfrac{1}{2}\mathbf{Z}$. The properties of $L$-function (\ref{eqn:lseries})  shown in the preceding section are proved also in the half integral weight case. At first we need to obtain the functional equations of Eisenstein series of half integral weight, and then we make the similar argument as the proof of Theorem \ref{thm:functional equation}.
\begin{prop}\label{prop:fn-eq} Let $\rho,\widetilde{\rho},N$ be as in (\ref{eqn:N2}) with a decomposition $\rho=\rho_{2}\rho_{\mathbf{c}}\rho_{\mathbf{r}}$ as in (\ref{eqn:dcmp-rho}). Let $k$ be a nonnegative integer with the same parity as $\rho$, and let $U_{k+1/2,\rho}(s),U_{k+1/2,\rho,R}(s)$ be as in (\ref{eqn:cffe}),(\ref{eqn:cffe2}) respectively. Then $U_{k+1/2,\rho}({-}k{+}1/2{-}s)^{-1}$ $\times y^{{-}k{+}1/2{-}2s}E_{k+1/2,\rho}(z,{-}k{+}1/2{-}s)$ is equal to 
\begin{align*}
\sum_{P|\mathfrak{f}_{\rho_{\mathbf{r}}}}[&(-1)^{k}\sqrt{-1}\iota_{P}U_{k+1/2,\rho,P}({-}k{+}1/2{-}s)E_{k+1/2,\chi_{P^{\vee}}}^{\rho_{2\mathbf{c}}\chi_{(\varepsilon_{P}4\mathfrak{f}_{\rho_{\mathbf{r}}}/P)}}(z,s)\\
&+(-1)^{k}\{\rho_{2}(-1)\chi_{{-}4}(P){+}\sqrt{{-}1}\}U_{k+1/2,\rho,4P}({-}k{+}1/2{-}s)E_{k+1/2,\rho_{2}\chi_{\varepsilon_{P}'4P}}^{\rho_{\mathbf{c}}\chi_{(\mathfrak{f}_{\rho_{\mathbf{r}}}/P)^{\vee}}}(z,s)]
\end{align*}
if $\rho_{2}$ is real, or to
\begin{align*}
\sum_{P|\mathfrak{f}_{\rho_{\mathbf{r}}}}(-1)^{k}\sqrt{-1}\iota_{P}2^{2}\mathfrak{f}_{\rho_{2}}^{-1}U_{k+1/2,\rho,P}({-}k{+}1/2{-}s)E_{k+1/2,\chi_{P^{\vee}}}^{\rho_{2\mathbf{c}}\chi_{(\varepsilon_{P}4\mathfrak{f}_{\rho_{\mathbf{r}}}/P)}}(z,s)
\end{align*}
if $\rho_{2}$ is complex where $\varepsilon_{P}:=\chi_{-4}(\mathfrak{f}_{\rho_{\mathbf{r}}}/P),\varepsilon_{P}':=\chi_{-4}(P)\in\{\pm1\}$.
\end{prop}
\begin{proof} We prove the assertion in the case $\rho_{2}$ is complex. In other cases the proof is similar. Lemma \ref{lem:vac} gives the constant terms of Fourier expansions with respect to $x$, of $y^{-k+1/2-2s}E_{k+1/2,\rho}(z,-k+1/2-s)$ at cusps, that is, $U_{k+1/2,\rho}(-k+1/2-s)^{-1}y^{-k+1/2-2s}E_{k+1/2,\rho}(z,-k+1/2-s)$ has the constant term $2^{2}\mathfrak{f}_{\rho_{2}}^{-1}U_{k+1/2,\rho,P}(-k+1/2-s)$ at $1/P$ for $P|\mathfrak{f}_{\rho_{\mathbf{r}}}$, and the constant term $y^{-k+1/2-2s}$ at $1/N$, and $0$ at other cusps. We put $G_{0}(z;\rho,s):=U_{k+1/2,\rho}({-}k{+}1/2{-}s)^{-1}y^{{-}k{+}1/2{-}2s} E_{k+1/2,\rho}(z,-k+1/2-s)-\sum_{P|\mathfrak{f}_{\rho_{\mathbf{r}}}}(-1)^{k}\sqrt{-1}\iota_{P}2^{2}\mathfrak{f}_{\rho_{2}}^{-1}U_{k+1/2,\rho,P}({-}k{+}1/2{-}s)E_{k+1/2,\chi_{P^{\vee}}}^{2\rho_{\mathbf{c}}\chi_{(\varepsilon_{P}4\mathfrak{f}_{\rho_{\mathbf{r}}}/P)}}(z,s)$. Our purpose is to show $G_{0}(z;\rho,s)=0$. It has the constant terms in the form (\ref{cte}) with $c_{0}=0$ at all the cusps by Lemma \ref{lem:vac2}.  By the construction of $G_{0}(z;\rho,s)$ and by Lemma \ref{lem:vac} and Lemma \ref{lem:vac2}, $G_{0}(z;\rho,s)$ has the constant term (\ref{cte}) with $\xi_{0}(s)\ne0$ at cusp $1/M\in\mathcal{C}_{0}(N)$ only if $\mathfrak{f}_{\rho_{2\mathbf{c}}}|M$.

Assume that $\Re s>3/4$. Then $y^{s}G_{0}(z;\rho,s)\in\mathcal{M}_{k+1/2,0}(N,\rho)$ and its absolute value has the order $O(y^{-k+1/2-\Re s})$ as $y\longrightarrow\infty $ at all cusps. Since $y^{(k+1/2)/2}|y^{s}$ $\times G_{0}(z;\rho,s)|$ is $\Gamma_{0}(N)$-invariant, and vanishes at each cusp, in particular we have $(|y^{s}G_{0}(z;\rho,s)|)|_{A_{r}}=o(y^{-(k+1/2)/2})$ as $y\longrightarrow0+$ for any cusp $r$, $A_{r}$ being as in (\ref{defAr}). 

Let $\xi_{0}(s)y^{-k+1/2-s}$ be the constant term of the Fourier expansion of $y^{s}G_{0}(z;\rho,s)\in\mathcal{M}_{k+1/2,0}(N,\rho)$ at the cusp $\sqrt{-1}\infty$. At first we show that $\xi_{0}(s)$ vanishes. We take a complex variable $t$ so that $-k/2+3/4<\Re t<\Re s$. The series $E_{k+1/2,\rho}(z,t)$ of (\ref{eqn:eshiw}) converges absolutely and locally uniformly on $\mathfrak{H}$ since $k+1/2+2\Re t>2$, and the unfolding trick is available to $E_{k+1/2,\rho}(z,t)$. A product $y^{k+1/2}y^{s}G_{0}(z;\rho,s)\overline{y^{\overline{t}}E_{k+1/2,\rho}(z,\overline{t})}$ $\in\mathcal{M}_{0,0}(N)$ has  the order $O(y^{1-\Re s+\Re t})$ as $y\longrightarrow\infty$ for all cusps. Since $1-\Re s+\Re t<1$, its integral over the fundamental domain converges. Then by using the unfolding trick, there holds
\begin{align*}
\int_{\mathfrak{F}(N)}y^{k+1/2}y^{s}G_{0}(z;\rho,s)\overline{y^{\overline{t}}E_{k+1/2,\rho}(z,\overline{t})}\tfrac{dxdy}{y^{2}}=\int_{\mathfrak{D}}y^{k-3/2+s+t}G_{0}(z;\rho,s)dxdy
\end{align*}
with $\mathfrak{D}$ as in (\ref{eqn:defD}), where the right hand side has meaning since $|y^{k-3/2+s+t}G_{0}(z;\rho,s)|=o(k/2{-}7/4{+}\Re t)$ as $y\longrightarrow 0+$ and $k/2-7/4+\Re t>-1$. However
\begin{align*}
\int_{\mathfrak{D}}y^{k-3/2+s+t}G_{0}(z;\rho,s)dxdy=\int_{0}^{\infty}\xi_{0}(s)y^{-1+t-s}dy,
\end{align*}
and $\xi_{0}(s)$ must be $0$ to converge the integral. Since the function $\xi_{0}(s)$ is meromorphic on the entire plane and since $\xi_{0}(s)=0$ for $s$ with $\Re s>\Re t$, $\xi_{0}(s)$ vanishes identically.

Next we prove that the constant term of $y^{s}G_{0}(z;\rho,s)$ at a cusp $1/(N/P)$ with $P|\mathfrak{f}_{\rho_{\mathbf{r}}}$ vanishes. An Eisenstein series $E_{k+1/2,\rho_{2\mathbf{c}}\chi_{(\varepsilon_{P}4\mathfrak{f}_{\rho_{\mathbf{r}}}/P)}}^{\chi_{P^{\vee}}}(z,s)$ has a constant term (\ref{cte}) with $c_{0}\ne0$ at the cusp $1/(N/P)$. Since $y^{k+t+1/2}y^{s}G_{0}(z;\rho,s)$\linebreak$\times\overline{E_{k+1/2,\rho_{2\mathbf{c}}\chi_{(\varepsilon_{P}4\mathfrak{f}_{\rho_{\mathbf{r}}}/P)}}^{\chi_{P^{\vee}}}(z,\overline{t})}$ is in $\mathcal{M}_{0,0}(N)$, $(y^{k+t+1/2}y^{s}G_{0}(z;\rho,s)\times$\linebreak$\overline{E_{k+1/2,\rho_{2\mathbf{c}}\chi_{(\varepsilon_{P}4\mathfrak{f}_{\rho_{\mathbf{r}}}/P)}}^{\chi_{P^{\vee}}}(z,\overline{t})})|_{A_{1/(N/P)}}$ is in $\mathcal{M}_{0,0}(A_{1/(N/P)}^{-1}\Gamma_{0}(N)A_{1/(N/P)})$. Here \linebreak$E_{k+1/2,\rho_{2\mathbf{c}}\chi_{(\varepsilon_{P}4\mathfrak{f}_{\rho_{\mathbf{r}}}/P)}}^{\chi_{P^{\vee}}}(z,t)|_{A_{1/(N/P)}}$ is in the form $\sum_{c,d}\eta_{c,d}(cz+d)^{-k-1/2}|cz+d|^{-2t}$ with constants $\eta_{c,d}$ of absolute value $1$ or $0$ where $c,d$ runs over the set of the second rows of matrices in $A_{1/(N/P)}^{-1}\Gamma_{0}(N)A_{1/(N/P)}$ with $c>0$, or with $c=0,d>0$. Then we can apply the unfolding trick, and we have
\begin{align*}
&\int_{(A_{1/(N/P)}^{-1}\Gamma_{0}(N)A_{1/(N/P)})\backslash\mathfrak{H}}y^{k+1/2}y^{s}G_{0}(z;\rho,s)\overline{y^{\overline{t}}E_{k+1/2,\rho_{2\mathbf{c}}\chi_{(\varepsilon_{P}4\mathfrak{f}_{\rho_{\mathbf{r}}}/P)}}^{\chi_{P^{\vee}}}(z,\overline{t})}\tfrac{dxdy}{y^{2}}\\
=&\int_{0}^{\infty}\int_{0}^{w_{1/(N/P)}}y^{k-3/2+s+t}G_{0}(z;\rho,s)|_{A_{1/(N/P)}}dxdy,
\end{align*}
$w_{1/(N/P)}$ denoting the width of the cusp $1/(N/P)$ of $\Gamma_{0}(N)$. Then by the same argument as in the case of the cusp $1/N$, it follows that the constant term of $G_{0}(z;\rho,s)$ at the cusp $1/(N/P)$ vanishes.

Now the constant terms of Fourier expansions of $G_{0}(z;\rho,s)$ at all the cusps in $\mathcal{C}_{0}(N)$ is $0$, and hence $G_{0}(z;\rho,s)$ is rapidly decreasing at all the cusp, namely, it is a cuspidal. Then integrals
\begin{align*}
&\int_{\mathfrak{F}(N)}y^{k+1/2}y^{s}G_{0}(z;\rho,s)\overline{y^{\overline{s}}E_{k+1/2,\chi_{P^{\vee}}}^{\rho_{2\mathbf{c}}\chi_{(\varepsilon_{P}4\mathfrak{f}_{\rho_{\mathbf{r}}}/P)}}(z,\overline{s})}\tfrac{dxdy}{y^{2}}\hspace{1em}(P|\mathfrak{f}_{\rho_{\mathbf{r}}}),\\
&\int_{\mathfrak{F}(N)}y^{k+1/2}y^{s}G_{0}(z;\rho,s)\overline{y^{-k+1/2-2\overline{s}}E_{k+1/2,\rho}(z,{-}k{+}1/2{-}\overline{s})}\tfrac{dxdy}{y^{2}}
\end{align*}
make sense for $s\in\mathbf{C}$. Applying the unfolding trick to the first integral at a cusp $1/P$ for $\Re s\gg0$, and to the second integral at $1/N$ for $\Re s\ll0$, it follows that they are $0$. This implies that $\int_{\mathfrak{F}(N)}y^{k+1/2}y^{s}G_{0}(z;\rho,s)\overline{y^{\overline{s}}E(z,\overline{s})}\tfrac{dxdy}{y^{2}}=0$. For $s$ real, the integrand turns out to be $y^{k+1/2}|y^{s}G_{0}(z;\rho,s)|^{2}(\ge0)$, and since the integral is $0$, $G_{0}(z;\rho,s)$ must be $0$ for $s\in\mathbf{R}$ and hence for $s\in\mathbf{C}$.
\end{proof}

For example, we have functional equations
\begin{align}
&(1{+}({-}1)^{k}\sqrt{-1})\tfrac{1-2^{2k-2+4s}}{2^{2k-1+4s}}\tfrac{\zeta(2k-1+4s)}{\zeta(2k+4s)}y^{s}w_{0}(y,k{+}1/2,s)E_{k+1/2,\chi_{-4}^{k}}(z,{-}k{+}1/2{-}s)\nonumber\\
=&y^{s}E_{k+1/2,\chi_{-4}^{k}}(z,s)+(1{+}({-}1)^{k}\sqrt{{-}1})\tfrac{1-2^{2k-1+4s}}{2^{2k+4s}}y^{s}E_{k+1/2}^{\chi_{-4}^{k}}(z,s),  \label{eqn:exfe}\\
&(1{+}({-}1)^{k}\sqrt{-1})(1{-}2^{2k-2+4s})\tfrac{\zeta(2k-1+4s)}{\zeta(2k+4s)}y^{s}w_{0}(y,k{+}1/2,s)E_{k+1/2}^{\chi_{-4}^{k}}(z,{-}k{+}1/2{-}s)\nonumber\\
=&(1{-}({-}1)^{k}\sqrt{{-}1})2(1{-}2^{2k-1+4s})y^{s}E_{k+1/2,\chi_{-4}^{k}}(z,s)+y^{s}E_{k+1/2}^{\chi_{-4}^{k}}(z,s).\nonumber
\end{align}

As is stated in the beginning of Section \ref{sect:CTEHIW},  an Eisenstein series $E_{k+1/2,\rho',N'}(z,s)$ for $\rho',N'$ satisfying (\ref{cond:M}) with $\rho=\rho',M=N'$, but not satisfying (\ref{eqn:N2}), is written as $E_{k+1/2,\rho}(mz,s)$ for some $m\in\mathbf{N}$ and for $\rho$ satisfying (\ref{eqn:N2}). Then the functional equation of $E_{k+1/2,\rho'}(z,s)$ under $s\longmapsto-k+1/2-s$ is obtained from that of $E_{k+1/2,\rho}(z,s)$.  We note that if $\Gamma_{0}(N)\supset\Gamma_{0}(N')$, then an Eisenstein series on $\Gamma_{0}(N)$ is written as a finite sum of Eisenstein series on $\Gamma_{0}(N')$ because $A_{r}^{-1}\Gamma_{0}(N)$ is a finite union of cosets $A_{r'}^{-1}\Gamma_{0}(N')$ for suitable $A_{r'}$. For example, replacing $z$ by $2z$ in (\ref{eqn:exfe}), we have a functional equation
\begin{align*}
&(1{+}({-}1)^{k}\sqrt{-1})\tfrac{1-2^{2k-2+4s}}{2^{3k-3/2+6s}}\tfrac{\zeta(2k-1+4s)}{\zeta(2k+4s)}y^{s}w_{0}(y,k{+}1/2,s)E_{k+1/2,\chi_{(-1)^{k}8}}(z,{-}k{+}1/2{-}s)\\
=&y^{s}E_{k+1/2,\chi_{(-1)^{k}8}}(z,s)+(1{+}({-}1)^{k}\sqrt{{-}1})\tfrac{1-2^{2k-1+4s}}{2^{3k-1/2+6s}}y^{s}E_{k+1/2}^{\chi_{(-1)^{k}8}}(z,s)\\
&+(1{+}({-}1)^{k}\sqrt{{-}1})\tfrac{1-2^{2k-1+4s}}{2^{2k+4s}}y^{s}\sum_{2c,d}\chi_{-4}(c)^{k}%not \chi_{-8}(c)^{k}%
\chi_{c^{\vee}}(d)\iota_{c}^{-1}(2cz+d)^{k-1/2}|2cz+d|^{-2s}
\end{align*}
where in the summation $2c,d$ runs over the set of second rows of $A_{1/2}^{-1}\Gamma_{0}(8)$ with $2c>0$, and the sum is in $\mathcal{M}_{2k+1+s,s}(8,\chi_{(-1)^{k}8})$.

\begin{thm}Assume that $F(z)$ of Theorem \ref{thm:analytic continuation} is in $\mathcal{M}_{l,l-k-1/2}(N,\rho)$ with $k\in\mathbf{Z},\ge0$ and with $\rho\in(\mathbf{Z}/N)^{\ast}$ where $\rho$ and $k$ have the same parity and where $\rho,N$ satisfy (\ref{cond:M}) with $M=N$. Put
\begin{align*}
R_{r}(F,s):=w^{(r)}\int_{0}^{\infty}y^{l+s-2}u_{0}^{(r)}(y)dy
\end{align*}
for $u_{0}^{(r)}(y)$ in (\ref{eqn:fearc}). Let $U({-}k{+}1/2{-}s)^{-1}y^{-k+1/2-s}E_{k+1/2,\rho,N}(z,{-}k{+}1/2{-}s)=$ $\sum_{r\in\mathcal{C}_{0}(N)}$ $U_{k+1/2,\rho}^{(r)}({-}k{+}1/2{-}s))y^{s}E_{k+1/2}^{(r)}(z,\rho,s)$ be the functional equation where\linebreak $E_{k+1/2}^{(r)}(z,\rho,s)\in\mathcal{M}_{k+1/2+s,s}(N,\rho)$ is the Eisenstein series which is a sum on second rows $(c,d)$ of $A_{r}^{-1}\Gamma_{0}(N)$ with $c>0$ or with $c=0,d>0$, and  it is normalized to have the constant term (\ref{cte}) with $c_{0}=1$ at a cusp $r$. Then $R(F,s)(=R_{1/N}(F,s))$ defined in (\ref{eqn:rst}) has the meromorphic continuation to the whole complex plane, and satisfies the functional equation
\begin{align}
R(F,{-}k{+}1/2{-}s)=U({-}k{+}1/2{-}s)\sum_{r\in\mathcal{C}_{0}(N)}U_{k+1/2,\rho}^{(r)}({-}k{+}1/2{-}s)R_{r}(F,s).\label{eqn:int-r}
\end{align}

Suppose that  $\rho,N$ satisfy (\ref{eqn:N2}), and let $U_{k+1/2,\rho},U_{k+1/2,\rho,R}$ be as in Proposition \ref{prop:fn-eq}. Then there holds the functional equation
\begin{align*}
&R(F,-k+1/2-s)\\
=&U_{k+1/2,\rho}({-}k{+}1/2{-}s)\sum_{P|\mathfrak{f}_{\rho_{\mathbf{r}}}}[U_{k+1/2,\rho,P}({-}k{+}1/2{-}s)R_{1/P}(F,s)\\
&\hspace{3em}+\{1+\rho_{2}(-1)\chi_{{-}4}(P)\sqrt{{-}1}\}U_{k+1/2,\rho,4P}({-}k{+}1/2{-}s)R_{1/(4P)}(F,s)]
\end{align*}
if $\rho_{2}$ is real, and
\begin{align*}
&R(F,{-}k{+}1/2{-}s)=2^{2}\mathfrak{f}_{\rho_{2}}^{-1}U_{k+1/2,\rho}({-}k{+}1/2{-}s)\sum_{P|\mathfrak{f}_{\rho_{\mathbf{r}}}}U_{k+1/2,\rho,P}({-}k{+}1/2{-}s)R_{1/P}(F,s)
\end{align*}
if $\rho_{2}$ is complex.
\end{thm}
\begin{proof} By Corollary \ref{cor:int-eis} (ii), we have
\begin{align*}
&U_{k+1/2,\rho}({-}k{+}1/2{-}s)^{-1}R(F,{-}k{+}1/2{-}s)\\
=&\int_{\mathfrak{F}(N)}\{y^{l+s}F(z)\sum_{r\in\mathcal{C}_{0}(N)}U_{k+1/2,\rho}^{(r)}({-}k{+}1/2{-}s)E_{k+1/2}^{(r)}(z,\rho,s)-\mathbf{E}_{0}(z,s)\}\tfrac{dxdy}{y^2}
\end{align*}
for $\Re s\gg0$, $\mathbf{E}_{0}(z,s)$ being a linear combination of Eisenstein series of weight $(0,0)$.  Let $\mathbf{E}_{0}(z,s)=\sum_{r\in\mathcal{C}_{0}(N)}E_{0}^{(r)}(z,s)$ where $E_{0}^{(r)}(z,s)$ are the sums of Eisenstein series vanishing at all the cusps in $\mathcal{C}_{0}(N)$ but $r$. Then $U_{k+1/2,\rho}({-}k{+}1/2{-}s)^{-1}R(F,-k+1/2-s)$ is equal to
\begin{align*}
&\sum_{r\in\mathcal{C}_{0}(N)}\int_{\mathfrak{F}(N)}\{U_{k+1/2,\rho}^{(r)}({-}k{+}1/2{-}s)y^{l+s}F(z)E_{k+1/2}^{(r)}(z,\rho,s)-E_{0}^{(r)}(z,s)\}\tfrac{dxdy}{y^2}\\
=&\sum_{r\in\mathcal{C}_{0}(N)}\int_{(A_{r}^{-1}\Gamma_{0}(N)A_{r})\backslash\mathfrak{H}}\{U_{k+1/2,\rho}^{(r)}({-}k{+}1/2{-}s)y^{l+s}F(z)|_{A_{r}}E_{k+1/2}^{(r)}(z,\rho,s)|_{A_{r}}\\
&\hspace{25em}-E_{0}^{(r)}(z,s)|_{A_{r}}\}\tfrac{dxdy}{y^{2}},
\end{align*}
where $y^{l+s}F(z)|_{A_{r}}E_{k+1/2}^{(r)}(z,\rho,s)|_{A_{r}}$ is in $\mathcal{M}_{0,0}(A_{r}^{-1}\Gamma_{0}(N)A_{r})$. In the summation of Eisenstein series $E_{k+1/2}^{(r)}(z,\rho,s)|_{A_{r}}$, $(c,d)$ runs over the set of the second rows of matrices in $A_{1/(N/P)}^{-1}\Gamma_{0}(N)A_{1/(N/P)}$ with $c>0$, or with $c=0,d>0$. Then we can apply the unfolding trick to the integration, and we obtain (\ref{eqn:int-r}).
\end{proof}

\begin{cor}\label{cor:lseries2} Let $k\in\mathbf{Z},\ge0$ and let $l\in\frac{1}{2}\mathbf{Z},l\ge k+1/2$. Let $f,g$ be holomorphic modulars forms for $\Gamma_{0}(N)$ of weight $l$ and of weight $l-k-1/2$ with characters respectively. We assume that $f\overline{g}\in\mathcal{M}_{l,l-k-1/2}(N,\rho)$ for $\rho\in(\mathbf{Z}/N)^{\ast}$ with the same parity as $k$ and where $\rho,N$ satisfy (\ref{cond:M}) with $M=N$. 

(i) Then $L(s;f,g)$ defined in (\ref{eqn:lseries}) converges at least if $\Re s>\max\{2l{-}k{-}1/2,1/2\}$, and extends meromorphically to the whole $s$-plane. Under the notation of the theorem  there holds the functional equation
\begin{align}
&L(l{-}k{-}1/2{-}s;f,g)\nonumber\\
=&\tfrac{(4\pi)^{{-}k{+}1/2{-}2s}\Gamma(l{-}1{+}s)U({-}k{+}1/2{-}s)}{\Gamma(l{-}k{-}1/2{-}s)}\nonumber\\
&\times\sum_{r\in\mathcal{C}_{0}(N)}{w^{(r)}}^{l+s}U_{k+1/2,\rho}^{(r)}({-}k{+}1/2{-}s)L(l{-}1{+}s;f|_{A_{r}},g|_{A_{r}}), \label{eqn:fn-eq4}
\end{align}
$w^{(r)}$ being the width of a cusp $r$ in $\mathcal{C}_{0}(N)$. Let $\rho,N$ be as in (\ref{eqn:N2}). Then there holds the functional equation
\begin{align*}
&L(l{-}k{-}1/2{-}s;f,g)\\
=&\tfrac{(4\pi)^{{-}k{+}1/2{-}2s}\Gamma(l{-}1{+}s)U_{k+1/2,\rho}({-}k{+}1/2{-}s)}{\Gamma(l{-}k{-}1/2{-}s)}\sum_{p|\mathfrak{f}_{\rho_{\mathbf{r}}}}[(N/P)^{l+s}U_{k+1/2,\rho,P}({-}k{+}1/2{-}s)\\
&\times L(l{-}1{+}s;f|_{A_{1/P}},g|_{A_{1/P}})+\{1{+}\rho_{2}({-}1)\chi_{{-}4}(P)\sqrt{{-}1}\}(N/(4P))^{l+s}\\
&\times U_{k+1/2,\rho,4P}({-}k{+}1/2{-}s)L(l{-}1{+}s;f|_{A_{1/(4P)}},g|_{A_{1/(4P)}})]
\end{align*}
if $\rho_{2}=\mathbf{1}_{2},\chi_{-4}$, and
\begin{align*}
&L(l{-}k{-}1/2{-}s;f,g)\\
=&\tfrac{2^{2}\mathfrak{f}_{\rho_{2}}^{-1}(4\pi)^{{-}k{+}1/2{-}2s}\Gamma(l{-}1{+}s)U_{k+1/2,\rho}({-}k{+}1/2{-}s)}{\Gamma(l{-}k{-}1/2{-}s)}\\
&\times\sum_{P|\mathfrak{f}_{\rho_{\mathbf{r}}}}(N/P)^{l+s}U_{k+1/2,\rho,P}({-}k{+}1/2{-}s)L(l{-}1{+}s;f|_{A_{r}},g|_{A_{r}})
\end{align*}
if $\rho_{2}$ is complex.

(ii)  If $\rho,N$ are as in (\ref{eqn:N}) with $\rho_{\mathbf{r}}=\mathbf{1}$, or $k\ge1$, Let $P_{y^{l}f\overline{g}}^{(r)}(y)$ be as in (\ref{eqn:dfP}).  If $P_{y^{l}f\overline{g}}^{(\sqrt{-1}\infty)}(y)$ does not have a term containing $y$ to the power of $1$, and if $P_{y^{l}f\overline{g}}^{(r)}(y)$ does not have a term containing $y^{k+1/2}$ for any $r\in\mathcal{C}_{0}(N)$, then there holds
\begin{align}
\langle f(z),g(z)E_{k+1/2,\rho}(z,0) \rangle_{\Gamma_{0}(N)}=(4\pi)^{-l+1}\Gamma(l{-}1)L(l-1;f,g) \label{eqn:psp-hiw}
\end{align}
with the scalar product defined in (\ref{eqn:psp}). Suppose otherwise. Let $y^{s}+\xi^{(1/N)}(s)y^{-k+1/2-s}$ be the constant term of the Fourier expansion of $y^{s}E_{k+1/2,\rho}(z,s)$ at a cusp $\sqrt{-1}\infty$, and let $\xi^{(r)}(s)y^{-k+1/2-s}$ be the constant term at a cusp $r\in\mathcal{C}_{0}(N),\ne1/N$. If $C_{0}$ is a coefficient of $y$ in $P_{y^{l}f\overline{g}}^{(\sqrt{-1}\infty)}(y)$, and if $\alpha(s)$ is the coefficient of $y^{k}$ in $\sum_{r\in\mathcal{C}_{0}(N)}\overline{\xi^{(r)}(\overline{s})}w^{(r)}P_{y^{l}f\overline{g}}^{(r)}(y)$, then the equation (\ref{eqn:psp-hiw}) holds replacing  the right hand side by $\{(4\pi)^{-l+1-s}\Gamma(l{-}1{+}s)L(l{-}1{+}s;f,g)+s^{-1}(C_{0}{-}\alpha(0))-\frac{d}{ds}\alpha(0)\}|_{s=0}$.
\end{cor}
\begin{proof} (i) We take $f(z)\overline{g}(z)$ as $F(z)$ in the theorem. Then $R(F,s)=(4\pi)^{-l+1-s}\Gamma(l-1+s)L(l-1+s;f,g)$, and the assertion follows from the theorem.

(ii) The Eisenstein series $E_{k+1/2,\rho}(z,s)$ is holomorphic in $s$ at $s=0$ under our assumption. Then the proof is the same as in the proof of Corollary \ref{cor:lseries} (ii).
\end{proof}

\section{Applications --- Scalar products}\label{sect:ASP}
\begin{prop}\label{prop:vl-sp}
Let $l\in\frac{1}{2}\mathbf{Z},\ge0,\,N\in\mathbf{N}$ where $N\ge3$ if $l$ is odd, and $4|N$ if $l$ is not integral. Let $\rho\in(Z/N)^{\ast}$ be so that $\rho$  has the same parity as $l$ if $l\in\mathbf{Z}$, and  that $\rho$ has the same parity as $l-1/2$ and $\rho,N$ satisfy (\ref{cond:M}) with $M=N$ if $l\not\in\mathbf{Z}$.  Assume that $E_{l,\rho,N}(z,s)$ is holomorphic in $s$ at $s=0$. Denote by $y^{s}+\xi^{(1/N)}(s)y^{-l+1-s}$,  the constant term of Fourier expansion with respect to $x$, of the Eisenstein series $y^{s}E_{l,\rho,N}(z,s)$ at a cusp $\sqrt{-1}\infty$, and by $\xi^{(r)}(s)y^{-l+1-s}$, the constant term at a cusp $r\in\mathcal{C}_{0}(N)$.

Let $F(z)$ be in $\mathcal{M}_{l,0}(N,\rho)$ which has the Fourier expansion (\ref{eqn:fearc}) at each cusp $r$ with $u_{0}(y)=0$. Denote by $b$, the coefficient of $y^{-l+1}$ in $P_{F}^{(1/N)}(y)$, and by $c^{(r)}$, the (absolutely) constant term in $P_{F}^{(r)}(y)$ for $r\in\mathcal{C}_{0}(N)$. Then $b=\sum_{r\in\mathcal{C}_{0}(N)}c^{(r)}w^{(r)}\overline{\xi}^{(r)}(0)$, and
\begin{align}
\langle F(z),E_{l,\rho,N}(z,0)\rangle_{\Gamma_{0}(N)} = -\sum_{r\in\mathcal{C}_{0}(N)}c^{(r)}w^{(r)}\tfrac{d}{ds}\overline{\xi}^{(r)}(0).\label{eqn:sproduct-eis-eis}
\end{align}
\end{prop}
\begin{proof} Since $u_{0}(y)=0$, the equality (\ref{eqn:RFs4}) turns out to be
\begin{align}
\int_{\mathfrak{F}_{T}(N)}y^{l+s}F(z)\overline{E_{l,\rho,N}(z,\overline{s})}\tfrac{dxdy}{y^2}=-g_{T}(z,s)+h(T,s)\label{eqn:prelim}
\end{align}
where $g_{T}(z,s)$ is the sum of the integrals of the right hand side of (\ref{eqn:RFs4}) except the first one, and $h(T,s)$ is as in (\ref{eqn:defHTS}) with $l'=l$. The function $g_{T}(z,s)$ is holomorphic in $s$ at $s=0$.

For $\Re s\gg0$, we have $\int_{0}^{T} y^{l}by^{-l+1}\cdot y^{s}\tfrac{dy}{y^{2}}=bs^{-1}T^{s}$, and $\int_{T}^{\infty}y^{l}c^{(r)}\cdot\overline{\xi^{(r)}(\overline{s})}y^{-l+1-s}\tfrac{dy}{y^{2}}=c^{(r)}\overline{\xi^{(r)}(\overline{s})}s^{-1}T^{-s}$. Since near $s=0$ there holds $bs^{-1}T^{s}=\frac{b}{s}+b\log T+O(s)$ and $c^{(r)}\overline{\xi^{(r)}(\overline{s})}s^{-1}T^{-s}=\frac{c^{(r)}\overline{\xi}^{(r)}(0)}{s}+c^{(r)}\frac{d}{ds}\overline{\xi}^{(r)}(0)-c^{(r)}\overline{\xi}^{(r)}(0)\log T+O(s)$, we have
\begin{align*}
h(T,s)=&\tfrac{b-\sum_{r\in\mathcal{C}_{0}(N)}c^{(r)}w^{(r)}\overline{\xi}^{(r)}(0)}{s}+\sum_{r\in\mathcal{C}_{0}(N)}Q_{y^{l}F\overline{E_{l,\rho,N}(z,0)}}^{(r)}(T)\\
&-\sum_{r\in\mathcal{C}_{0}(N)}c^{(r)}w^{(r)}\tfrac{d}{ds}\overline{\xi}^{(r)}(0)+n(T)+O(s)
\end{align*}
near $s=0$ where $Q_{y^{l}F\overline{E_{l,\rho,N}(z,0)}}^{(r)}(T)$ is as in (\ref{eqn:defQr}) and $n(T)$ is a finite sum of terms in which the real parts of powers of $T$ are all negative. In (\ref{eqn:prelim}), the right hand side is holomorphic in $s$ at $s=0$ since it is the integral over a compact set, and $g_{T}(z,s)$ is also holomorphic in $s$. Then $h_{T}(s)$ is holomorphic, and in particular $b-\sum_{r\in\mathcal{C}_{0}(N)}c^{(r)}w^{(r)}\overline{\xi}^{(r)}(0)=0$. Then
\begin{align*}
&\int_{\mathfrak{F}_{T}(N)}y^{l+s}F(z)\overline{E_{l,\rho,N}(z,\overline{s})}\tfrac{dxdy}{y^2}-\sum_{r\in\mathcal{C}_{0}(N)}Q_{y^{l}F\overline{E_{l,\rho,N}(z,0)}}^{(r)}(T)\\
=&-g_{T}(z,s)-\sum_{r\in\mathcal{C}_{0}(N)}c^{(r)}w^{(r)}\tfrac{d}{ds}\overline{\xi}^{(r)}(0)+n(T)+o(s).
\end{align*}
At $s=0$ taking the limit as $T\longrightarrow\infty$, the left hand side tends to $\langle F(z),E_{l,\rho,N}(z,0)\rangle_{\Gamma_{0}(N)}$ by (\ref{eqn:psp}) and we obtain the equality (\ref{eqn:sproduct-eis-eis}).
\end{proof}

Let $\rho$ be a character with $N=\mathfrak{e}_{\rho}$ and with the same parity as $k$. From (\ref{eqn:fe}) and (\ref{eqn:G-E}), we have
\begin{align}
E_{k,\rho}(z,s)=&1+\delta_{\mathfrak{f}_{\rho},1}\tfrac{(-\sqrt{{-}1})^{k}\pi\varphi(\mathfrak{e}_{\rho})}{2^{k-2+2s}\mathfrak{e}_{\rho}^{k+2s}}\tfrac{\Gamma(k{-}1{+}2s)}{\Gamma(s)\Gamma(k+s)}\tfrac{\zeta(k{-}1{+}2s)}{\zeta(k{+}2s)\prod_{p|\mathfrak{e}_{\rho}}(1{-}p^{{-}k{-}2s})}y^{-k+1-2s}\nonumber\\
&+\tfrac{\tau(\overline{\widetilde{\rho}})\mu(\mathfrak{e}_{\rho}\mathfrak{f}_{\rho}^{-1})\overline{\widetilde{\rho}}(\mathfrak{e}_{\rho}\mathfrak{f}_{\rho}^{-1})}{\mathfrak{e}_{\rho}^{k+2s}L(k+2s,\overline{\rho})}\sum_{-\infty<n<\infty\atop n\ne 0}n^{-k}|n|^{1-2s}\sum_{0<d|n}\prod_{p|(d,\mathfrak{e}_{\rho}\mathfrak{f}_{\rho}^{-1})}(1-p)\nonumber \\
&\hspace{12em}\times\widetilde{\rho}(d)d^{k-1+2s}w_{n}(y,k,s)\mathbf{e}(nx),\label{eqn:fee}\\
E_{k}^{\rho}(z,s)=&\delta_{\mathfrak{e}_{\rho},1}+\tfrac{(-\sqrt{{-}1})^{k}\pi}{2^{k-2+2s}}\tfrac{\Gamma(k-1+2s)}{\Gamma(s)\Gamma(k+s)}\tfrac{L(k{-}1{+}2s,\rho)}{L(k{+}2s,\rho)}y^{-k+1-2s}+\tfrac{1}{L(k{+}2s,\rho)}\nonumber\\
&\times\sum_{-\infty<n<\infty\atop n\ne 0}n^{-k}|n|^{1-2s}\sum_{0<d|n}\rho(n/d)d^{k-1+2s}w_{n}(y,k,s)\mathbf{e}(nx), \label{eqn:fee2}
\end{align}
and $\{y^{s}E_{k,\rho}(z,s)\}|_{\left(\,0\ -1\atop 1\ \ 0\right)}=\mathfrak{e}_{\rho}^{-k-s}\{y^{s}E_{k}^{\overline{\rho}}(z,s)\}|_{z\to z/\mathfrak{e}_{\rho}}$, $\{y^{s}E_{k}^{\rho}(z,s)\}|_{\left(\,0\ -1\atop 1\ \ 0\right)}=(-1)^{k}$ $\times \mathfrak{e}_{\rho}^{s}\{y^{s}E_{k,\overline{\rho}}(z,s)\}|_{z\to z/\mathfrak{e}_{\rho}}$. 
%If $\rho$ is primiteve, then for $k\ge1$, we have
%\begin{align*}
%E_{k,\rho}(z,0)&=1+\delta_{\mathfrak{f}_{\rho},1}\delta_{k,2}\tfrac{3}{\pi^{2}}y^{-1}+\tfrac{2}{L(1-k,\rho)}\sum_{n=1}^{\infty}\sigma_{k-1,\rho}(n)\mathbf{e}(nz),\\
%E_{k}^{\rho}(z,0)&=\delta_{\mathfrak{f}_{\rho},1}+\delta_{k,1}\tfrac{-\sqrt{-1}\pi L(0,\rho)}{L(1,\rho)}+\delta_{k,2}\delta_{\mathfrak{f}_{\rho},1}\tfrac{-3}{\pi}y+
%\tfrac{(-1)^{k}2N^{k-1}\tau(\overline{\rho})}{L(1-k,\overline{\rho})}\sum_{n=1}^{\infty}\sigma_{k-1}^{\rho}(n)\mathbf{e}(nz)
%\end{align*}
%from (\ref{eqn:fee}) and (\ref{eqn:fee2}).

We compute some scalar products by using Proposition \ref{prop:vl-sp}. \\

(I) The case of integral weight and $\mathfrak{f}_{\rho}=1$.

Then $k$ is even. Using the notation of Proposition \ref{prop:vl-sp}, $y^{s}E_{k,\rho}(z,s)$ has $\xi^{(1/\mathfrak{e}_{\rho})}(s)=\tfrac{({-}1)^{k/2}\pi\varphi(\mathfrak{e}_{\rho})}{2^{k-2+2s}\mathfrak{e}_{\rho}^{k+2s}}\tfrac{\Gamma(k{-}1{+}2s)}{\Gamma(s)\Gamma(k+s)}\tfrac{\zeta(k{-}1{+}2s)}{\zeta(k{+}2s)\prod_{p|\mathfrak{e}_{\rho}}(1{-}p^{{-}k{-}2s})}$ by (\ref{eqn:fee}), and $\xi^{(0)}(s)=\tfrac{({-}1)^{k/2}\pi}{2^{k-2+2s}\mathfrak{e}_{\rho}}\tfrac{\Gamma(k{-}1{+}2s)}{\Gamma(s)\Gamma(k+s)}$ $\times \tfrac{\zeta(k{-}1{+}2s)\prod_{p|\mathfrak{e}_{\rho}}(1{-}p^{{-}k{+}1{-}2s})}{\zeta(k{+2s)\prod_{p|\mathfrak{e}_{\rho}}(1{-}p^{{-}k{-}2s})}}$ by (\ref{eqn:fee2}) and by the transformation law written below (\ref{eqn:fee2}). In the case $\mathfrak{e}_{\rho}=1$, namely,  $\rho=\mathbf{1}$, we have $\tfrac{d}{ds}\overline{\xi}(0)=-\frac{\pi}{3}\,(k=0)$,\ $3\pi^{-1}\{{-}1{+}4\log2{+}$ $2\log\pi$ ${+}24\tfrac{d}{ds}\zeta(s)|_{s=-1}\}$ $(k=2)$,\ $\frac{(-1)^{k/2}\pi\zeta(k-1)}{2^{k-2}(k-1)\zeta(k)}$ $(k\ge4)$. Then by Proposition \ref{prop:vl-sp},
\begin{align*}
&\langle E_{k}(z,0),E_{k}(z,0)\rangle_{\Gamma_{0}(1)}=-\tfrac{d}{ds}\overline{\xi}(0)\\
=&\begin{cases}\pi/3&(k=0),\\-3\pi^{-1}\{{-}1{+}4\log2{+}2\log\pi{+}24\tfrac{d}{ds}\zeta(s)|_{s=-1}\}&(k=2),\\
-\frac{(-1)^{k/2}\pi\zeta(k-1)}{2^{k-2}(k-1)\zeta(k)}&(2|k\ge4).
\end{cases}
\end{align*}
 The Eisenstein series $E_{0}(z,0)$ is a constant $1$, and the above formula shows $\pi/3=\langle E_{k}(z,0),E_{k}(z,0)\rangle_{\Gamma_{0}(1)}=\int_{\Gamma_{0}(1)\backslash\mathfrak{H}}1\tfrac{dxdy}{y^{2}}$, namely, the volume of the fundamental domain of $\mathrm{SL}_{2}(\mathbf{Z})$ is $\pi/3$, which is a well-known fact. The values of the scalar products $\langle E_{k}(z,0),E_{k}(z,0)\rangle_{\Gamma_{0}(1)}$  for  $k\ge 4$ are already obtained in Chiera \cite{Chiera}, and our result  coincides with his.

Let $\mathfrak{e}_{\rho}>1$. When $k=0$, $y^{s}E_{0,\rho}(z,s)$ is holomorphic in $s$ at $s=0$ only if $\mathfrak{e}_{\rho}$ is a prime, say $\mathfrak{e}_{\rho}=p$. Then $\langle E_{0,\rho}(z,0),E_{0,\rho}(z,0)\rangle_{\Gamma_{0}(p)}=-\tfrac{d}{ds}\overline{\xi}^{(1/p)}(0)=$ $-6^{-1}\pi(p-1)+\pi(p-1)(3\log p)^{-1}\{1-2\log2-\log\pi-12\tfrac{d}{ds}\zeta(s)|_{s=-1}\}$. An Eisenstein series $E_{0}^{\rho}(z,0)$ has $c^{(0)}=1$ and $c^{(r)}=0\ (r\in\mathcal{C}_{0}(p),\ne0)$ in the notation of Proposition \ref{prop:vl-sp}.  Then $\langle E_{0}^{\rho}(z,0),E_{0,\rho}(z,0)\rangle_{\Gamma_{0}(p)}=-\mathfrak{e}_{\rho}\tfrac{d}{ds}\xi^{(0)}(0)=6^{-1}(1+p)\pi+(3\log p)^{-1}(1-p)\pi\{1-2\log2-\log\pi-12\tfrac{d}{ds}\zeta(s)|_{s=-1}\}$. Let $\mathfrak{e}_{\rho}>1$ be square free, and let $k=2$. Then $\langle E_{2,\rho}(z,0),E_{2,\rho}(z,0)\rangle_{\Gamma_{0}(\mathfrak{e}_{\rho})}=3\mathfrak{e}_{\rho}^{-2}\varphi(\mathfrak{e}_{\rho})\pi^{-1}[\tfrac{d}{ds}\prod_{p|\mathfrak{e}_{\rho}}(1{-}p^{{-}2{-}2s})^{{-}1}|_{s=-1}+\prod_{p|\mathfrak{e}_{\rho}}(1{-}p^{{-}2})^{{-}1}\{1-4\log2-2\log(\pi\mathfrak{e}_{\rho})-24\tfrac{d}{ds}\zeta(s)|_{s=0}\}]$ and $\langle E_{2}^{\rho}(z,0),E_{2,\rho}(z,0)\rangle_{\Gamma_{0}(\mathfrak{e}_{\rho})}$ $=3\mathfrak{e}_{\rho}^{2}\pi^{-1}[\tfrac{d}{ds}\prod_{p|\mathfrak{e}_{\rho}}\tfrac{1-p^{-2-2s}}{1-p^{-1-2s}}+\prod_{p|\mathfrak{e}_{\rho}}(1{+}p^{-1})\{1-4\log2-2\log\pi+\log\mathfrak{e}_{\rho}-24\tfrac{d}{ds}\zeta(s)|_{s=-1}\}]$. For $k\ge4$ even,
\begin{align*}
\langle E_{k,\rho}(z,0),E_{k,\rho}(z,0)\rangle_{\Gamma_{0}(\mathfrak{e}_{\rho})}&=-\tfrac{({-}1)^{k/2}\pi\varphi(\mathfrak{e}_{\rho})\zeta(k{-}1)}{2^{k-2}(k-1)\prod_{p|\mathfrak{e}_{\rho}}(p^{k}{-}1)\zeta(k)},\\
\langle E_{k}^{\rho}(z,0),E_{k,\rho}(z,0)\rangle_{\Gamma_{0}(\mathfrak{e}_{\rho})}&=-\tfrac{({-}1)^{k/2}\pi\zeta(k{-}1)\prod_{p|\mathfrak{e}_{\rho}}(1{-}p^{{-}k{+}1})}{2^{k-2}(k-1)\zeta(k)\prod_{p|\mathfrak{e}_{\rho}}(1{-}p^{{-}k})}.
\end{align*}

(II) The case of integral weight and $\mathfrak{f}_{\rho}>1$.

Using the notation of Proposition \ref{prop:vl-sp}, $y^{s}E_{k,\rho}(z,s)$ has $0$ as $\xi^{(1/\mathfrak{e}_{\rho})}(s)$ by (\ref{eqn:fee}) if $\mathfrak{f}_{\rho}\ne1$. Hence if $F(z)\in\mathcal{M}_{l,0}(\mathfrak{e}_{\rho},\rho)$ satisfies  $c^{(r)}=0\ (r\in \mathcal{C}_{0}(\mathfrak{e}_{\rho}),\ne1/\mathfrak{e}_{\rho})$, then $\langle F(z),E_{k,\rho}(z,0)\rangle_{\Gamma_{0}(\mathfrak{e}_{\rho})}=0$. For example,
\begin{align*}
\langle E_{k,\rho}(z,0),E_{k,\rho}(z,0)\rangle_{\Gamma_{0}(\mathfrak{e}_{\rho})}=0\hspace{1.5em}(k\ne1).
\end{align*}
If  $k=1$, then $E_{1,\rho}(z,0)$ has nonzero $c^{(r)}$ other than $c^{(1/\mathfrak{e}_{\rho})}$. Suppose that $\rho$ is primitive. Then $c^{(1/\mathfrak{f}_{\rho})}=1$, $c^{(0)}=\tfrac{-\sqrt{-1}\pi L(0,\overline{\rho})}{\mathfrak{f}_{\rho}L(1,\overline{\rho})}$, $c^{(r)}=0\ (r\in\mathcal{C}_{0}(\mathfrak{f}_{\rho}),\ne 1/\mathfrak{f}_{\rho},0)$, and $y^{s}E_{1,\rho}(z,s)$ has $\xi^{(1/\mathfrak{f}_{\rho})}(s)=0$ and $\xi^{(0)}(s)=\tfrac{(-\sqrt{{-}1})\pi}{2^{-1+2s}\mathfrak{f}_{\rho}}\tfrac{\Gamma(2s)}{\Gamma(s)\Gamma(1+s)}\tfrac{L(2s,\overline{\rho})}{L(1+2s,\overline{\rho})}$ in the notation of Proposition \ref{prop:vl-sp}. Then
\begin{align}
&\langle E_{1,\rho}(z,0),E_{1,\rho}(z,0)\rangle_{\Gamma_{0}(\mathfrak{f}_{\rho})}=-c^{(0)}\mathfrak{f}_{\rho}\tfrac{d}{ds}\overline{\xi}^{(0)}(0)\nonumber\\
=&\{2\log2+2\tfrac{d}{ds}\log\tfrac{L(1+s,\rho)}{L(s,\rho)}|_{s=0}\}\ \ 
(k=1,\,\rho\mbox{\, is primitive}). \label{eqn:psp-w1}
\end{align}
For primitive $\rho$, the equality $E_{1,\rho}(z,0)=\tfrac{\tau(\overline{\rho})L(1,\rho)}{\mathfrak{f}_{\rho}\,L(1,\overline{\rho})}E_{1}^{\rho}(z,0)$ holds, and the scalar product $\langle E_{1}^{\rho}(z,0),E_{1,\rho}(z,0)\rangle_{\Gamma_{0}(\mathfrak{f}_{\rho})}$ can be obtained from (\ref{eqn:psp-w1}).

We compute $\langle E_{k}^{\rho}(z,0),E_{k,\rho}(z,0)\rangle_{\Gamma_{0}(\mathfrak{e}_{\rho})}$ for $k\ge2$ and for $\rho$ not necessarily primitive. We take $E_{k}^{\rho}(z,0)$ as $F(z)$ in Proposition \ref{prop:vl-sp}. Then $c^{(0)}=(-1)^{k}$ and $c^{(r)}=0\ (r\in C_{0}(\mathfrak{e}_{\rho}),\ne0)$, and $\langle E_{k}^{\rho}(z,0),E_{k,\rho}(z,0)\rangle_{\Gamma_{0}(\mathfrak{e}_{\rho})}={-}({-}1)^{k}\mathfrak{e}_{\rho}\tfrac{d}{ds}\overline{\xi}^{(0)}(0)$ where $y^{s}E_{k,\rho}(z,s)$ has $\tfrac{(-\sqrt{-1})^{k}\pi\Gamma(k-1+2s)L(k-1+2s,\overline{\rho})}{2^{k-2+2s}\mathfrak{e}_{\rho}\Gamma(s)\Gamma(k+s)L(k+2s,\overline{\rho})}$ as $\xi^{(0)}(s)$.  Then
\begin{align*}
\langle E_{k}^{\rho}(z,0),E_{k,\rho}(z,0)\rangle_{\Gamma_{0}(\mathfrak{e}_{\rho})}=-\tfrac{(-\sqrt{-1})^{k}\pi L(k-1,\rho)}{2^{k-2}(k-1)L(k,\rho)}\hspace{1em}(k\ge2).
\end{align*}
If $k=0$ and $\rho$ is primitive, then the same computation leads to $\langle E_{0}^{\rho}(z,0),E_{0,\rho}(z,0)\rangle_{\Gamma_{0}(\mathfrak{f}_{\rho})}$ $=4\pi\tau(\rho)^{-1}L(-1,\rho)L(1,\overline{\rho})^{-1}\{-1-\log(\pi/\mathfrak{f}_{\rho})+\tfrac{d}{ds}\log\Gamma(s)|_{s=1/2}+L(-1,\rho)^{-1}$ \linebreak$\times\tfrac{d}{ds} L(s,\rho)|_{s=-1}+L(1,\overline{\rho})^{-1}\tfrac{d}{ds}L(s,\overline{\rho})|_{s=1}\}$.\\

(III) The case of half integral weight.

By (\ref{eqn:feehiw1}),(\ref{eqn:feehiw2}) and by Lemma \ref{lem:vac0}, $y^{s}E_{k+1/2,\chi_{-4}^{k}}(z,s)\in\mathcal{M}_{k+1/2,0}(4,\chi_{-4}^{k})$ has $\xi^{(1/4)}(s)=\tfrac{(-1)^{k(k+1)/2}}{2^{k-1+2s}(2^{2k+4s}-1)}\tfrac{\pi\Gamma(k-1/2+2s)}{\Gamma(s)\Gamma(k+1/2+s)}\tfrac{\zeta(2k-1+4s)}{\zeta(2k+4s)}$, $\xi^{(1/2)}(s)=0$, $\xi^{(0)}(s)=$\linebreak$\tfrac{({-}\sqrt{{-}1})^{k}(1{-}\sqrt{{-}1})(2^{2k-1+4s}-1)}{2^{k+2s}(2^{2k+4s}-1)}\tfrac{\pi\Gamma(k-1/2+2s)}{\Gamma(s)\Gamma(k+1/2+s)}\tfrac{\zeta(2k-1+4s)}{\zeta(2k+4s)}$. The Fourier expansion of\linebreak $E_{k+1/2,\chi_{-4}^{k}}(z,0)$ is obtained from (\ref{eqn:feehiw1s0}),(\ref{eqn:feehiw2s0}). Then by Proposition \ref{prop:vl-sp},
\begin{align}
&\langle E_{k+1/2,\chi_{-4}^{k}}(z,0),E_{k+1/2,\chi_{-4}^{k}}(z,0)\rangle_{\Gamma_{0}(4)}=-\tfrac{d}{ds}\overline{\xi}^{(1/4)}(0)\nonumber\\
=&\begin{cases}-\tfrac{\pi}{3\log 2}\{-2+5\log 2+2\log\pi+24\tfrac{d}{ds}\zeta(s)|_{s=-1}\}&(k=0),\\
\frac{2}{3\pi}\{3-20\log2-6\log\pi-72\zeta(s)|_{s=-1}\}&(k=1),\\
-\tfrac{(-1)^{k(k+1)/2}2^{-k+2}\pi\zeta(2k-1)}{(2^{2k}-1)(2k-1)\zeta(2k)}&(k\ge2).
\end{cases}\label{eqn:vl-psp-e}
\end{align}
Since $E_{k+1/2,\chi_{-4}^{k}}(z,0)|_{\left({0\ {-}1/2\atop2\ \ \ 0\ }\right)}=2^{-k-1/2}E_{k+1/2}^{\chi_{-4}^{k}}(z,0)$, $2^{-2k-1}\langle E_{k+1/2}^{\chi_{-4}^{k}}(z,0),$\linebreak$E_{k+1/2}^{\chi_{-4}^{k},4}(z,0)\rangle_{\Gamma_{0}(N)}$ is equal to (\ref{eqn:vl-psp-e}) where there is the additional term $2\pi/3\ (k=0)$, $(4\log2)/\pi\ (k=1)$ by Lemma \ref{lem:psp-sup}. Again by Proposition \ref{prop:vl-sp}, 
\begin{align*}
&\langle E_{k+1/2}^{\chi_{-4}^{k}}(z,0),E_{k+1/2,\chi_{-4}^{k}}(z,0)\rangle_{\Gamma_{0}(4)}=-(-1)^{k-1}\sqrt{-1}\cdot4\cdot\tfrac{d}{ds}\overline{\xi}^{(0)}(0)\\
=&\begin{cases}
\tfrac{(1-\sqrt{-1})\pi}{3\log 2}\{-2+7\log 2+2\log \pi+24\tfrac{d}{ds}\zeta(s)|_{s=-1}\}&(k=0)\\
\tfrac{4(1+\sqrt{-1})}{3\pi}\{3-8\log2 -6\log \pi-72\tfrac{d}{ds}\zeta(s)|_{s=-1}\}&(k=1)\\
-\tfrac{(-\sqrt{-1})^{k}(1-\sqrt{-1})2^{-k+3}(2^{2k-1}-1)\pi\zeta(2k-1)}{(2^{2k}-1)(2k-1)\zeta(2k)}&(k\ge2)
\end{cases}
\end{align*}
and
\begin{align*}
\langle \mathscr{E}_{k+1/2,\chi_{-4}^{k}}(z,0),\mathscr{E}_{k+1/2,\chi_{-4}^{k}}(z,0)\rangle_{\Gamma_{0}(4)}=\begin{cases}
2\pi&(k=0),\\
-12(\log2)/\pi&(k=1),\\
\tfrac{(-1)^{k(k+1)/2}(2^{2k-2}-1)\pi\zeta(2k-1)}{2^{k-2}(2^{2k-1}-1)^{2}(2k-1)\zeta(2k)}&(k\ge2).
\end{cases}
\end{align*}
In particular $\langle \theta,\theta\rangle_{\Gamma_{0}(4)}=2\pi,\,\langle \theta^{3},\theta^{3}\rangle_{\Gamma_{0}(4)}=-(12\log2)/\pi,\,\langle \theta^{5},\theta^{5}\rangle_{\Gamma_{0}(4)}=-\tfrac{2\cdot3^{2}5\zeta(3)}{7^{2}\pi^{3}}$, $\langle \theta^{7},\theta^{7}\rangle_{\Gamma_{0}(4)}=\tfrac{3^{4}5\cdot7\zeta(5)}{2\cdot31^{2}\pi^{5}}$. From $\langle \theta,\theta\rangle_{\Gamma_{0}(4)}=2\pi$, we see that $L(s;\theta,\theta)$ has the residue $2$ at $s=1/2$ by (\ref{eqn:psp-formula2}), however this is obvious because $L(s;\theta,\theta)=4\zeta(2s-1)$.
%$\langle\theta(z)^{4},\theta(z)^{4}\rangle_{\Gamma_{0}(4)}=2^{-4}\pi^{-1}\{3-86\log2-6\log\pi-72\tfrac{d}{ds}\zeta(s)|_{s=-1}\}$
%$\langle\theta(z)^{8},\theta(z)^{8}\rangle_{\Gamma_{0}(4)}=\tfrac{2153\zeta(3)}{2^{11}\pi^{3}}$

\section{Applications --- $L$-functions}\label{sect:ALF}
\begin{prop}\label{prop:mr-pole} Let $f,g$ be holomorphic modulars forms for $\Gamma_{0}(N)$ of weight $l,l'\in\tfrac{1}{2}\mathbf{N}\ (l\ge l',\,l+l'>1)$ respectively with $f\overline{g}\in\mathcal{M}_{l,l'}(N,\rho)$ for $\rho\in(\mathbf{Z}/N)^{\ast}$. If $l-l'$ is odd, then $N\ge3$, and if $l-l'\not\in\mathbf{Z}$, then $4|N$. If $l-l'$ is integral, then $\rho$ has the same parity as $l-l'$, and if otherwise, then $\rho$ has the same parity as $l-l'-1/2$, and $\rho,N$ satisfy (\ref{cond:M}) with $M=N$. Let $a_{0}^{(r)},b_{0}^{(r)}$ be the $0$-th Fourier coefficients of $f,g$ respectively at a cusp $r\in\mathcal{C}_{0}(N)$. Denote by $y^{s}+\xi^{(1/N)}(s)y^{-l+l'+1-s}$,  the constant term of Fourier expansion with respect to $x$, of the Eisenstein series $y^{s}E_{l-l',\rho,N}(z,s)$ at a cusp $\sqrt{-1}\infty$, and by $\xi^{(r)}(s)y^{-l+l'+1-s}$, the constant term at a cusp $r\in\mathcal{C}_{0}(N)$. Let $m\in\mathbf{Z}$ be the order of a pole of $E_{l-l',\rho,N}(z,s)$ at $s=l'$. Let $\xi^{(r)}(s)=c_{-m}^{(r)}(s-l')^{-m}+O((s-l')^{-m+1})$. Then 
\begin{align}
\lim_{s\to0}s^{m+1}L(l{+}l'{-}1{+}s;f,g)=(4\pi)^{l+l'-1}\Gamma(l{+}l'{-}1)^{-1}\sum_{r\in\mathcal{C}_{0}(N)}w^{(r)}a_{0}^{(r)}\overline{b}_{0}^{(r)}\overline{c}_{-m}^{(r)}.\label{eqn:al-sng}
\end{align}
In particular if $E_{l-l',\rho,N}(z,s)$ is holomorphic at $s=l'$, then
\begin{align}
\mathrm{Res}_{s=l+l'-1}L(s;f,g)=(4\pi)^{l+l'-1}\Gamma(l{+}l'{-}1)^{-1}\sum_{r\in\mathcal{C}_{0}(N)}w^{(r)}a_{0}^{(r)}\overline{b}_{0}^{(r)}\overline{\xi}^{(r)}(l').\label{eqn:l-res}
\end{align}
If $E_{l-l',\rho,N}(z,s)$ is holomorphic on a domain $\Re s\ge l'$, then $L(s;f,g)$ has the only possible pole on the domain at $s=l'$.
\end{prop}
\begin{proof} By (\ref{eqn:RFs4}), we have
\begin{align}
&(4\pi)^{-l+1-s}\Gamma(l-1+s)(s-l')^{m}L(l-1+s;f,g)=(s-l')^{m}R(f\overline{g},s)\nonumber\\
=&\int_{\mathfrak{F}_{T}(N)}y^{l+s}f(z)\overline{g(z)}(s{-}l')^{m}\overline{E_{0,\mathbf{1}_{N},N}(z,\overline{s})}\tfrac{dxdy}{y^2}{+}g_{T}(z,s){-}(s{-}l')^{m}h(T,s) \label{eqn:lpsp}
\end{align}
where $g_{T}(z,s)$ is $(s-l')^{m}$ times the sum of all integrals but the first one in (\ref{eqn:RFs4}). Then $g_{T}(z,s)$ is holomorphic in $s$ at $s=l'$, and
\begin{align*}
h(T,s)=\tfrac{a_{0}^{(1/N)}\overline{b}_{0}^{(1/N)}}{l-1+s}T^{l-1+s}+\sum_{r\in\mathcal{C}_{0}(N)}w^{(r)}a_{0}^{(r)}\overline{b}_{0}^{(r)}\tfrac{\overline{\xi^{(r)}(\overline{s})}}{l'-s}T^{l'-s}.
\end{align*}
Then integral of (\ref{eqn:lpsp}) is holomorphic  in $s$ at $s=l'$, and $h(T,s)=-\sum_{r}w^{(r)}a_{0}^{(r)}\overline{b}_{0}^{(r)}\overline{c}_{-m}^{(r)}$ $\times(s-l')^{-m-1}+O((s-l')^{-m})$ around $s=l'$. Then replacing $l'+s$ by $s$, we obtain (\ref{eqn:al-sng}).
\end{proof}
The $L$-series $L(s;f,g)$ for $f,g$ in Proposition \ref{prop:mr-pole} converges for $s>l+l'-1$. Hence if $L(s;f,g)$ has a pole at $s=l+l'-1$, then it is the rightmost pole.

\begin{thm} Let $f,g$ be as in Proposition \ref{prop:mr-pole}. Let $f(z)=\sum_{n=0}^{\infty}a_{n}\mathbf{e}(nz),g(z)=\sum_{n=0}^{\infty}b_{n}\mathbf{e}(nz)$ be the Fourier expansions. Assume that (i) there is a nonzero constant c for which $c\,a_{n}\overline{b}_{n}\ (n\ge1)$ are all real and non-negative, (ii) $E_{l-l',\rho,N}(z,s)$ is holomorphic at $s=l'$, and (iii) the right hand side of (\ref{eqn:l-res}), which we denote by $C$, is not zero. Then  $\sum_{0<n\le X}a_{n}\overline{b}_{n}\sim C\frac{X^{l+l'-1}}{l+l'-1}$ as $X\longrightarrow\infty$.
\end{thm}
\begin{proof} We just apply the Wiener-Ikehara theorem to the $L$-function $L(s;f,g)$ at the rightmost pole $s=l+l'-1$.  
\end{proof}

\begin{cor}\label{cor:om}  (i) Let $k\ge0$ be even and let $l\in\tfrac{1}{2}\mathbf{N}$ with $l>\max\{k,1/2\}$. Let $f,g$ be holomorphic modular forms for $\Gamma_{0}(N)$ of wight $l,l-k$ respectively and with same character. Let $a_{0}^{(r)},b_{0}^{(r)}$ denote the $0$-th Fourier coefficients at a cusp $r\in\mathcal{C}_{0}(N)$, of $f,g$ respectively. Then for $l=1$ and $k=0$, an equality $\lim_{s\to0}s^{2}L(1{+}s;f,g)=\tfrac{12}{\prod_{p|N}(1+p^{-1})}\sum_{i/M\in\mathcal{C}_{0}(N)}\tfrac{a_{0}^{(i/M)}\overline{b}_{0}^{(i/M)}}{(M^{2},N)}$ holds. If $l\ge3/2$, then $L(s;f,g)$ has a possible pole of order $1$ at $s=2l-k-1$ with the residue
\begin{align}
\mathrm{Res}_{s=2l-k-1}L(s;f,g)=&\tfrac{(-1)^{k/2}2^{2l-k}\pi^{2l-k}\varphi(N)\zeta(2l-k-1)}{\Gamma(l-k)\Gamma(l)\zeta(2l-k)\prod_{p|N}(1-p^{-2l+k})}\nonumber\\
&\times\sum_{i/M\in\mathcal{C}_{0}(N)}\tfrac{a_{0}^{(i/M)}\overline{b}_{0}^{(i/M)}\prod_{p|(N/M)}(1-p^{-2l+k+1})}{(M^{2},N)\varphi(N/M)M^{2l-k-1}}. \label{eqn:mr-pole}
\end{align}

(ii) Let $f(z)=\sum_{n=0}^{\infty}a_{n}\mathbf{e}(nz)$ be a holomorphic modular form for $\Gamma_{0}(N)$ of weight $l\in\frac{1}{2}\mathbf{N},\ge 3/2$ with any character. Let $a_{0}^{(r)}$ be the $0$-th Fourier coefficients at cusps $r\in\mathcal{C}_{0}(N)$ where $a_{0}^{(1/N)}=a_{0}$. We assume that at least one of $a_{0}^{(r)}$ is not zero. Then
\begin{align*}
\sum_{0<n\le X}|a_{n}|^{2}\sim\tfrac{2^{2l}\pi^{2l}\varphi(N)\zeta(2l-1)}{\Gamma(l)^{2}\zeta(2l)\prod_{p|N}(1-p^{-2l})}\sum_{i/M\in\mathcal{C}_{0}(N)}\tfrac{|a_{0}^{(i/M)}|^{2}\prod_{p|(N/M)}(1-p^{-2l+1})}{(M^{2},N)\varphi(N/M)M^{2l-1}}\tfrac{X^{2l-1}}{2l-1}.
\end{align*}

(iii) Let $Q,Q'$ be positive definite integral quadratic forms of $2l,2l-2k$ variables respectively with $l\in\frac{1}{2}\mathbf{N},\ge 3/2,\,2|k\ge0$. Let $N$ be the maximum of the levels of $Q,Q'$. Let $a_{0}^{(r)},b_{0}^{(r)}$ be the $0$-th Fourier coefficients at $r\in\mathcal{C}_{0}(N)$, of theta series associated with $Q,Q'$ respectively. Then if the discriminants $d_{Q},d_{Q'}$ are equal to each other up to square factors, then $\sum_{0<n\le X}r_{Q}(n)r_{Q'}(n)\sim C X^{2l-k-1}/(2l-k-1)$ where $C$ denotes the right hand side of (\ref{eqn:mr-pole}).
\end{cor}
\begin{proof} (i) The assertion follows form (\ref{eqn:l-res}), (\ref{eqn:al-sng}) where $w^{(i/M)}$ is given in (\ref{eqn:w-cusp}), and $\xi^{(i/M)}(s)$ is given in (\ref{eqn:ctk}). 

As for (ii) and (iii), the assertions follow from Theorem, since $|a_{n}|^{2}\ge0$ in the case (ii), and since $r_{Q}(n)r_{Q'}(n)\ge0$ in the case (iii). 
\end{proof}

For $k>0$, let $r_{k}(n)$ denote the number of representations of $n$ as a sum of $k$ squares. Then $\theta(z)^{k}=1+\sum_{n=1}^{\infty}r_{k}(n)\mathbf{e}(nz)$, which is a modular form for $\Gamma_{0}(4)$ with the value $1$ at $\sqrt{-1}\infty$, the value $2^{-k/2}\mathbf{e}(-k/8)$ at $0$, the value $0$ at $1/2$.  We have $L(s;\theta^{k}\overline{\theta}{}^{k'})=\sum_{n=1}^{\infty}r_{k}(n)r_{k'}(n)n^{-s}$. Then by applying Corollary \ref{cor:om} (ii) to $L(s;\theta^{k}\overline{\theta}{}^{k})$, we obtain
\begin{align*}
\sum_{0<n\le X}r_{k}(n)^{2}\sim\tfrac{\pi^{k}\zeta(k-1)}{\Gamma(k/2)^{2}\zeta(k)(1-2^{-k})}&\tfrac{X^{k-1}}{k-1}
\end{align*}
for $k\ge3$, which is known as Wagon's conjecture proved by R.~Crandall and S.~Wagon (for the detail see Borwein and Choi \cite{Borwein-Choi}, Choi, Kumchev and Osburn \cite{Choi-Kumchev-Osburn}). More generally  we obtain from Corollary \ref{cor:om} (iii),
\begin{align*}
\sum_{0<n\le X}r_{k}(n)r_{k-4m}(n)\sim\tfrac{\{1-(1{-}({-}1)^{m})2^{-k+2m+1}\}\pi^{k-2m}\zeta(k-2m-1)}{\Gamma(k/2-2m)\Gamma(k/2)\zeta(k-2m)(1-2^{-k+2m})}\tfrac{X^{k-2m-1}}{k-2m-1}
\end{align*}
for $k\in\mathbf{N},m\in\mathbf{Z},\ge0,\,k>\max\{4m,2\}$. The estimate of this kind is obtained for many other quadratic forms.

For modular forms on $\Gamma_{0}(4)$, we have the following;
\begin{cor}\label{cor:om2} (i) Let $f(z)=\sum_{n=0}^{\infty}a_{n}\mathbf{e}(nz),g(z)=\sum_{n=0}^{\infty}b_{n}\mathbf{e}(nz)$ be holomorphic modular forms for $\Gamma_{0}(4)$ of weight $l\in\tfrac{1}{2}\mathbf{N}$ and $l-k>0$ respectively with odd $k\ge1$. Let $a_{0}^{(r)},b_{0}^{(r)}$ denote the $0$-th Fourier coefficients at a cusp $r\in\mathcal{C}_{0}(4)$, of $f,g$ respectively.  If $a_{0}^{(0)}b_{0}^{(0)}\ne0$, then $L(s;f,g)$ has a simple pole at $s=2l-k-1$ with residue
\begin{align}
\mathrm{Res}_{s=2l{-}k{-}1}L(s;f,g)=\tfrac{\sqrt{{-}1}^{k}2^{2l-k}\pi^{2l-k}a_{0}^{(0)}\overline{b}_{0}^{(0)}}{\Gamma(l{-}k)\Gamma(l)}\tfrac{L(2l{-}k{-}1,\chi_{-4})}{L(2l{-}k,\chi_{-4})}. \label{eqn:mr-pole2}
\end{align}
Further we suppose that there is a nonzero constant $c$ so that $c\,a_{n}\overline{b}_{n}\ (n\ge1)$ are all non-negative. Then $\sum_{0<n\le x}a_{n}\overline{b}_{n}\sim CX^{2l-k-1}/(2l-k-1)$ where $C$ is the left hand side of (\ref{eqn:mr-pole2}).

(ii) Let $k\in\mathbf{Z},\ge0$, and let $l\in\tfrac{1}{2}\mathbf{N}$ with $l>k+1/2$. Let $f,g$ be holomorphic modular forms for $\Gamma_{0}(4)$ with automorphy factors  $j(*,z)^{2l},j(*,z)^{2l{-}2k{-}1}$ respectively where $j$ denotes the automorphy factor of $\theta$ as in Introduction. Then if $l\ge3/2$, then $L(s;f,g)$ has a possible pole of order $1$ at $s=2l-k-3/2$ with the residue
\begin{align}
\mathrm{Res}_{s=2l-k-3/2}L(s;f,g)=&\tfrac{(-1)^{k(k+1)/2}2^{2l-k-3}\pi^{2l{-}k{-}1/2}\zeta(4l{-}2k{-}3)}{(2^{4l-2k-2}{-}1)\Gamma(l{-}k{-}1/2)\Gamma(l)\zeta(4l{-}2k{-}2)}\{a_{0}^{(1/4)}\overline{b}_{0}^{(1/4)}\nonumber\\
&\hspace{3em}+(1{+}({-}1)^{k}\sqrt{{-}1})2^{3}(2^{4l{-}2k{-}3}{-}1)a_{0}^{(0)}\overline{b}_{0}^{(0)}\}.\label{eqn:mr-pole3}
\end{align}
If  $L(s;f,g)$ has the pole and if there is a nonzero constant $c$ so that $c\,a_{n}\overline{b}_{n}\ (n\ge1)$ are all non-negative, then $\sum_{0<n\le x}a_{n}\overline{b}_{n}\sim CX^{2l-k-3/2}/(2l-k-3/2)$ where $C$ is the left hand side of (\ref{eqn:mr-pole3}).
\end{cor}
\begin{proof} (i) Any modular form for $\Gamma_{0}(4)$ of odd weight has character $\chi_{-4}$. We apply Theorem to $f,g$ and  $y^{s}E_{k,\chi_{-4}}(z,s)$ where $y^{s}E_{k,\chi_{-4}}(z,s)$ has $\xi^{(1/4)}(s)=\xi^{(1/2)}(s)=0,\,\xi^{(0)}(s)=\tfrac{(-\sqrt{{-}1})^{k}\pi}{2^{k+2s}}\tfrac{\Gamma(k-1+2s)}{\Gamma(s)\Gamma(k+s)}\tfrac{L(k{-}1{+}2s,\chi_{-4})}{L(k{+}2s,\chi_{-4})}$. Then the assertion follows from Proposition \ref{prop:mr-pole} and Theorem.

(ii) We apply Proposition \ref{prop:mr-pole} and Theorem to $f,g$ and $E_{k+1/2,\chi_{-4}^{k}}(z,s)$ where $\xi^{(i/M)}(s)$ is written down in Section \ref{sect:ASP}, (III).
\end{proof}
By Corollary \ref{cor:om2} we have
\begin{align*}
\sum_{0<n\le X}r_{k}(n)r_{k-2m}(n)\sim\tfrac{\pi^{k-m}L(k{-}m{-}1,\chi_{-4})}{\Gamma(k/2{-}m)\Gamma(k/2)L(k{-}m,\chi_{-4})}\tfrac{X^{k-m-1}}{k{-}m{-}1}
\end{align*}
for $m\equiv1\pmod{2}, k>2m$, and 
\begin{align*}
\sum_{0<n\le X}r_{k}(n)r_{k-m}(n)\sim&\tfrac{\{1{+}\chi_{8}(m)2^{-k+(m-3)/2}{-}2^{-2k+m+2}\}\pi^{k{-}m/2}\zeta(2k{-}m{-}2)}{\Gamma((k{-}m)/2)\Gamma(k/2)\zeta(2k{-}m{-}1)(1{-}2^{-2k+m+1})}\tfrac{X^{k-m/2-1}}{k{-}m/2{-}1}
\end{align*}
for $m\equiv1\pmod{2} $ and $k\ge 3/2,\,k>m$.

The $L$-function  $L(s;\theta^{k}\overline{\theta}{}^{k})=\sum_{n=1}^{\infty}r_{k}(n)^{2}n^{-s}$ has a pole also at $s=k/2$. Their residues up to $8$ are $2\,(k=1),\ 2^{2}\{\log2 +\tfrac{d}{ds}\log\tfrac{L(1+s,\chi_{-4})}{L(s,\chi_{-4})}|_{s=0}\} \,(k=2)$,\linebreak$ -2^{5}3\pi^{-1}\log2\,(k=3),\ 2^{-1}\{3-86\log2-6\log\pi-72\tfrac{d}{ds}\zeta(s)|_{s=-1}\}\,(k=4),\ -2^{7}3\cdot7^{-2}\pi^{-2}\zeta(3)\,(k=5),\ 2\pi^{-2}L(2,\chi_{-4})\,(k=6),\ 2^{8}3^{3}7\cdot31^{-2}\pi^{-3}\zeta(5)\,(k=7)$ and $2^{-5}3^{-1}\cdot 2153\zeta(3)\,(k=8)$.

\end{document}